\documentclass[10pt]{amsart}
\usepackage{amsmath,amssymb,amsthm,amscd,mathrsfs}
\usepackage[all]{xy}
\usepackage{color} 
\CompileMatrices
\numberwithin{equation}{section}
\newtheorem{prop}{Proposition}[section]
\newtheorem{theo}[prop]{Theorem}
\newtheorem{lemm}[prop]{Lemma}
\newtheorem{coro}[prop]{Corollary}
\newtheorem{rema}[prop]{Remark}

\numberwithin{equation}{section}

\newcommand{\be}{\begin{equation}}
\newcommand{\ee}{\end{equation}}

\newcommand{\IP}{\mathbb{P}}
\newcommand{\IG}{{\mathbb{G}}}

\newcommand\IZ{\mathbb {Z}}

\newcommand\IQ{\mathbb {Q}}
\newcommand{\IC}{\mathbb{C}}

\newcommand{\IR}{\mathbb{R}}
\newcommand{\CU}{{\mathcal U}}
\newcommand{\ba}{\begin{array}}
\newcommand{\ea}{\end{array}}

\newcommand{\CV}{{\mathcal V}}
\newcommand{\CX}{{\mathcal X}}
\newcommand{\IF}{{\mathbb F}}
\newcommand{\IA}{{\mathbb A}}

\newcommand{\CS}{{\mathcal S}}

\def\Ext{\operatorname{Ext}}

\newcommand{\bal}{\begin{aligned}}
\newcommand{\eal}{\end{aligned}}

\newcommand{\CZ}{{\mathcal Z}}

\newcommand{\rk}{{\rm{rk}}}

\newcommand{\longto}{\longrightarrow}

\newcommand{\ch}{{\mathrm{ch}}}

\newcommand{\CO}{{\mathcal O}}

\newcommand{\CA}{{\mathcal A}}
\newcommand{\CH}{{\mathcal H}}

\newcommand{\CM}{{\mathcal M}}
\newcommand{\CW}{{\mathcal W}}
\newcommand{\CY}{{\mathcal Y}}
\newcommand{\CL}{{\mathcal L}}
\newcommand{\CJ}{{J}}
\newcommand{\CR}{{\mathcal R}}
\newcommand{\CN}{{\mathcal N}}
\newcommand{\CC}{{\mathcal C}}
\newcommand{\CI}{{\mathcal I}}

\newcommand{\RHom}{{\mathrm{RHom}}}
\newcommand{\CQ}{{\mathcal Q}}
\newcommand{\cO}{{\mathcal O}}
\newcommand{\cA}{{\mathcal A}}

\newcommand{\CT}{{\mathcal T}}

\newcommand{\calP}{{\mathcal P}}

\newcommand{\obj}{{\mathfrak {Ob}}}

\newcommand{\IL}{{\mathbb L}}
\title{HOMFLY polynomials, stable pairs and motivic Donaldson-Thomas invariants}
\author[Duiliu-Emanuel Diaconescu, Zheng Hua, 
Yan Soibelman]{Duiliu-Emanuel Diaconescu${}^1$, 
Zheng Hua${}^2$, Yan Soibelman${}^3$}
\address{${}^1$NHETC, Rutgers University, Piscataway, NJ 08854-0849 USA}
\email{duiliu@physics.rutgers.edu}
\address{${}^2$Department of Mathematics, Kansas State
University, Manhattan, KS 66506-2602 USA}
\email{zheng.hua.zju@gmail.com}
\address{${}^3$Department of Mathematics, Kansas State
University, Manhattan, KS 66506-2602 USA}
\email{soibel@math.ksu.edu}

\date{}
\begin{document}
\begin{abstract}
Hilbert scheme topological invariants of plane curve singularities 
are identified to framed 
threefold stable pair invariants.
As a result, the conjecture of Oblomkov and Shende on HOMFLY polynomials of links of plane curve singularities is given a 
Calabi-Yau threefold interpretation. The motivic 
Donaldson-Thomas theory developed by M. Kontsevich and the 
third author then yields natural motivic invariants
for algebraic knots. 
This construction is motivated by previous work 
of V. Shende, C. Vafa and the first author on the large $N$ duality 
derivation of the above conjecture.
\end{abstract} 
\maketitle 

\tableofcontents

\section{Introduction}\label{sectionone}

The starting point of this work is a conjecture of 
 Oblomkov and Shende \cite{hilbert-links} relating the 
 HOMFLY polynomial of the link of a plane curve singularity 
 to  topological invariants of its Hilbert scheme of points. 
 It was then explained in \cite{largeNknots} that this 
 conjecture has a natural physical interpretation in terms 
 of large $N$ duality for conifold transitions. 
 The conifold transition is a topology changing process from a 
 smooth hypersurface 
 \[
 xz-yw=\mu, \qquad \mu\neq 0,
 \]
 in $\IC^4$ to a small resolution the conifold singularity
 \[
 xz-yw=0,
 \]
which is isomorphic to the total space $Y$ of the rank two bundle 
$\CO_{\IP^1}(-1)^{\oplus 2}$ on $\IP^1$.
In this context, the construction of \cite{largeNknots} 
assigns to an algebraic knot $K$ in $S^3$ a lagrangian 
cycle $M_K$ in $Y$ which intersects a singular plane curve 
$C^\circ$ 
contained in a fiber of $Y\to\IP^1$ along a circle. 
Moreover,  $C^\circ$ has a unique singular point at the 
intersection with the zero section, its link being 
isotopic to $K$.
Then large $N$ duality leads to a 
 conjectural relation 
between HOMFLY polynomials of algebraic knots and 
Gromov-Witten theory on $Y$ with lagrangian boundary conditions 
on $M_K$. This conjecture has been tested in \cite{largeNknots} 
by explicit {\bf A}-model computations for torus knots. 

The relation between large $N$ duality and the conjecture of 
Oblomkov and Shende follows from the observation 
that Gromov-Witten theory is conjecturally equivalent to 
Donaldson-Thomas theory \cite{MNOP-I}, and also 
stable pair theory \cite{stabpairs-I}. For
Gromov-Witten theory  counting stable maps 
with compact domain without boundary, 
these relations have been proven for toric 
threefolds in \cite{GWDTtoric}, \cite{curvesKthree}. 
String duality arguments 
\cite{Ooguri:1999bv,Labastida:2000yw,framedknots} predict
that Gromov-Witten theory with lagrangian boundary 
conditions should be similarly related to certain certain D6-D2-D0 
counting invariants. The latter have not been given a 
rigorous mathematical construction since a definition of  
Donaldson-Thomas or stable pair theory 
with lagrangian boundary conditions is not known so far. 
In certain special cases, such as lagrangian cycles associated to 
the unknot, one can employ relative 
Donaldson-Thomas or stable pair theory in order to fill this gap. 
Then the correspondence reduces 
via \cite{KL} to certain identities for cubic Hodge integrals 
on the moduli space of curves which have been proven 
in \cite{unknotHodge, proofMV,two-part}. This approach is not 
however expected to work for more general lagrangian cycles, in particular 
for the lagrangian cycles for algebraic knots constructed in \cite{largeNknots}. Therefore one is left with the question whether there is a Donaldson-Thomas/stable pair construction for counting invariants 
corresponding via large $N$ duality to general algebraic knots. 

The main claim of the present paper is that there is a natural construction 
of such invariants in terms of stable pairs subject to a framing condition
explained below. Given a singular plane 
curve $C^\circ$ in a fiber of the projection $Y\to\IP^1$, there is a 
natural moduli space of $C^\circ$-framed stable pairs on $Y$. 
These are pairs $\CO_Y{\buildrel s\over \longto} F$ on $Y$ 
where $F$ is topologically supported on the union of $C^\circ$ 
with the zero section $C_0\subset Y$, and has multiplicity one along $C^\circ$. Then the main result  is that such 
moduli spaces are related to the nested Hilbert schemes 
employed in \cite{hilbert-links,refinedknots} by a variation 
of stability condition. For technical reasons, this is proven 
embedding of the affine curve $C^\circ$ in a suitable compact 
Calabi-Yau threefold $X$. In particular the embedding will 
factor through the natural projective completion $C\subset \IP^2$ 
of $C^\circ$.
Using previous results on stability conditions for perverse coherent sheaves 
\cite{limit,generating}, the nested Hilbert schemes of 
\cite{hilbert-links,refinedknots} are then geometrically related to 
moduli spaces of framed stable objects in a certain stability 
chamber. 

Enumerative invariants for $C$-framed stable pairs are defined 
by integration of a certain constructible function $\nu$ on the 
moduli space of $C^\circ$-framed stable pairs. 
 Since the Hilbert scheme invariants 
used in \cite{hilbert-links} are topological, one can simply 
take $\nu=1$ obtaining the topological Euler numbers of 
the moduli spaces. Then a wallcrossing 
formula shows that the resulting 
invariants are then in agreement with those of \cite{hilbert-links}.
Alternative constructions may be carried out, using either Behrend 
constructible functions \cite{micro} as in \cite{genDTI} or 
motivic weight functions as in \cite{wallcrossing}. 
Motivated by previous connections 
between motivic and refined Donaldson-Thomas invariants 
\cite{DG,motivic-Hilbert,motivic-conifold,motivic-crepant}, the second approach will be considered in this paper. 
Assuming the foundational aspects of \cite{wallcrossing}, it will 
be shown 
that the virtual motivic invariants of $C$-framed 
objects are in agreement with the refined conjecture formulated 
in \cite{refinedknots} if certain technical conditions are met. 
Removing the technical conditions in question reduces to a 
comparison conjecture 
between motivic weights of stable pairs and sheaves (see Section 4.2)
which is at the moment open.

Appearance of motivic Donaldson-Thomas invariants supports an old idea of S. Gukov and third author that there should exist a {\it motivic knot invariants theory}. In such theory skein relations should correspond to wall-crossing formulas for the motivic Donaldson-Thomas invariants introduced in \cite{wallcrossing} (and further developed in \cite{COHA}). Knot invariants themselves should be derived from an appropriate $3$-dimensional Calabi-Yau category. 

The idea can be traced back to \cite{Gukov:2004hz}, where Khovanov-Rozansky theory was linked to the count of BPS states in topological string theory. It was further developed  in \cite{Dunfield:2005si} in the form of a conjecture about knot superpolynomial. After the work \cite{wallcrossing} of Kontsevich and third author it became clear that motivic Donaldson-Thomas invariants (DT-invariants for short) introduced in the loc. cit. provide the right mathematical foundation for the notion of (refined) BPS state. This was pointed out in  \cite{DG} based on physics arguments,  rigorous mathematical statements confirming this claim being 
first formulated and proved in \cite{motivic-Hilbert}. 
Further results along these lines have been obtained in \cite{motivic-conifold,motivic-crepant}. The parameter $y$ which appears in knot invariants should correspond to the motive 
${\mathbb L}=[{\mathbb A}^1]$ of affine line in the theory of motivic DT-invariants. Then the question is: what is an appropriate $3$-dimensional Calabi-Yau category?
 From the point of view of the large $N$ duality  it is natural to expect that the $3$-dimensional Calabi-Yau category should be somehow derived from the resolved conifold $Y$. Unfortunately it is difficult to make this idea mathematically precise since $Y$ is non-compact (as well as the lagrangian cycle $M_K$). One can see that the partition function for the unknot derived in \cite{Ooguri:1999bv} coincides with the motivic DT-series for the $3$-dimensional Calabi-Yau category generated by one spherical object (both are given essentially by the quantum dilogarithm). But there was no general conjecture about the desired relationship.
Although such a conjecture does not exist at present, the works \cite{hilbert-links} and \cite{refinedknots} give a hope that it can be formulated soon. Our paper can be considered as another step in this direction.

A more detailed overview including technical details is presented at length below. 

\subsection{The conjectures of Oblomkov, Rasmussen and 
Shende}\label{OSconjecture} 
 
Let $C^\circ\subset \IC^2$ be a reduced 
pure dimension one curve with one singular point $p\in C^\circ$. Let $H^n_p(C^\circ)$ be the punctual Hilbert scheme parameterizing length 
$n$ zero dimensional subschemes 
of $C^\circ$ with topological support at $p$. 
Let $m:H_p^n(C^\circ)\to \IZ$ 
be the constructible function assigning to any subscheme $Z\subset C^\circ$ with topological support at $p$ 
the minimal number of generators of the defining 
ideal $\CI_{Z,p}\subset \CO_{C^\circ,p}$ at $p$. For any scheme 
${X}$ 
of finite type over $\IC$, and any constructible function 
$\nu :{X}\to \IZ$ let 
\[
\int_{X} \nu d\chi = \sum_{n\in \IZ} n\chi(\nu^{-1}(n))
\]
where $\chi$ denotes the topological Euler character. 
Then let 
\be\label{eq:OSconjB} 
Z_{C^\circ,p}(q,a) 
=\sum_{n\geq 0} q^{2n} \int_{H_p^n(C^\circ)} (1-a^2)^m d\chi.
\ee
Let $K_{C^\circ,p}$ 
denote the link of 
the plane curve singularity at $p$. Let $P_{K_{C^\circ,p}}(a,q)$ denote the HOMFLY polynomial 
of $K_{C^\circ,p}$. It satisfies the skein relation of the type:

$$aP_{L_+}-a^{-1}P_{L_-}=(q-q^{-1})P_{L_0}.$$

As opposed to \cite{hilbert-links}, the HOMFLY polynomial 
will be normalized such that it takes value 
\[
{a-a^{-1}\over q-q^{-1}}
\]
for the unknot. 
Then the conjecture of Oblomkov and Shende \cite{hilbert-links} states that 
\be\label{eq:OSconjA}
	P_{K_{C^\circ,p}}(q,a) = (a/q)^{\mu-1} Z_{C^\circ,p}(q,a),
\ee
 where $\mu$ is the Milnor number of the singularity at $p$.

\subsubsection{Refinement}\label{refinedOS}
 
The correspondence between knot polynomial invariants and 
 Hilbert scheme invariants 
 of  curve singularities admits a refined generalization \cite{refinedknots} due to Oblomkov, Rasmussen and Shende. 
 Given an algebraic knot or link $K$, let $P_K^{ref}(q,a,y)$ denote the refined HOMFLY 
 polynomial introduced in \cite{Gukov:2004hz,Dunfield:2005si}. This is the polynomial 
 invariant called {\it reduced superpolynomial} in \cite{Dunfield:2005si}, which specializes to the 
 HOMFLY polynomial at $y=-1$. 
 In the previous notation consider the incidence cycle 
 \[
 H_p^{[l,r]}(C^\circ)\subset H_p^l(C^\circ)\times 
 H^{l+r}_p(C^\circ)
 \]
 parameterizing 
 pairs of ideals $(J,I)$ in the local structure ring 
 $\CO_{C^\circ,p}$ 
 satisfying the following condition
  \[
     m_p J \subseteq I \subseteq J, 
     \]
 where $m_p\subset \CO_{C^\circ,p}$ is the maximal ideal of the singular point. Let $H_p^{[l,r]}(C^\circ)$ be equipped with the reduced induced 
subscheme structure and 
\be\label{eq:refHilbinv}
Z^{ref}_{C^\circ,p}(q,a,y) = \sum_{l,r\geq 0} q^{2l} a^{2r}y^{r^2} 
 P_y(H_p^{[l,r]}(C^\circ)),
 \ee
 where $ P_y$ denotes the virtual Poincar\'e polynomial (also known as Serre polynomial).
 Then Oblomkov, 
 Rassmusen and Shende \cite{refinedknots}
 conjecture the following relation 
  \be\label{eq:refknotA}
 P_K^{ref}(q,a,y) =  \left({a\over q}\right)^{\mu-1} Z^{ref}_{C^\circ,p}(q,a,y).
   \ee


\subsection{Framed stable pair invariants of the conifold}\label{oneone}

The resolved conifold $Y$ is a small crepant resolution of  the nodal 
hypersurface $xz-yw=0$ in $\IC^4$. 
It can be easily identified with the total space of the rank two bundle 
$\mathrm{Tot}(\CO_{\IP^1}(-1)\oplus \CO_{\IP^1}(-1))$ such that 
the exceptional cycle $C_0\simeq \IP^1$ of the resolution is the zero section.

There is closed embedding $C^\circ\hookrightarrow Y$ 
which factors through the natural embedding of $C$ in a 
fiber of the projection $Y\to \IP^1$. 
Therefore the curve $C^\circ$ in the conjecture of Oblomkov and 
Shende is naturally identified with a vertical complete intersection 
on $Y$. 
Recall \cite{stabpairs-I} that a stable pair on $Y$ is determined 
by the data $(F,s)$, where $F$ is a pure dimension one 
coherent torsion sheaf on $Y$, and $s:\CO_Y\to F$ is a 
morphism with zero dimensional cokernel. Note that $F$ will not be assumed to have proper support.  Let $\CI_{C^\circ}$ be the defining ideal 
sheaf of $C^\circ\subset Y$. 
A $C^\circ$-framed stable 
pair on $Y$ is a stable pair $(F,s)$ such that 
\begin{itemize}
\item $F$ is topologically supported on the union $C^\circ\cup C_0$ 
\item The annihilator ideal $Ann(F)$ of $F$ is a subsheaf 
of the defining ideal $\CI_{C^\circ}$ of $C^\circ$ and the quotient 
$\CI_{C^\circ}/Ann(F)$ is topologically supported on the zero 
section $C_0$. 
\end{itemize} 
Note that the second condition is equivalent to the requirement that 
the scheme theoretic support $Z_F$ of $F$ have at most two irreducible 
components, $C^\circ$ and an additional component supported on $C_0$, 
which may be empty. The numerical invariants of a $C^\circ$-framed stable 
pair on $Y$ will be the generic multiplicity $r$ of $F$ 
along the zero section, and $l=\chi(\mathrm{Coker}(s))$. 

Let ${\overline Y} = 
\IP(\CO_Y(-1)^{\oplus 2} \oplus \CO_Y)$ be a projective 
completion of $Y$, and ${C}\subset {\overline Y}$ 
the resulting projective completion of $C^\circ$. Projective plane curve ${C}$ is  
contained in a fiber of the projection 
${\overline Y}\to \IP^1$. 
According to \cite{stabpairs-I}, there exists a fine 
projective moduli space 
$\calP({\overline Y}, r, n)$ of stable pairs $(G,v)$ on 
${\overline Y}$, where $\ch_2(G) = [ C]+r[C_0]$, 
and $\chi(G)=n$. Then it can be easily proved that there exists a
fine quasi-projective moduli space $\calP(Y,C^\circ,r,l)$ of 
$C^\circ$-framed 
stable pairs on $Y$ with $l=n-\chi(\CO_{ C})$. 
Moreover, $\calP(Y,C^\circ,r,l)$ is the locally closed subscheme of 
$\calP({\overline Y}, r, n)$ determined by the conditions 
\begin{itemize}
\item $Ann(G)\subset \CI_{ C}$ 
\item The support of $\mathrm{Coker}(v)$ is contained in the 
open part $Y\subset {\overline Y}$.
\end{itemize}
Let $\calP^\circ({\overline Y},r,n)$ denote the open subspace 
of $\calP({\overline Y},r,n)$ parameterizing 
pairs satisfying only the second condition above.

Counting invariants $P^\nu(Y,C^\circ, r,n)$
are defined by integrating a constructible function $\nu$ 
on the ambient moduli space $\calP^\circ({\overline Y},r,n)$
over the subspace $\calP(Y, C^{\circ},r,n)$.
Several choices are in principle available for such a constructible 
function: the constant function $\nu=1$, Behrend's constructible 
function \cite{micro}, or the motivic weight function conjectured in 
\cite{wallcrossing}. Two cases will be considered in this paper, namely 
$\nu=1$, or the motivic weight function of \cite{wallcrossing}. 

In the first case, the resulting invariants are simply topological Euler 
numbers of moduli spaces, 
\[
P^{top}(Y,C^\circ, r,n) = \chi(\calP(Y, C^{\circ},r,n)).
\]
In the second case, assuming the foundational problems solved, 
the construction of  \cite{wallcrossing} produces a motivic weight 
function $\nu^{mot}$ together with a finite 
stratification of $\{\CS_\alpha\}$ 
of the moduli space $\calP^{\circ}({\overline Y},r,n)$ such that 
$\nu^{mot}$ takes a constant value $\nu^{mot}_\alpha$ 
on each locally closed stratum $\CS_\alpha$. 
The motives $\nu^{mot}_\alpha$ belong to a certain ring of motives 
presented in detail in \cite[Sect. 4.3 and 6.2]{wallcrossing} which 
contains the Grothendieck ring $K_0(Var/\IC)$ of complex algebraic 
varieties as a subring, as well as a formal square root $\IL^{1/2}$ 
and formal inverses $\IL^{-1}, [GL(k,\IC)]^{-1}$, $k\in \IZ_{\geq 1}$. 
Then the motivic 
Donaldson-Thomas invariants of $C^{\circ}$-framed stable pairs are
defined by 
\[ 
P^{mot}(Y, C^{\circ}, r,n) = \sum_\alpha [\CS_\alpha]\nu_\alpha^{mot} 
\]
where $[\CS_\alpha]\in K_0(Var/\IC)$ is the Chow motive of the 
stratum $\CS_\alpha$. 
In both cases, let 
\[ 
Z^\nu(Y,C^\circ,u,T) = \sum_{n\in \IZ_{\geq 0}} \sum_{r\geq 0} 
P^\nu(Y,C^\circ, r,n) u^nT^r.
\]
be the resulting generating function. 
Let also $P^\nu(Y,r,n)$ denote the corresponding 
counting invariants for stable 
pairs $(F,s)$ on $Y$ with $\ch_2(Y)=r[C_0]$ and $n=\chi(F)$, 
and 
\[ 
Z^\nu(Y,u,T) := \sum_{n\in \IZ} \sum_{r\geq 0} 
P^\nu(Y,r,n) u^nT^r.
\]
their generating function.
Then large N duality \cite{largeNknots} predicts a conjectural 
factorization formula 
\be\label{eq:factformulaA}
Z^{\nu}(Y,C^\circ, u,T)= Z^{\nu}(Y,u,T) Z^{\nu}(C^\circ,u,T)
\ee
where $Z^{\nu}(C^\circ,u,T)$ is a formal power series in $(T,u)$, 
possibly up to multiplication by an overall Laurent  
monomial in $(T,u)$. This is a reflection 
 of the  natural factorization of Wilson loop expectation values 
 in large $N$ Chern-Simons theory, 
\[
\langle W_K(U)\rangle_{CS, N\to \infty} = P_K(q,a) Z_{CS}(q,a),
\]
where $Z_{CS}(q,a)$ is the large $N$ limit of the $U(N)$ Chern-Simons 
partition function on $S^3$. 
Furthermore $Z^{\nu}(C^\circ,u,T)$ is conjectured to have an intrinsic 
interpretation in terms of D6-D2-D0-bound state counting on $Y$ and 
is conjecturally related to the generating function 
$Z_{C^\circ,p}(q,a)$ in equation \eqref{eq:OSconjA}
or its refined counterpart \eqref{eq:refHilbinv} as explained in detail below. 
Let 
\be\label{eq:globalHilbinv}
\bal 
Z^{top}_{C^\circ}(q,a) & = \sum_{n\geq 0} q^{2n} \int_{H^n(C^\circ)}
(1-a^2)^m d\chi,\\
 \eal 
\ee
be the global version of equation \eqref{eq:OSconjA}, where 
the punctual Hilbert scheme $H^n_p(C^\circ)$ 
is replaced by the Hilbert scheme $H^n(C^\circ)$ of length 
$n$ zero dimensional subschemes of $C^\circ$ 
with no support 
condition. Similarly, consider the following 
global motivic version of \eqref{eq:refHilbinv}
\be\label{eq:globalHilbmotivic}
Z_{C^\circ}^{mot}(a,q)  = \sum_{l,r\geq 0} q^{2l} a^{2r} \IL^{r^2/2} 
[ H^n(C^\circ)]
\ee
where $[ H^n(C^\circ)]\in K_0(Var/\IC)$ denotes the 
Chow motive of the Hilbert scheme. Taking the virtual Poincar\'e 
polynomial with compact support, one obtains the global refined 
generating function
\[ 
Z_{C^\circ}^{ref}(a,q,y)  = \sum_{l,r\geq 0} q^{2l} a^{2r} y^{r^2} 
P_y(H^n(C^\circ))
\]
Note that a simple stratification argument shows that 
\[
Z^{top}_{C^\circ}(a,q) = (1-q^2)^{1-\chi(C^\circ)} Z_{C^\circ,p}(a,q),
\]
respectively
\[
Z_{C^\circ}^{ref}(a,q,y)= \bigg(\sum_{n\geq 0} q^{2n} P_y(S^n(
C^\circ\setminus \{p\}))\bigg) Z_{C^{\circ},p}^{ref}(a,q,y),
\] 
where $S^n(C^\circ\setminus\{p\})$ are the symmetric 
powers of the punctured curve $C^\circ\setminus\{p\}$. 
The compactly supported 
cohomology $H^k_c(C^\circ\setminus\{p\})$, $k\geq 0$
is endowed with Deligne's weight filtration. Let $h_c^{k,w}((C^\circ\setminus\{p\})$ be the dimension of the successive quotient 
of weight $w$. Then, 
using the results of \cite{cheah}, the above formula can be rewritten in closed 
form as follows:
\[
Z_{C^\circ}^{ref}(a,q,y)= 
\prod_{k,w\geq 0} 
\left({1\over 1-(-1)^w y^kq^2}\right)^{(-1)^w h_c^{k,w}(C^\circ \setminus\{p\})}Z_{C^{\circ},p}^{ref}(a,q,y).
\]

Then large $N$ duality leads to the conjecture that there is a 
monomial change of variables $T=T(a,q,\IL^{1/2})$, $u=u(a,q,\IL^{1/2})$
such that the following identity holds 
\be\label{eq:largeNconjecturesA} 
\bal 
Z^{mot}(Y,C^\circ,T,u) = a^\alpha q^\beta \IL^\gamma 
Z^{mot}_{C^\circ}(q,a)
\eal
\ee
for some $\alpha, \beta\in \IZ$, $\gamma \in {1\over 2}\IZ$.
Taking virtual Poincar\'e polynomials with compact support 
yields a similar identity for refined invariants 
\[
Z^{ref}(Y,C^\circ,T,u,y) = a^\alpha q^\beta y^{2\gamma} 
Z^{ref}_{C^\circ}(q,a,y),
\]
subject again to a
monomial change of variables $T=T(a,q,y)$, $u=u(a,q,y)$.
Specializing the refined identity 
to $y=1$ yields a similar conjectural relation 
\be\label{eq:largeNconjecturesB} 
Z^{top}(Y,C^\circ,T,u) = a^\alpha q^\beta 
Z^{top}_{C^\circ}(q,a).
\ee
for topological invariants invariants. 

The main result of this paper, Theorem \ref{topCframedinv} below, 
proves an identity 
of the form \eqref{eq:largeNconjecturesB} for framed stable pair 
invariants on a smooth projective Calabi-Yau threefold $X$. 
Compactness is needed here for technical reasons, as the proof 
relies heavily on the wallcrossing formalism of 
\cite{wallcrossing,genDTI} applied to abelian categories 
of perverse coherent sheaves as in  
\cite{limit, generating}. 
As explained in Section \ref{onethree}, the threefold $X$ is a 
smooth crepant resolution of a nodal threefold $X_0$ 
and contains a 
projective completion $C\subset \IP^2$ of $C^\circ$, 
assumed  to be smooth away from $p$. 
Moreover, a compact version of the motivic 
identity \eqref{eq:largeNconjecturesA} can be in principle 
derived along the same lines from the formalism of 
\cite{wallcrossing}, assuming the required foundational
results as well as certain technical results on motivic weights. The 
main steps are summarized in Section \ref{motiviccorresp} and 
explained in detail in Section \ref{motivicsect}.

\subsection{Embedding in a compact Calabi-Yau threefold}\label{onetwo}

Theory of  stable pairs of Pandharipande and Thomas deals with compact varieties. Since the resolved conifold $Y$ is non-compact we need to formulate the problem in an appropriate compactification. We start with some generalities.
Let $X_0$ be a projective Calabi-Yau 
threefold with a single conifold singularity $q\in X_0$. 
Since all ordinary double points are analytically equivalent, the 
formal neighborhood of $q\in X_0$ is isomorphic to the 
formal neighborhood of the origin in the singular hypersurface 
$xz-yw=0$ in $\IC^4$. Suppose moreover, there exists a 
 Weil divisor $\Delta\simeq \IP^2\subset X_0$ containing $q$ 
 which is locally determined by $z=0$. Blowing up $X_0$ along 
 the divisor 
 $\Delta$ yields a crepant resolution $X\to X_0$, the exceptional 
 locus being a $(-1,-1)$ curve $C_0\subset X$. Let $D$ be the 
 strict transform of $\Delta$ in $X$. A local computation shows that 
$D\simeq \Delta$ 
intersects $C_0$ transversely at a point $p$. 

Although the considerations below are not particular 
to a specific model, an example will be provided next for 
concreteness.
Let $X^-$ be a smooth elliptic fibration with a section over 
the Hirzebruch surface $\IF_1$. Let $D^-\subset X^-$ denote the 
image of the canonical section, and $C_0^-\subset D^-$ the 
unique $(-1)$ curve on $D^-$. As shown in \cite{Ftheory} 
using toric methods, there exists a morphism 
$X^-\to  X_0$ contracting the curve $C^-_0$, 
where $X_0$ is a nodal Calabi-Yau threefold. 
Moreover there is a second smooth crepant resolution of 
$X\to X_0$ equipped with a projection to $\IP^2$, and 
a section $D\simeq \IP^2$. 
 The exceptional locus is in this case a rational 
$(-1,-1)$ curve intersecting  $D$  transversely at a point $p$. 
More examples with two or four conifold singularities where 
$D$ is a toric surface have 
been studied in the context of large $N$ duality in \cite{largeNcompact}.

In this context, let $\Gamma\subset X_0$ be a reduced irreducible 
plane curve contained in the 
Weil divisor $\Delta\simeq \IP^2$ passing through the conifold point $q$. 
Suppose $\Gamma$ has a singularity at $q$ and is otherwise 
smooth. Let $C\subset X$ be the strict transform of 
$\Gamma$ 
in $X$. Note that $C$ is a plane curve
 in 
$D\simeq \IP^2\subset X$, and the restriction of the 
contraction $X\to X_0$ to $C$ is an isomorphism
$C{\buildrel \sim \over \longto}\Gamma$. 
Moreover $C$ intersects the
 exceptional 
 curve $C_0\subset X$ at the point $p$, which is 
 the only singular point of $C$ under the current assumptions.

By analogy with Section \ref{oneone} a stable pair 
$(F,s)$ on $X$ will be called $C$-framed of type 
$(r,n)\in \IZ_{\geq 0}\times \IZ$ \begin{itemize}
\item $F$ is topologically supported on the union $C\cup C_0$
\item $\ch_2(F)=[C]+r[C_0]$, $\chi(F)=n$. 
\end{itemize}
Then there is a closed subscheme ${\mathcal P}(X,C,r,n)
\subset {\mathcal P}(X,\beta,n)$, with $\beta= [C]+r[C_0]$ 
parameterizing $C$-framed stable pairs. 

Enumerative invariants are defined as explained above 
equation \eqref{eq:factformulaA}
by integration with respect to an appropriate constructible 
function $\nu$ on the ambient space 
$ {\mathcal P}(X,\beta,n)$. For $\nu=1$, the resulting 
invariants are topological Euler numbers of the moduli 
spaces ${\mathcal P}(X,C,r,n)$ and they will be denoted 
by $P^{top}(X,C,r,n)$. Taking $\nu$ to be 
the motivic weight function \cite[Sect. 6.2]{wallcrossing}
on the ambient space ${\mathcal P}(X,\beta,n)$, one 
obtains motivic $C$-framed stable pair invariants $P^{mot}(X,C,r,n)$. 
Their generating functions are 
\[ 
Z^\nu(X,C,T,u) = \sum_{n\in \IZ}\sum_{r\geq 0} T^ru^n 
P^\nu(X,C,r,n).
\]
One similarly defines constructible function invariants 
$P^\nu(X,C_0,r,n)$ 
for stable pairs $(F,s)$, where $F$ is topologically supported on 
$C_0$, and has numerical invariants $\ch_2(F)=n[C_0]$, 
$\chi(F)=n$. Their generating function will be denoted by 
$Z^\nu(X,C_0,T,u)$. In order to make a connection with the large 
$N$ duality conjectures in Section \ref{oneone}, note that 
\[
Z^{\nu}(Y,T,u) = Z^\nu(X,C_0,T,u)
\]
since the formal neighborhood of $C_0$ in $X$ is isomorphic to 
the formal neighborhood of the zero section in $Y$. 

As anticipated in Section \ref{oneone}, the generating functions 
\eqref{eq:globalHilbinv}, 
\eqref{eq:globalHilbmotivic} admit natural compact versions
\be\label{eq:compactHilbinv}
Z_{C}^{top}(q,a) = \sum_{n\geq 0} q^{2n} \int_{H^n(C)} 
(1-a^2)^m d\chi
\ee
respectively
\be\label{eq:compactHilbmotivic}
Z_C^{mot}(q,a) = \sum_{n\geq 0} q^{2n} a^{2r} \IL^{r^2/2}
[H^{[n,r]}(C)]. 
\ee
The notation is analogous to Section \ref{OSconjecture}, except that the 
punctured curve $C^\circ$ is replaced with the compact curve $C$. 
Again, a stratification argument shows that 
\[
Z_{C}^{top}(q,a) = (1-q^2)^{1-\chi(C)} Z_{C,p}(q,a),
\]
respectively
\[
Z_{C}^{ref}(q,a) = 
\prod_{k,w\geq 0} 
\left({1\over 1-(-1)^wy^kq^2}\right)^{(-1)^wh_c^{k,w}(C\setminus\{p\})}
 Z_{C,p}^{ref}(q,a,y),
\]
by analogy with the similar formulas in Section \ref{oneone}. 
The integers $h_c^{k,w}(C\setminus \{p\})$ are the weighted Betti 
numbers of $C\setminus \{p\}$ of compactly supported cohomology 
equipped with Deligne's weight filtration. 

Then one of the main results of this paper is the following theorem 
for topological Euler character invariants.
\begin{theo}\label{topCframedinv} 
\be\label{eq:factformulaB}
Z^{top}(X,C,q^2,-a^2)= Z^{top}(X,C_0,q^2,-a^2) 
q^{2\chi(\CO_C)}Z_{C}^{top}(q,a).
\ee
\end{theo} 

Theorem \ref{topCframedinv} follows from Proposition \ref{factorization}
and Theorem \ref{smallBidentity} below, which rely heavily on  
wallcrossing for framed stable pair invariants.
The general framework is outlined in the 
next subsection, and presented in more detail in 
Section \ref{derivcatsect}. 

The motivic version 
of identity \eqref{eq:factformulaB} will 
be discussed in Section \ref{motiviccorresp}, once the main 
steps in the proof of equation \eqref{eq:factformulaB} 
are clearly understood. 

\subsection{$C$-framed perverse coherent sheaves and 
stability}\label{onethree}

Let $D^b(X)$ be the bounded derived category of 
$X$. Let $\CA\subset D^b(X)$ be the heart of the
perverse t-structure on $D^b(X)$ determined by the torsion 
pair $({{Coh}}_{\geq 2}(X),{{Coh}}_{\leq 1}(X))$. The objects of $\CA$ are objects $E$ of $D^b(Y)$ 
such that the cohomology sheaves $\CH^i(E)$ are nontrivial 
only for $i=-1,0$, $\CH^{-1}(E)$ has no torsion in codimension 
$\geq 2$, and $\CH^{0}(E)$ is torsion, 
of dimension $\leq 1$. Let $\omega$ be a fixed K\"ahler class 
on $X$. 

The stable pair theory of $X$ has been 
studied in \cite{limit,generating} employing a construction 
of limit (or weak) stability conditions on $\CA$ which we  review in Section 
\ref{twoone}. The main motivation for the study of limit stability conditions in the loc.cit.  was 
to prove the rationality conjecture of Pandharipande and Thomas 
\cite{stabpairs-I}. 
The main tool in the proof is the wallcrossing 
formalism of  \cite{wallcrossing, genDTI} applied to a one-parameter family of stability conditions on $\CA$ 
parameterized by a $B$-field, $B=b\omega\in H^2(X)$. In fact, as was pointed out in \cite{wallcrossing}, the wall-crossing formulas for the weak stability conditions is a special case of those considered in the loc.cit. as soon as one allows the central charge to take values in an ordered field. Weak stability conditions are  easy to construct \cite{limit, generating}
for the derived category of coherent sheaves $D^b(X)$ on a Calabi-Yau manifold $X$, differently from conventional Bridgeland stability conditions. 
More specifically, there is a slope function $\mu_{(\omega,b)}$ 
on the Grothendieck group $K_0(\CA)$ which defines a 
family of weak stability conditions on $\CA$,
as reviewed in Section \ref{twoone}. Moreover, 
the following  results are proven in \cite{generating}.

{\bf 1}. For fixed $(\beta,n)$ there is an algebraic 
moduli stack of finite type 
${\mathcal M}_{b}^{ss}(\CA,\beta,n)$
of $\mu_{(\omega,b)}$-semistable objects 
of $\CA$ with $\ch(E)=(-1,0,\beta,n)$. 

{\bf 2}. For fixed $(\beta,n)$ there are finitely many critical parameters 
$b_c$ such that strictly $\mu_{(\omega,b)}$-semistable objects exist. 
The moduli stacks $\CM^{ss}_{{(\omega,b')}}(\CA,\beta,n)$, \\
$\CM^{ss}_{{(\omega,b'')}}(\CA,\beta,n)$ are canonically 
isomorphic if there is no critical stability parameter in the interval 
$[b',b'']$. Moreover, if $b$ is not critical
all closed points of $\CM^{ss}_{{(\omega,b)}}(\CA,\beta,n)$
are $\mu_{{(\omega,b)}}$-stable and their stabilizers are canonically isomorphic to $\IC^\times$. 

{\bf 3}. For fixed 
$\omega$, $(\beta,n)$, there exists $b_{-\infty}$ such 
that for any $b<b_{-\infty}$ 
the moduli stack $\CM^{ss}_{{(\omega,b)}}(\CA,\beta,n)$ is
an $\IC^\times$-gerbe over 
the moduli space ${P}(X,\beta,n)$
of stable pairs on $X$.

 A similar construction will be employed in the proof of Theorem \ref{topCframedinv}. 
A full subcategory $\CA^C$ of $\CA$ consisting of 
$C$-framed perverse coherent sheaves $\CA^C$ 
is defined by conditions $(C.1)$, $(C.2)$ 
in Section \ref{twotwo}. Then it is shown that the slope construction 
of weak stability conditions \cite{generating} and basic properties of slope limit 
semistable objects carry over to the $C$-framed category. 
In particular, one can construct a one parameter family of weak stability conditions parameterized by the $B$-field $B=b\omega\in H^2(X)$. 

The moduli stacks of 
$\mu_{(\omega,b)}$-semistable objects $E$ 
in $\CA^C$ with numerical invariants 
$\ch(E)=(-1,0,[C]+r[C_0],n)$, $r\in \IZ_{\geq 0}$, 
$n\in \IZ$, will be denoted by ${\mathcal P}_{(\omega,b)}(X,C,r,n)$.
Their properties are completely analogous 
({\bf 1}) -- ({\bf 3})
above. 
In particular they are algebraic stacks of finite type, 
and for fixed $\omega$ and numerical invariants $(r,n)$ strictly 
semistable objects exist only for finitely many critical values of $b$. 
Moreover, there 
exists $b_{-\infty}\in \IR_{<0}$ such that for $b<b_{-\infty}$
 ${\mathcal P}_{(\omega,b)}(X,C,r,n)$ 
is a $\IC^\times$-gerbe over the moduli space of $C$-framed 
stable pairs. 

Let ${Ob}(\CA)$ be the stack of all objects of $\CA$, 
which is an algebraic stack locally finite type over $\IC$. 
For all $b\in \IR$ and all 
$(r,n)\in \IZ_{\geq 0}\times \IZ$ the natural forgetful
morphism
\[
{\mathcal P}_{(\omega,b)}(X,C,r,n) \hookrightarrow {Ob}(\CA)
\]
determine a stack function 
in the motivic Hall algebra $H(\CA)$. 

Counting invariants $P^\nu_{(\omega,b)}(X,C,r,n)$ are again  
defined  by integration with respect to a suitable 
constructible function $\nu$ on the stack of all objects 
${Ob}(\CA)$.
Let 
\be\label{eq:CframedseriesA}
Z^{\nu}_{(\omega,b)}(X,C; u,T) = \sum_{n\in \IZ} \sum_{r\geq 0} 
P_{(\omega,b)}^\nu(X,C,r,n) u^nT^r
\ee
denote the resulting generating series. 
When $\nu$ is a motivic weight function, the invariants 
$P_{(\omega,b)}(X,C,r,n)$ take values in a ring of motives, 
and refined invariants $P^{ref}_{(\omega,b)}(X,C,r,n;y)$
are obtained by taking 
virtual Poincar\'e polynomials. 
For future reference note that counting invariants for objects 
$E$ of $\CA^C$ with $\ch(E)=(0,0,r[C_0],n)$ are defined analogously, and coincide with the counting invariants of the 
conifold \cite{szendroi-noncomm,NH}, \cite[Ex 6.2]{genDTI}, 
\cite{motivic-conifold}. 
Their generating function will be denoted by $Z^\nu(X,C_0,u,T)$. 
Let us fix $\omega$ and $(r,n)\in \IZ_{\geq 0}\times \IZ$. The stability parameter $b>0$ will be called {\it small} if there are no critical stability parameters of type $(r,n)$ in the interval 
$(0,\ b]$. The corresponding invariants will be denoted by 
$P^\nu_{0+}(X,C,r,n)$, and their generating function, 
$Z^\nu_{0+}(X,C,u,T)$. Moreover, for $b<<0$, the corresponding invariants 
$P_{-\infty}^\nu(X,C,r,n)$ specialize to
stable pair invariants. In the following the function $\nu$ will be either the constant function $\nu=1$ or the motivic weight 
function defined in \cite[Sect. 6.2]{wallcrossing}.

\subsection{Factorization via wallcrossing and 
small $b$ chamber}\label{onefive}

The first step in the proof of Theorem 
\ref{topCframedinv} is the derivation of a wallcrossing 
formula relating $b<<0$ invariants to small $b>0$ invariants. 
More precisely, the following result is proven in Appendix \ref{proofwallcrossing}.

\begin{prop}\label{factorization}
\be\label{eq:factformulaD} 
Z_{-\infty}^{top}(X,C,u,T) = Z^{top}(X,C_0,u,T) Z_{0+}^{top}(X,C,u,T).
\ee
\end{prop}

This is in agreement with the natural factorization of Wilson loop 
expectation values in Chern-Simons theory, 
as we explained below equation \eqref{eq:factformulaA}. 

In order to complete the proof of Theorem \ref{topCframedinv}. 
one has to find a connection between 
the moduli spaces of stable $C$-framed 
objects for small $b>0$ and the Hilbert scheme invariants 
\eqref{eq:compactHilbinv}.  
This is the content of Theorem \ref{smallBidentity} below, 
which follows from Propositions \ref{projhilbert}, \ref{geombij}.

\begin{theo}\label{smallBidentity}
There is an identity of generating functions 
\be\label{eq:smallBinvariants} 
Z_{0+}^{top}(X,C;q^2,-a^2)=q^{2\chi(\CO_{C})} 
Z_{C}^{top}(q,a)
\ee
where $Z_{C}^{top}(q,a)$ is the series \eqref{eq:compactHilbinv}. 
\end{theo} 

The proof of 
Theorem \ref{smallBidentity} relies on the construction in Section 
\ref{threetwo} of a moduli stack $\CQ(X,C,r,n)$ 
of decorated sheaves on $X$ interpolating between the 
nested Hilbert schemes $H^{[l,r]}(C)$, $l=n-\chi(\CO_C)$
and the moduli stacks ${\mathcal P}_{0+}(X,C,r,n)$.
More precisely, Proposition \ref{projhilbert} proves that 
$\CQ(X,C,r,n)$ is a $\IC^\times$ gerbe over a relative {\it Quot}
scheme $Q^{[r,n]}(C)$ which is geometrically bijective to 
$H^{[r,n]}(C)$. 
At the same time Proposition \ref{geombij} shows that 
$\CQ(X,C,r,n)$ is 
equipped with a natural 
geometric bijection $f:\CQ(X,C,r,n)\to {\mathcal P}_{0+}(X,C,r,n)$. 
Then the proof of Theorem \ref{smallBidentity} 
reduces to a straightforward stratification computation 
explained in detail at the end of Section \ref{smallsect}. 

\subsection{Generalization to motivic Donaldson-Thomas 
invariants}\label{motiviccorresp}

A motivic version of identity \eqref{eq:factformulaB} can be 
derived from the formalism of \cite{wallcrossing} following the same main steps. Assuming the required foundational results, the motivic wallcrossing 
formula of \cite{wallcrossing} implies the motivic version of the factorization formula \eqref{eq:factformulaD}, repeating the computations 
in Appendix \ref{motivicHallidentities} and \ref{summcrtval}. A similar computation using 
refined wallcrossing formulas has been carried for example in 
\cite[Sect 2.4]{wall-pairs}, hence the details will be omitted. 
This reduces the problem to the motivic analogue of Theorem 
\ref{smallBidentity}. 

The motivic Donaldson-Thomas theory of $C$-framed stable objects 
at small $b>0$ is analyzed in Section \ref{motivicsect}.  
As shown in Section \ref{fivefive}, the following identity 
holds 
\be\label{eq:motivicidentity}
Z^{mot}_{0+}(X,C;q^2\IL^{1/2},a^2) = 
\IL^{(1-k^2)/2}q^{2\chi(\CO_C)}Z^{mot}_C(q,a).
\ee
provided that the virtual 
motive of the moduli stack 
${\mathcal P}_{0+}(X,C,r,n)$ is related 
to the Chow motive by the formula 
\be\label{eq:smallBvirtmotive} 
[{\mathcal P}_{0+}(X,C,r,n)]^{vir} = \IL^{(r^2-k^2-n+1)/2} 
[{\mathcal P}_{0+}(X,C,r,n)],
\ee
where $k$ is the degree of the curve $C$ in $\IP^2$. 
This formula is proven in Section \ref{motivicsect} 
for sufficiently high degree $n>>0$, assuming the 
foundational aspects of the motivic Donaldson-Thomas 
theory of \cite{wallcrossing}, as well as a specific 
choice of orientation data. 
For arbitrary values of 
$n\in \IZ$, the equation \eqref{eq:smallBvirtmotive} 
reduces to a relation \eqref{eq:fiberintC} 
between motivic weights of moduli stacks of pairs and 
sheaves for irreducible curve classes. 
This is  a virtual motivic 
 counterpart to \cite[Thm. 4]{stabpairs-III}. Motivated by this 
 analogy, it is natural to conjecture that this relation 
 equation holds for all
 $n\in \IZ$ with a suitable choice of orientation data. 
 Granting equation \eqref{eq:smallBvirtmotive}, identity 
 \eqref{eq:motivicidentity} follows from a stratification 
 computation presented in Section \ref{fivefive}.
 
\subsection{Outlook and future directions} 

This section records potential
generalizations and extensions of the conjecture 
 of Oblomkov and Shende 
motivated by the string theory construction of \cite{largeNknots}. 
These are just possible future directions of study, not 
established mathematical results, or, in some cases, not even 
precise conjectures. Nevertheless, they are recorded here 
for the interested reader in the hope that they will lead to 
interesting developments at some point in the future.

\subsubsection{BPS states and nested Jacobians}\label{onesix}

As observed in Remarks \ref{forgetsection}, \ref{forgetsectionB}, 
a second moduli space $M^{[l,r]}(C)$ naturally enters the picture,
which can be identified with a moduli space of nested Jacobians. 
The closed points of $M^{[l,r]}(C)$ are pairs $(J,\psi)$ where 
$J$ is a rank one torsion free sheaf on $C$ of degree $-l$, and 
$\psi: J \to \CO_p^{\oplus r}$ a surjective morphism. 
According to Lemma \ref{smoothmoduli} and Remark \ref{smoothbound}, 
allowing the curve $C$ to vary in the linear system $|kH|$ on $D$
results in a smooth moduli space $\CN(D,k,r,n)$. 
Moreover, this moduli space is equipped with a natural 
determinant map 
\[
h: {\mathcal M}(D,k,r,n) \to |kH|
\]
to the linear system and $M^{[l,r]}(C)$ is the fiber 
of $h$ at the point corresponding to $C$. 
Then physics arguments \cite{MtopI,MtopII} predict that the
cohomology of $M^{[l,r]}(C)$ should admit a perverse sheaf decomposition 
\[
H(M^{[l,r]}(C))\simeq \oplus_p Gr^pH(M^{[l,r]}(C))
\] 
determined by an $h$-relative ample class. Moreover the 
dimensions of the perverse graded pieces, 
$N_r^p=\mathrm{dim}(Gr^pH(M^{[l,r]}(C))$, should 
be independent of the polarization and $n$, and the 
$C$-framed small $b$ generating function $Z_{0+}(X,C;u,T)$ should 
admit a Gopakumar-Vafa expansion 
\[ 
Z_{0+}(X,C;u,T) = {\sum_{r\geq 1} \sum_{p} N_r^p T^r u^p \over 
(1-u)^2}.
\]
Note that the $r=0$ version of these conjectures is a rigorous 
mathematical result by work of \cite{severi,MY,MS}. The construction sketched 
above provides a possible generalization for $r\geq 1$ which 
deserves further study.

\subsubsection{A conjecture for colored HOMFLY polynomials}

Theorem \ref{topCframedinv}  and the conjecture of Oblomkov and Shende imply that $C^\circ$-framed 
stable pairs on the conifold are 
related to the HOMFLY polynomial of the link of the 
singular point $p\in C^\circ$. 
Large $N$ duality arguments \cite{largeNknots} lead to 
the following generalization.

Let  $(x,y,z)$ be the affine local coordinates on 
$Y$ such that the projection $Y\to \IP^1$ is locally given 
by $(x,y,z)\to z$ and $C^\circ$ is contained in the fiber 
$z=0$. Hence $C^\circ$ is a complete intersection of the form 
\[
z =0, \qquad f(x,y)=0
\]
where $f$ is a degree $k\geq 1$ 
irreducible polynomial of two variables. 

Let $\mu$ be a Young diagram consisting of 
$m_i$ columns of height  $h_i\in \IZ_{\geq 1}$, 
where $1\leq i \leq s$ and
\[
h_1>h_2>\cdots > h_s. 
\]
Let $C^\circ_\mu$ be the complete intersection on $Y$ 
determined by the equations 
\be\label{eq:thickeningA} 
z^{h_i} f(x,y)^{m_1+\cdots+m_{i-1}}=0,\qquad 1\leq i \leq s,
\ee
 where by convention $m_0=0$. 
Note that $C^\circ_\mu$ is a nonreduced irreducible 
subscheme of $Y$ of pure dimension 
one.

In complete analogy with Sections \ref{oneone} one can define 
$C^\circ_\mu$-framed stable pair invariants of $Y$ employing the 
framing condition $Ann(F) \subset \CI_{C^\circ_\mu}$. 
Let $P(Y,C_\mu;r,n)$ denote the counting invariants 
obtained by taking the quasiclassical limit of motivic 
Donaldson-Thomas invariants  of the 
ambient space ${\calP}^{circ}({\overline Y},\beta,n)$, where 
$\beta=
[{C}_\mu]+r[C_0]$. Based on large $N$ duality, the generating function
\[
Z(Y,{C^\circ_\mu},q^2,a^2) = \sum_{n\in \IZ} \sum_{r\geq 0} q^{2n} a^{2r} 
P (Y,C^\circ_\mu,r,n). 
\]
 is expected to be related to the $\mu$-colored HOMFLY polynomial 
 $P_{K,\mu}(q,a)$. More specifically, a relation of the form 
  \be\label{eq:coloredOS}
P_{K,\mu}(q,a) =a^\alpha q^\beta  
Z^B(Y,{C_\mu},q^2,a^2)Z^B(Y,q^2,a^2).
\ee
is expected to hold, for certain
 integral exponents $\alpha, \beta$, 
possibly depending on $\mu$.

{\it Acknowledgements} We would like to thank Sergei Gukov, Sheldon Katz, 
Melissa Liu,  Davesh Maulik, Kentaro Nagao, Alexei Oblomkov, Andrei 
Okounkov, Rahul Pandharipande, Vivek Shende, Richard Thomas, and Cumrun Vafa for very helpful 
discussions.  D.-E.D is very grateful to Alexei Oblomkov, Vivek Shende 
and Cumrun Vafa for sharing their ideas and insights during collaboration on related projects. 
D.-E.D. would also like to acknowledge the partial support of the Moduli Space Program 2011
at Isaac Newton Institute, 
the Simons Workshop on Mathematics and Physics 2011,
the Simons Center for Geometry and Physics,
and NSF grant PHY-0854757-2009 during 
completion of this work. Y.S. thanks IHES for excellent research conditions. His work was partially supported by NSF grant 
DMS-1101554.

\section{Framed stable pairs in the derived 
category}\label{derivcatsect} 

\subsection{Review of slope limit stability}\label{twoone}
This section is a brief review of limit slope stability conditions 
on the derived category of a smooth projective Calabi-Yau threefold 
following \cite{polynomial,limit,generating}. 

Let $D^b(X)$ be the bounded derived category of $X$. 
Let $\CA$ be the heart of the t-structure determined by the torsion pair $({ {Coh}}_{\geq 2}(X),{{Coh}}_{\leq 1}(X))$. The objects of $\CA$ are objects $E$ of $D^b(Y)$ 
such that the cohomology sheaves $\CH^i(E)$ are nontrivial 
only for $i=-1,0$, $\CH^{-1}(E)$ has no torsion in codimension 
$\geq 2$, and $\CH^{0}(E)$ is torsion, 
of dimension $\leq 1$. 

Let $\omega$
 be a K\"ahler class on $X$ and 
$B\in H^2(X)$, a real cohomology class i.e. a $B$-field. 
Let $\CZ_{(\omega,B)}:K(X) \to \IC$ be the central charge 
function 
\[
\CZ_{(\omega,B)}(E) = -\int_X \ch(E) e^{-(B+i\omega)} 
\sqrt{\mathrm{Td}(X)}.
\]
For any $m\in \IR_{>0}$ let 
\[
\CZ^\dagger_{(\omega,B)}(E) = 
(\mathrm{Re}\CZ_{(m\omega,B)}(E))^\dagger 
+ i (\mathrm{Im}\CZ_{(m\omega,B)}(E))^\dagger 
\]
where $f^\dagger(m)$ denotes the leading monomial of 
a polynomial $f(m)$. Then for $m>>0$ the following 
\[
\mu_{(\omega,B)}(E) = -
{(\mathrm{Re}\, e^{-i\pi/4}\CZ_{(m\omega,B)}(E))^\dagger \over 
(\mathrm{Im}\, e^{-i\pi/4}\CZ_{(m\omega,B)}(E))^\dagger }
\]
is a well-defined map to the field of rational functions $\IR(m)$. 

An object $E$ of $\CA$ is said to be 
$\mu_{(\omega,B)}$-(semi)stable if any proper nonzero subobject
$0\subset F\subset E$ in $\CA$ satisfies 
\[ 
\mu_{(\omega,B)}(F) \ (\leq)\ \mu_{(\omega,B)}(E).
\]
Here rational functions $f,g\in\IR(m)$ are ordered by 
\[ 
f \geq g \qquad \Leftrightarrow \qquad f(m) \geq g(m) \quad 
\forall m>>0.
\]
According to \cite{generating} the above  slope stability gives rise to a weak stability condition 
on $\CA$. 

In order to study the properties of semistable objects of $\CA$, 
it is helpful to consider the following 
full subcategories $(\CA_1,\CA_{1/2})$ 
of $\CA$ (see \cite{limit,generating}). The category
$\CA_1\subset \CA$ consists of objects 
$E$ such that $\CH^{-1}(E)$ is torsion and $\CH^0(E)$ 
is zero dimensional. By definition $\CA_{1/2}$ is the subcategory of 
$\CA$ consisting of objects $E$ such that 
$\mathrm{Hom}_{\CA}(\CA_1,E)=0$ (i.e. it is right orthogonal to $\CA_1$) . 
Note that $\CH^{-1}(E)$ is torsion-free for 
all objects of $E$ of $\CA_{1/2}$, and also 
$\mathrm{Hom}(T,E)=0$ 
for any zero-dimensional sheaf $T$.
 According to \cite[Lemm. 2.16]{limit} the subcategories $(\CA_1, \CA_{1/2})$ 
define a  torsion pair in $\CA$. A morphism $E\to F$ of objects 
in $\CA_i$, $i=1,1/2$ will be called a {\it strict 
monomorphism/epimorphism}  if it is 
injective/surjective as a morphism in $\CA$, 
 and its cokernel/kernel belongs to $\CA_i$. 

In the following we consider objects $E$ of $\CA$ with 
$\ch_0(E)=-1$ and $\ch_1(E)=0$. The first observation 
following from \cite[Lemm. 3.8]{generating} is that 
if such an object is $\mu_{(\omega, B)}$-semistable, then 
it must belong to $\CA_{1/2}$. 
Moreover the following stability
criterion holds \cite[Prop. 3.13]{generating}. 

\begin{prop}\label{limitstabcrtA}
An object $E$ of $\CA_{1/2}$ with $\ch_0(E)=-1$ and $\ch_1(E)=0$ is $\mu_{(\omega,B)}$-(semi)stable if and only if 
the following hold. 

$(i)$ For any strict epimorphism $E\twoheadrightarrow G$ in 
$\CA_{1/2}$, with $G$ a pure dimension one sheaf on $X$ 
\[
\mu_{(\omega,B)}(G)\ (\geq)\ -{3B\omega^2\over \omega^3}. 
\]

$(ii)$ For any strict monomorphism  $F\hookrightarrow E$ in 
$\CA_{1/2}$, with $F$ a pure dimension one sheaf on $X$, 
\[
\mu_{(\omega,B)}(F)\ (\leq)\  -{3B\omega^2\over \omega^3}. 
\]
\end{prop}

Next let $\beta\in H_2(X)$ and $n\in \IZ$. Suppose 
$B=b\omega$, $b\in \IR$. 
Then the following results are proven in \cite{generating} 
for fixed $\omega$, $(\beta,n)$.

{\bf 1}. For any $b\in \IR$, there is an algebraic 
moduli stack of finite type 
${\mathcal M}_{(\omega,B)}^{ss}(\CA,\beta,n)$
of $\mu_{(\omega,B)}$-semistable objects 
of $\CA$ with $\ch(E)=(-1,0,\beta,n)$. 

{\bf 2}. There are finitely many critical parameters 
$b_c$ such that strictly $(\omega,B_c)$-semistable objects exist. 
The moduli stacks $\CM^{ss}_{(\omega,B')}(\CA,\beta,n)$, \\
$\CM^{ss}_{(\omega,B'')}(\CA,\beta,n)$ are isomorphic if there is no critical stability parameter in the interval 
$[b',b'']$. Moreover, if $b$ is not critical
all closed points of $\CM^{ss}_{(\omega,B)}(\CA,\beta,n)$
are $\mu_{(\omega,B)}$-stable and 
have $\IC^\times$ stabilizers. 

{\bf 3}. There exists $b_{-\infty}$ such 
that for any $b<b_{-\infty}$ 
the moduli stack $\CM^{ss}_{(\omega,B)}(\CA,\beta,n)$ is
a $\IC^\times$-gerbe over the 
the moduli space ${\mathcal P}(Y,\beta,n)$
of stable pairs on $Y$ constructed in \cite{stabpairs-I}. 

{\bf 4} One can define counting invariants and wallcrossing 
formulas using either the formalism of Joyce and Song or the one of Kontsevich and Soibelman. 
In particular there is a Hall algebra of motivic stack functions 
associated to the abelian category $\CA$.
The corresponding wallcrossing formulas are in agreement with 
those of   Kontsevich and Soibelman \cite{wallcrossing}. 

\subsection{A $C$-framed subcategory}\label{twotwo}
In this section $X$ will be a small crepant resolution of a
nodal Calabi-Yau threefold $X_0$  as in Section 
\ref{onethree}. In particular, it will be assumed that all 
conditions listed there are satisfied. Therefore there is only one 
conifold point lying on a Weil divisor $\Delta\simeq \IP^2\subset X_0$. The exceptional locus of the blow-up map $X\to X_0$ consists of a single rational $(-1,-1)$
curve $C_0\subset X$ which intersects the strict transform $D\subset X$ of $\Delta$ transversely at a point $p$.  
Let $B= b\omega$, $b\in \IR$, where $\omega$ is a fixed 
K\"ahler class on $X$ as above. Without loss of generality, 
it will be assumed from now on that $\omega$ is normalized such that 
$\int_{C_0}\omega=1$. 

Let $C\subset X$ be a irreducible reduced plane curve contained in $D$ 
passing through the point $p$ of intersection between 
$D$ and $C_0$.
Consider the full subcategory $\CA^C$ of $\CA$ consisting of objects 
$E$ satisfying the conditions 
\begin{itemize} 
\item[(C.1)] $\CH^{-1}(E)$ 
is a subsheaf of the defining ideal $\CI_C$. 
In particular, if $\CH^{-1}(E)$ is not trivial, it must be the 
ideal sheaf of a proper closed subscheme $Z_E\subset X$.
\item[(C.2)] The structure sheaf $\CO_{Z_E}$ and the 
cohomology sheaf $\CH^0(E)$ are topologically supported on the union $C\cup C_0$. Moreover, the quotient $\CH^0(E)/Q$ is 
topologically supported to $C_0$, 
where $Q\subset \CH^0(E)$ is  the maximal dimension 
zero subsheaf. 
\end{itemize}



\begin{lemm}\label{Cframedclosure} 
Consider an exact sequence 
\be\label{eq:exseqheart}
0\to F \to E \to G \to 0
\ee
in $\CA$ where $\ch_0(E)\in \{0,-1\}$. 
Then the following statements hold 
\begin{itemize} 
\item[$(i)$] If $F$, $G$ belong to $\CA^C$ and 
then $E$ 
belongs to $\CA^C$. 
\item[$(ii)$] If $F$, $E$ belong to $\CA^C$ 
then $G$ belongs to $\CA^C$. 
\item[$(iii)$] If $E$, $G$ belong to $\CA^C$ 
then $F$ belongs to $\CA^C$. 
\end{itemize}
\end{lemm} 

{\it Proof.}
The above statements are obvious if $\ch_0(E)=0$ since 
then $\ch_0(F)=\ch_0(G)=0$ and \eqref{eq:exseqheart} 
is a sequence of sheaves on $X$.

Suppose $\ch_0(E)=-1$, which implies 
 $\ch_0(F)=0$, $\ch_0(G)=-1$ or 
 $\ch_0(F)=-1$, $\ch_0(G)=0$. Then 
all the above statements follow easily from the long exact sequence 
\be\label{eq:longexseqheart}
0\to \CH^{-1}(F)\to \CH^{-1}(E) \to 
\CH^{-1}(G) \to \CH^{0}(F)\to \CH^{0}(E) \to 
\CH^{0}(G) \to 0
\ee
except case $(ii)$, $\ch_0(G)=-1$, $\ch_0(F)=0$, which requires more work. 
In this case, $\CH^{-1}(F)=0$ and 
$\CH^{-1}(G)$ is a rank one sheaf on $X$ which 
admits torsion at 
most in codimension one by the definition of $\CA$. Therefore the maximal torsion subsheaf $T\subset \CH^{-1}(G)$ is either zero or a nontrivial pure dimension two 
sheaf on $X$. Below it will be shown that $T$ must be zero.

Suppose $T$ is a nontrivial pure dimension two sheaf, 
and let $I\subseteq \CH^0(F)$ denote
 the image of $T$ in $\CH^0(F)$. By assumption, 
$\CH^0(F)$ has topological support on $C\cup C_0$
since $F$ belongs to $\CA^C$. Therefore $I$ has at most 
one dimensional support. Let $K$ be the 
 kernel of the induced surjective morphism 
 $T\twoheadrightarrow  I$. Then $K$ must be a
 nontrivial sheaf of pure dimension two as well. 
 Next note that there is a commutative diagram 
 \[ 
 \xymatrix{ 
 0\ar[r] & K \ar[r] & T \ar[r] \ar[d]^-{j} & I \ar[r] \ar[d] & 0 \\
 0 \ar[r] & \CH^{-1}(E) \ar[r] & \CH^{-1}(G) \ar[r] & \CH^0(F)\\}
 \]
 where the top row is exact, the bottom row is exact at the first two terms, 
 and the vertical arrows are injective. Then the image of the restriction 
 $j|_{K}$ is a subsheaf of $\CH^{-1}(E)$ and the snake lemma implies that  ${\rm Ker}(j|_K) \subset {\rm Ker}(j)$, which is trivial. Therefore 
 ${\rm Ker}(j|_K) =0$, which implies that $K$ is a subsheaf of 
 $\CH^{-1}(E)$. However, by assumption, $E$ belongs to $\CA^C$ and has rank -1, hence $\CH^{-1}(E)$ 
 is a torsion free sheaf of rank 1. This implies that $K$ must be trivial,
 leading to a contradiction.

 In conclusion, $T$ is trivial, hence 
  $\CH^{-1}(G)$ must be a rank one torsion free 
 sheaf. Moreover, the exact sequence \eqref{eq:longexseqheart} 
 implies under the current assumptions that $\CH^{-1}(G)$ 
 must have trivial determinant i.e. it must be isomorphic to the ideal 
 sheaf of a closed subscheme $Z_G$ on $X$ of dimension at most one. 
   There is also an inclusion $\CH^{-1}(E)\hookrightarrow \CH^{-1}(G)$ which implies that 
  $Z_G$ is a closed subscheme of $Z_E$ and a
   simple application of the snake lemma yields an isomorphism 
  \[
 K= \CH^{-1}(G)/ \CH^{-1}(E) \simeq 
  \mathrm{Ker}(\CO_{Z_E}\twoheadrightarrow \CO_{Z_G})
  \]
  Since $K\subset \CH^0(F)$ and both $\CO_{Z_E}$ and $\CH^0(F)$ 
   are  topologically supported on $C\cup C_0$, it follows 
   that $\CO_{Z_G}$ satisfies the same condition. Moreover, in the exact sequence \eqref{eq:longexseqheart}, $\CH^0(F)$, 
   $\CH^0(E)$
  satisfy condition (C.2), which implies 
  that $\CH^0(G)$ also satisfies (C.2). 
  Finally note that $K\subset \CH^0(F)$ is topologically supported 
  on a union of $C_0$ and a finite set of closed points lying on $C$.
  Therefore ${\rm Hom}_X(K,\CO_C)=0$ since $\CO_C$ is pure 
  dimension one by assumption. This implies that the canonical projection $\CO_{Z_E}\twoheadrightarrow \CO_C$ factors 
  through $\CO_{Z_E}\twoheadrightarrow \CO_{Z_G}$ i.e. 
  there is a surjective morphism $\CO_{Z_G}\twoheadrightarrow 
  \CO_C$ such that the diagram 
  \[
  \xymatrix{ 
  \CO_{Z_E} \ar@{>>}[r] \ar@{>>}[dr] &\CO_{Z_G}\ar@{>>}[d] \\
  & \CO_C\\}
  \]
  is commutative. 
  Hence 
  ${\CH}^{-1}(G)$ is a subsheaf of $\CI_C$.   

\hfill $\Box$

Limit slope stability for objects of $\CA^C$ will be defined by analogy 
with \cite{limit,generating}.
An object $E$ of $\CA^C$ is $(\omega, B)$-(semi)stable 
if 
\[
\mu_{(\omega,B)}(F) \ (\leq) \ \mu_{(\omega,B)}(E)
\]
for any proper 
nontrivial subobject $0\subset F\subset E$ in $\CA^C$. 
Since the K\"ahler class $\omega$ will be fixed, 
and $B=b\omega$ with $b\in \IR$, the slope $\mu_{(\omega,B)}$
will be denoted by $\mu_{(\omega,b)}$. Moreover, 
$(\omega,B)$-limit slope (semi)stable objects of $\CA^C$
will be called 
simply $\mu_{(\omega,b)}$-(semi)stable when the meaning is clear from 
the context.

Let $\CA^C_i$
 be the full subcategories of $\CA^C$ consisting of objects belonging to 
 $\CA_i$, $i=1,1/2$. Given the definition of $\CA^C$, it follows that $\CA^C_1$ is the subcategory of 
 zero dimensional subsheaves with topological support 
 on $C\cup C_0$. 
  Let $E$ be an object of $\CA^C$. Since the pair 
 $(\CA_1,\CA_{1/2})$ is a torsion pair in $\CA$ 
(see  \cite{limit}), there is an exact sequence 
 \[
 0\to E_1\to E\to E_{1/2} \to 0 
 \]
 in $\CA$ with $E_i$ in $\CA_i$, $i=1,1/2$.  Then the following holds 
 \begin{lemm}\label{Cframedtorsion} 
Let $E$ be an object of $\CA^C$. Then $E_i$ belongs to $\CA^C_i$, $i=1,1/2$. 
 \end{lemm} 
 
 {\it Proof.} Consider again the exact sequence 
\[
 0\to \CH^{-1}(E_1)\to \CH^{-1}(E) \to 
\CH^{-1}(E_{1/2}) \to \CH^{0}(E_1)\to 
\CH^{0}(E) \to 
\CH^{0}(E_{1/2}) \to 0.
\] 
By definition, $\CH^{-1}(E_1)$ must be a torsion 
sheaf of dimension two, hence it must be trivial since
$\CH^{-1}(E)$ is torsion free. Therefore $E_1\simeq 
\CH^0(E_1)$ must be a zero dimensional sheaf. 
Let $I\subset \CH^0(E)$ denote its image in 
$\CH^0(E)$ and $K= \mathrm{Ker}(\CH^0(E_1)
\twoheadrightarrow I)$.  Then there is 
an exact sequence of sheaves 
\be\label{eq:torsequence}
0\to \CH^{-1}(E)\to \CH^{-1}(E_{1/2})\to K\to 0.
\ee

Note that both $I$ and $K$ are zero dimensional sheaves and $I$ is topologically supported 
on $C\cup C_0$. 
Suppose there exists a subsheaf $K'\subset K$ with support disjoint from $C,C_0$. Since $\CH^{-1}(E)=\CI_{Z_E}$, and $Z_E$ is topologically supported on $C\cup C_0$, it follows that 
\[
{\mathcal Ext}^1_X(K', \CH^{-1}(E)) \simeq 
{\mathcal Ext}^1_X(K',\CO_X).
\] 
However \cite[Prop. 1.1.6]{huylehn} shows  that 
${\mathcal Ext}^1_X(K',\CO_X)=0$ since $K'$ is zero dimensional.
Therefore, using the local to global spectral sequence, 
${\rm Ext}^1_X(K', \CH^{-1}(E))=0$, 
which implies that there is an injection 
$K'\hookrightarrow \CH^{-1}(E_{1/2})$. This 
leads to a contradiction since $\CH^{-1}(E_{1/2})$ cannot 
have zero dimensional torsion by construction. 
 In conclusion, 
$K$, hence also $\CH^0(E_1)$, is topologically 
supported on $C\cup C_0$. In particular $E_1$ belongs to $\CA^C$.
Then Lemma \ref{Cframedclosure} implies that $E_{1/2}$ must also belong to  $\CA^C$.

\hfill $\Box$ 

A consequence of Lemma \ref{Cframedtorsion}
is that properties of limit slope semistable objects 
in $\CA$ proven in \cite{limit, generating} 
also hold in $\CA^C$. More specifically, 
 strict monomorphisms and epimorphisms of objects in $\CA^C_i$, 
 $i=1,1/2$ may be defined again by requiring that the cokernel, respectively 
 kernel belong to $\CA^C_{i}$.
 Then, 
 by analogy with \cite[Lemm. 2.27]{limit}, 
 \cite[Lemm. 3.8]{generating}, it follows again that any 
$\mu_{(\omega,b)}$-(semi)stable  object  of $\CA^C$ with 
$\ch_0(E)=-1$,
must belong to $\CA^C_{1/2}$. Moreover the stability 
criterion in Proposition \ref{limitstabcrtA}
holds for objects of $\CA^C_{1/2}$ 
provided that  $F\hookrightarrow E$, $E\twoheadrightarrow G$ 
are strict monomorphisms, respectively epimorphisms in 
$\CA^C_{1/2}$. 

Some more specific properties of limit slope semistable 
objects in $\CA^C$ are recorded below. 

\subsection{Properties of $C$-framed limit slope stable 
objects}\label{twothree}
First note that any nontrivial object $E$ of $\CA^C$ with $\ch_0(E)=0$ must be a sheaf with topological support on 
$C\cup C_0$ and $\ch_2(E)=r[C_0]$,  $r\geq 0$. Moreover, 
if  $r\geq 1$, 
\[
\mu_{(\omega,b)}(E) = {\chi(E)\over r} -b
\]
Therefore 
${(\omega,b)}$-stability for such objects 
reduces to  $\omega$-slope stability for dimension one sheaves on $X$. For completeness recall that a sheaf $E$
as above with $r\geq 1$ is 
$\omega$-slope (semi)stable if 
\[ 
\langle \omega, \ch_2(E)\rangle \chi(E') \ (\leq) \ 
\langle \omega, \ch_2(E')\rangle 
 \chi(E)
\]
for any nontrivial proper subsheaf $0\subset E'\subset E$. 
Since in the present case $\ch_2(E)=r[C_0]$, $\ch_2(E')=r'[C_0]$ 
for some $r,r'\in \IZ_{\geq 0}$, and $\omega$ is normalized 
such that $\int_{C_0}\omega=1$, this condition reduces to 
\[ 
r\chi(E')\ (\leq)\ r'\chi(E).
\]
Since any $\omega$-slope semistable sheaf must be pure, 
the defining property $(C.2)$ implies that any 
$\mu_{(\omega,b)}$-semistable 
object of $\CA^C$ must be a pure dimension one sheaf with set theoretic 
support on $C_0$. Then note the following lemma.
\begin{lemm}\label{sstablesheaves} 
Let $F$ be an $\omega$-slope semistable sheaf 
supported on $C_0$ with 
 $\ch_2(F)=rC_0$, $r\geq 1$. Then $F$ is the extension by zero of a semistable 
locally free sheaf on $C_0$. 
\end{lemm} 

{\it Proof.} 
By construction the morphism $X\to X_0$ contracts $C_0$ to an ordinary 
double point, which is analytically isomorphic to the hypersurface 
singularity $xz-yw=0$ in $\IC^4$. Therefore the formal neighborhood 
of $C_0$ in $X$ is isomorphic to the formal neighborhood 
of the zero section in the total space $Y$ of $\CO_{\IP^1}^{\oplus 2} 
\to \IP^1$. Since any sheaf $F$ as in Lemma (\ref{sstablesheaves}) 
is set theoretically supported in this neighborhood, it suffices 
to prove Lemma (\ref{sstablesheaves}) for sheaves on $Y$ with 
topological support on the zero section. Abusing notation, throughout 
the proof  the zero section of $Y\to \IP^1$ will also be denoted by 
$C_0$. 

Any sheaf $F$ on $Y$  with topological support on $C_0$ has 
$\ch_2(F) =r[C_0]$. Then slope semistability is defined by the 
condition 
\[
r \chi(F') \leq r'\chi(F)
\] 
for any proper non-trivial subsheaf $0\subset F'\subset F$. 

Let $\pi:Y\to \IP^1$ denote the canonical projection, 
and let $\CO_Y(-1)=\pi^*\CO_{\IP^1}(-1)$.
Then note that $C_0\subset Y$ is a complete intersection 
\[
s_1=s_2=0
\]
where $s_1,s_2$ are sections of $\CO_Y(-1)$. 

Now let $F$ be a slope semistable sheaf on $Y$ with topological 
support on $C_0$. Suppose one of the morphisms 
\[
\xymatrix{ 
F \ar[rr]^-{s_i\otimes 1_F} & & F\otimes_Y \CO_Y(-1) \\}
\] 
is nonzero for some $i=1,2$. Let $G\subset 
F\otimes_Y \CO_Y(-1)$ denote its image and $K\subset F$ its 
kernel. Both $G,K$ are pure of dimension one with $\ch_2(G)=r_G[C_0]$, $\ch_2(K)=r_K[C_0]$, $r_G,r_K\geq 1$, $r_G+r_K=r$. Then 
\[
{\chi(G)\over r_G} \geq {\chi(F)\over r}, \qquad {\rm{and}}
\qquad
{\chi(K) \over r_K } \leq {\chi(F)\over r}. 
\]
Moreover, it is straightforward to show that $F\otimes_Y \CO_Y(-1)$ 
must be slope semistable as well and $\chi(F\otimes_Y \CO_Y(-1)) = \chi(F) -r$. Therefore 
\[
{\chi(G)\over r_G} \leq {\chi(F)\over r} -1,
\] 
which leads to a contradiction.  In conclusion $s_i\otimes 1_F=0$ 
for both $i=1,2$. 

\hfill $\Box$


Another simple class of objects of $\CA^C$ are stable 
pairs $E=\big(\CO_X{\buildrel s\over \longto} F\big)$ 
with $F$ a pure dimension one sheaf supported on $C_0$ 
and $\mathrm{Coker}(s)$ zero dimensional. 
With the present conventions, $\CO_X,F$ are in degrees $-1,0$ 
respectively, unlike \cite{stabpairs-I}, where they have degrees 
$0,1$. 
The following 
result will be useful later.
\begin{lemm}\label{zerosect} 
Suppose $E=\big(\CO_X{\buildrel s\over \longto} F\big)$ 
is a stable pair on $X$ with $F$ a pure dimension one sheaf with topological support on $C_0$ and 
$\ch_2(F)=r[C_0]$, $r\geq 1$. Then $\chi(F)\geq r(r+1)/2$. 
\end{lemm}

{\it Proof.} 
Let $Z_F\subset X$ be the scheme theoretic support of $F$. 
Since $F$ is topologically supported on $C_0$, $Z_F$ must be 
a thickening of $C_0$. Moreover there is  an exact sequence 
\[
0\to \CO_{Z_F} {\buildrel s|_{Z_F}\over \longto} F \to 
Q\to 0
\]
with $Q$ zero dimensional. This implies that $\chi(F)\geq \chi(\CO_{Z_F})$ 
and $\ch_2(\CO_{Z_F}) = \ch_2(F)=r[C_0]$ for some 
$r\geq 1$. Note that $\CO_{Z_F}$ must be pure of dimension one 
since $F$ is so by assumption. 

In order to conclude the proof it will be shown inductively in $r\geq 1$
that $\chi(\CO_{Z})\geq r(r+1)/2$ for any pure dimension 
one $Z$ thickening of 
$C_0$ with $\ch_2(\CO_Z)=r[C_0]$. 
By analogy with the proof of Lemma (\ref{sstablesheaves}), 
it suffices to prove this for the zero section $C_0\subset Y$ 
of the total space $Y$ of the rank two bundle $\CO_{\IP^1}(-1)^{\oplus 2}$. Then $C_0$ is a complete intersection 
\[ 
s_1=s_2 =0
\]
where $s_1,s_2$ are sections of $\CO_Y(-1)=\pi^*\CO_{\IP^1}^{-1}$. 

The first step, $r=1$, is clear since in that case $\chi(\CO_{Z})=
\chi(\CO_{C_0})=1$. 
Let $r\geq 2$ and suppose the statement is true for any 
$1\leq r'\leq r$. 

First prove that 
for any pure dimension one thickening 
$Z$ of $C_0$ with $\ch_2(\CO_Z)=r[C_0]$ there exists a section 
$s_1^{k_1}s_2^{k_2}$ of $\CO_Y(-r)$, with $k_1,k_2\geq 0$, 
$k_1+k_2=r$ such that the morphism 
\[
\zeta(k_1,k_2) : \CO_Z\otimes_Y \CO_{Y}(r) \to \CO_Z
\]
given by multiplication by $s_1^{k_1}s_2^{k_2}$ is nonzero 
and $\ch_2({\rm Coker}(\zeta(k_1,k_2)))= r'[C_0]$ 
for some $r'>0$. 

If all morphisms $\zeta(k_1,k_2)$, $k_1,k_2\geq 0$, $k_1+k_2=r$ 
are trivial, it is easy to show that $\ch_2(\CO_Z)\leq r-1$, leading to a 
contradiction. Therefore at least one of them, $\zeta(j_1,j_2)$ 
must be nontrivial. Let $I={\rm Im}(\zeta(j_1,j_2))$
and $G={\rm Coker}(\zeta(j_1,j_2))$ and let 
\[
\ch_2(I)=r_I[C_0], \qquad \ch_2(G)=r_G[C_0]
\]
with $r_I,r_G\geq 0$, $r_I+r_G=r$. 

Suppose $r_G=0$, that is $G$ is zero dimensional, possibly 
trivial. Then $r_I=r$ and the kernel of $\zeta(j_1,j_2)$ must 
be a zero dimensional sheaf. Hence $\zeta(j_1,j_2)$ must be injective 
since $\CO_Z\otimes_Y \CO_Y(-r)$ is pure of dimension one. Therefore 
\[
\chi(G) = \chi(\CO_Z) -\chi(\CO_Z\otimes_Y \CO_Y(r)) = -r^2 <0
\]
which is a contradiction. In conclusion $r_G>0$, hence 
$0<r_I,r_G<r$. 

Let $T(G)\subset G$ be the maximal dimension zero subsheaf of $G$. 
Then $G/T(G)$ is a 
pure dimension one quotient of $\CO_Z$,  hence it must be the structure 
sheaf of a thickening of $C_0$.
Using the inductive hypothesis, 
\[
\chi(G) \geq \chi(G/T(G)) \geq {1\over 2}r_G(r_G+1).
\]
Similarly, $I\otimes_Y \CO_Y(-r)$ is a pure dimension one 
quotient of $\CO_Z$ with $r_I<r$. Therefore the induction hypothesis 
implies 
\[
\chi(I\otimes_Y \CO_Y(-r)) \geq {1\over 2}r_I(r_I+1),
\]
which yields 
\[
\chi(I) \geq r_I r + {1\over 2}r_I(r_I+1).
\]
In conclusion, since $r_G=r-r_I$, 
\[ 
\chi(\CO_Z) = \chi(G)+\chi(I) \geq r_I r + {1\over 2}r_I(r_I+1)+{1\over 2}r_G(r_G+1) = {1\over 2}r(r+1) + {1\over 2} r_I^2 
\geq {1\over 2}r(r+1).
\]

\hfill $\Box$

In order to derive similar structure results for more general
$\mu_{(\omega,b)}$-semistable objects, it will be helpful to note the following technical result. 

\begin{lemm}\label{extlemma} 
Let $F_C$ be a pure dimension one sheaf on $X$ 
with scheme 
theoretic support on $C$ and $F_0$ a pure dimension 
one sheaf on $X$ 
with topological support on $C_0$. Recall that the curve $C$, hence 
also the sheaf $F_C$, is scheme theoretically supported on a divisor $D
\simeq \IP^2$ in $X$ which intersects $C_0$ transversely at one point. 
Then, choosing a trivialization of the dualizing sheaf of $D$ at $p$,  
$\omega_D|_p\simeq 
\CO_p$, there are isomorphisms 
\be\label{eq:extisomG}
\varphi_k:{\rm Ext}^k_X(F_C,F_0) {\buildrel \sim 
\over \longto}
\mathrm{Ext}^{k-1}_D(F_C, \CO_D\otimes_X F_0)
\ee 
for all $k\in \IZ$, where ${\rm Ext}_D^k$ are global extension groups 
of $\CO_D$-modules. 
Moreover, suppose 
\be\label{eq:extensionA} 
0\to F_0 \to F \to F_C \to 0 
\ee
is an extension of $\CO_X$-modules corresponding to an extension class $e\in {\rm Ext}^1_X(F_C,F_0)$ and 
let 
$F_C'\subset F_C$ be a subsheaf of $F_C$. 
Then $e$ is in the kernel of the natural map 
\[ 
{\rm Ext}_X^1(F_C,F_0)\to {\rm Ext}^1_X(F'_C,F_0)
\]
if and only if $F'_C\subset \mathrm {Ker}(\varphi_1(e))$, where 
we consider $\varphi_1(e)$ to be a map $F_C\to F_0|_D$ of $\CO_D$-modules. 
\end{lemm} 

{\it Proof.} 
The adjunction formula the canonical embedding $i:D
\hookrightarrow X$ yields a quasi-isomorphism 
\be\label{eq:adjunctionA}
{\rm RHom}_X(F_C,F_0) \simeq {\rm RHom}_D(F_C, i^!F_0)
\ee
where 
\[ 
i^!F_0 = Li^*F_0 \otimes \omega_D[-1].
\]
Note that the cohomology sheaves of the complex $Li^*F_0$ are 
isomorphic to the local tor sheaves 
\[ 
\CH^k(Li^*F_0) \simeq {\mathcal Tor}_{-k}^{X}(F_0,\CO_D)
\]
for all $k\in \IZ$. Moreover local tor is symmetric in its arguments,
that is  
\[ 
{\mathcal Tor}_{-k}^{X}(F_0,\CO_D)\simeq {\mathcal Tor}_{-k}^{X}(\CO_D,F_0).
\]
Since $F_0$ is pure of dimension one, 
using the canonical locally free resolution 
\[
\CO_X(-D){\buildrel \zeta_D\over \longto} \CO_X
\] 
of $\CO_D$, it follows that ${\mathcal Tor}_{-k}^{X}(\CO_D,F_0) =0$ 
 for all $k\neq 0$. Therefore the complex $Li^*F_0$ is quasi-isomorphic to the sheaf ${\mathcal Tor}_0^X(\CO_D,F_0) \simeq \CO_D\otimes_X F_0$. Then \eqref{eq:adjunctionA} yields isomorphisms 
of the form \eqref{eq:extisomG}. 

The second statement follows from the functoriality of the adjunction formula. 

\hfill $\Box$ 

For future reference note the following corollary of Lemma 
\ref{extlemma}.
\begin{coro}\label{nondegext}
Under the same conditions as in Lemma \ref{extlemma} 
suppose $F_0= V\otimes \CO_{C_0}(-1)$ with 
$V$ a finite dimensional vector space and let 
$e\in {\rm Ext}^1_X(F_C,F_0)$ be an extension class. 
Let $\psi =\varphi_1(e)\in {\rm Hom}_D(F_C, V\otimes\CO_p)$ 
be the corresponding morphism of $\CO_D$-modules, 
where $\CO_p$ is the structure sheaf of the transverse 
intersection point $\{p\}=D\cap C_0$. 
Then the following conditions are equivalent 
\begin{itemize}
\item[$(a)$] The class $e$ is not in the kernel of the natural map 
\[
q_* :{\rm Ext}^1_X(F_C,V\otimes\CO_{C_0}(-1)) \to {\rm Ext}^1_X(F_C, V'\otimes\CO_{C_0}(-1)),
\]
for 
any nontrivial quotient $q:V\twoheadrightarrow V'$.
\item[$(b)$] The morphism $\psi:F_C \to V\otimes\CO_p$ 
is surjective. 
\end{itemize} 
\end{coro} 

{\it Proof.} Suppose an extension class $e$ satisfies condition $(a)$ 
and $\psi$ is not surjective. 
Then the image of $\psi$ is $V''\otimes 
\CO_p$ where $V''\subset V$ is a proper subspace of $V$. 
Let $V\simeq V'\oplus V''$ be a direct sum decomposition, 
and $q:V\twoheadrightarrow V'$ the natural projection. 
Then the second part of Lemma \ref{extlemma} implies that 
such that $e$ lies in the kernel of the map $q_*$, leading 
to a contradiction. The proof of the converse statement is 
analogous. 

\hfill $\Box$

Now let $E$ be an object of $\CA^C$ with $\ch_0(E)=-1$ 
and let $F$ 
be a torsion sheaf on $X$ of dimension at most one. 
Then there is an exact sequence 
\be\label{eq:homexseqA}  0\to \mathrm{Ext}_X^1(F,\CH^{-1}_E) \to 
\mathrm{Hom}_{D^b(X)}(F,E) \to 
\mathrm{Ext}_X^0(F,\CH^0(E)) \to \mathrm{Ext}_X^2(F,\CH^{-1}(E)) 
\to \cdots
\ee
Moreover, since $\CH^{-1}(E)= \CI_{Z_E}$ is the ideal sheaf of a 
dimension one subscheme, there is also 
an exact sequence 
\be\label{eq:homexseqB}
0\to \mathrm{Ext}_X^0(F,\CO_{Z_E}) \to
\mathrm{Ext}_X^1(F,\CH^{-1}(E))\to 
\mathrm{Ext}^2_X(F,\CO_X) \to \cdots
\ee
\begin{lemm}\label{structlemmaA}
Suppose $E$ is an object of $\CA^C_{1/2}$
with $\ch_0(E)=-1$, $\ch_2(E) = [C] + r [C_0]$, 
$r\geq 0$. Then the following hold 
\begin{itemize} 
\item [$(i)$] $Z_E$ is a pure dimension one subscheme of $X$. 
\item [$(ii)$] There is a commutative diagram 
of morphisms of $\CO_X$-modules 
\be\label{eq:structdiagA}
\xymatrix{ 
& & 0 \ar[d]& 0\ar[d] & \\
& & K_C\ar[r]^-{1}\ar[d]& K_C \ar[d]& \\
0\ar[r] & K_0\ar[d]_-{1}\ar[r] & \CO_{Z_E}\ar[r]^-{f} \ar[d] & \CO_C
\ar[d]_-{\phi}\ar[r]&0\\
0\ar[r] & {K_0}\ar[r] & \CO_{Z_0}\ar[r] \ar[d] & 
\CO_{Z_0\cap C} \ar[d]\ar[r]&0\\
& & 0 & 0 & \\}
\ee
where $Z_0\subset X$ is a pure dimension one closed 
subscheme of $X$ with topological support on $C_0$, 
and $\CO_{Z_0\cap C}$ the structure sheaf of the scheme 
theoretic intersection $Z_0\cap C$ in $X$. 
\item[$(iii)$] $\chi(K_0)\geq 0$.
\end{itemize}
\end{lemm} 

{\it Proof.} 
Purity of $\CO_{Z_E}$ follows from the 
observation that any nontrivial morphism $F\to\CO_{Z_E}$ with 
$F$ zero dimensional would yield via the 
exact sequences \eqref{eq:homexseqA}, 
\eqref{eq:homexseqB} 
a nontrivial morphism $F\to E$ in $\CA^C$. 
This contradicts the assumption that $E$ belongs 
to $\CA^C_{1/2}$. 

Next, the given conditions on the Chern classes of $E$ imply that 
\be\label{eq:chernclasses}
\ch_2(\CH^{0}(E)) = r_0[C_0], \qquad 
\ch_2(\CH^{-1}(E))= -[C]-r_{-1}[C_0]
\ee
with $r_0,r_{-1}\geq 0$, $r_0+r_{-1}=r$. 
Moreover there is an exact sequence of $\CO_X$-modules   
\be\label{eq:ZEext}
0\to K_0 \to \CO_{Z_E} \to \CO_C \to 0 
\ee
where $K_0$ is a pure dimension one sheaf with topological 
support on $C_0$. 
According to Lemma \ref{extlemma}, 
there is an isomorphism 
\[
{\varphi}_1:{\rm Ext}^1_X(\CO_C,K_0) \simeq  {\rm Hom}_{X}(\CO_C,
\CO_{D}\otimes_X K_0).
\]
identifying the extension class $e\in {\rm Ext}^1_X(\CO_C,K_0)$
determined by \eqref{eq:ZEext} with a morphism 
$\phi\in {\rm Hom}_{X}(\CO_C,
\CO_{D}\otimes_X K_0)$.
Let $K_C =\mathrm{Ker}(\phi)$ and $I={\rm Im}(\phi)$.  Then Lemma \ref{extlemma} 
also implies 
that the restriction of the extension class $e$
to $K_C\subset \CO_C$ is trivial. 
Therefore there is a commutative diagram 
\[
\xymatrix{ 
& & 0 \ar[d]& 0\ar[d] & \\
& & K_C\ar[r]^-{1}\ar[d]& K_C \ar[d]& \\
0\ar[r] & K_0\ar[d]_-{1}\ar[r] & \CO_{Z_E}\ar[r]^-{f} \ar[d] & \CO_C
\ar[d]_-{\phi}\ar[r]&0\\
0\ar[r] & {K_0}\ar[r] & G \ar[r] \ar[d] & 
I\ar[d]\ar[r]&0\\
& & 0 & 0 & \\}
\]
with exact rows and columns. 
Obviously, $G$ is the structure sheaf of 
a closed subscheme $Z_0\subset X$. 
The support conditions on 
$\CO_{Z_E}$ and equations \eqref{eq:chernclasses} 
 imply that $G$ is 
topologically supported on $C_0$ and $\ch_2(G)=r_{-1}[C_0]$. 
Moreover $I$ is isomorphic to the structure sheaf $\CO_{Z_0\cap C}$
of the scheme 
theoretic intersection $Z_0\cap C$. 

In order to prove that $G$ is pure, 
suppose $T\subset G$ is the maximal 0 dimensional 
subsheaf and let $G'=G/T$. Then $G'$ is pure dimension one 
with $\ch_2(G') =  r_{-1}[C_0]$. Let $K_C'$ be the kernel of
the resulting epimorphism $\CO_{Z_E}\twoheadrightarrow G'$. 
Then there is a commutative diagram 
\[
\xymatrix{ 
0\ar[r] & K_C \ar[r]\ar[d] & \CO_{Z_E} \ar[r]\ar[d]^-{1} & G \ar[r]\ar[d] & 0 \\
0\ar[r] & K'_C \ar[r] & \CO_{Z_E} \ar[r] & G' \ar[r] & 0 \\}
\]
which implies that the morphism $K_C\to K'_C$ is injective and 
$K'_C/K_C \simeq T$. 
Hence $K'_C$ is a pure dimension one subsheaf of $\CO_{Z_E}$
with support on $C$ and $\ch_2(K'_C)=\ch_2(K_C)=[C]$. 
Then the restriction $f\big|_{K_C'}:K_{C}'\to \CO_C$ 
must be injective since ${\rm Hom}_X(K'_C,K_0)=0$. 
Therefore the restriction 
\[
0\to K_0 \to F \to K_C' \to 0 
\] 
of the 
extension \eqref{eq:ZEext} to $K_C'\subset \CO_C$ 
must be trivial. This implies that $K_C'$ is contained in the 
kernel of $\phi$, which is $K_C$. Therefore $K_C'=K_C$, which 
implies $G'=G$, and $T=0$. In conclusion $G$ is of pure 
dimension one. 

The third statement of Lemma \ref{structlemmaA} 
follows from the observation that 
the canonical surjective morphism $\CO_X\twoheadrightarrow \CO_{Z_0}$ determines a stable pair on $X$ with support on $C_0$. According to 
Lemma \ref{zerosect}, this implies that 
\[
\chi(\CO_{Z_0}) \geq r_{-1}.
\]
However, since $K_0$ is pure dimension one 
with topological support on $C_0$, and the divisor 
$D$ is transverse to $C_0$, 
there is an exact sequence 
\[
0\to K_0(-D) \to K_0 \to \CO_D\otimes_X K_0\to 0.
\]
Then 
the Riemann-Roch theorem yields 
\[
\chi(\CO_{D}\otimes_X K_0)= r_{-1}
\]
as $\ch_2(K_0) = r_{-1}[C_0]$ and $C_0\cdot D=1$. Since $
\CO_{Z_0\cap C}\subseteq 
\CO_D \otimes_X K_0$ is an inclusion of zero dimensional sheaves, 
$\chi(\CO_{Z_0\cap C})\leq r_{-1}$. 
Therefore 
\[
\chi(K_0) = \chi(\CO_{Z_0}) - \chi(\CO_{Z_0\cap C}) \geq 0.
\]

\hfill $\Box$

Next note that there is an injective morphism 
\[
\CO_{Z_E} \hookrightarrow \CH^{-1}(E)[1]
\]
in $\CA^C$ corresponding to the canonical extension 
\be\label{eq:canseq}
0\to \CH^{-1}(E)\to \CO_X \to \CO_{Z_E}\to 0.
\ee
Therefore  the canonical inclusion $K_0\subset
\CO_{Z_E}$ is a subobject of $\CH^{-1}(E)[1]\subset E$. 
\begin{lemm}\label{K0subobject} 
Suppose $E$ is an object of $\CA^C_{1/2}$
with $\ch_0(E)=-1$, $\ch_2(E) = [C] + r [C_0]$, 
$r\geq 0$. Then there is an exact sequence 
\be\label{eq:structseqA}
0\to K_0 \to  E
\to G \to 0 
\ee
in $\CA^C$ where $\CH^{-1}(G)\simeq \CI_C$ 
and $\CH^0(G)\simeq \CH^0(E)$.
\end{lemm} 

{\it Proof.} 
Since $K_0$ is pure of dimension one, 
it belongs to $\CA^C_{1/2}$.  According to Lemma \ref{Cframedclosure}, $G=E/K_0$ belongs to 
$\CA^C$. Moreover, note that the  morphism 
\[ 
K_0\to  E\twoheadrightarrow \CH^0(E) 
\]
is trivial since $K_0$ is a subobject of $\CH^{-1}(E)[1]$. 
Then the long exact cohomology sequence 
of \eqref{eq:structseqA} yields exact sequences 
of $\CO_X$-modules 
\[ 
\bal 
0\to & \CH^{-1}(E) \to \CH^{-1}(G) \to K_0 \to 0 \\
0\to & \CH^0(E) \to \CH^0(G) \to 0. \\
\eal 
\] 
The first exact sequence is the restriction of \eqref{eq:canseq} 
to $K_0\subset \CO_{Z_E}$. 
Then a simple application of the snake lemma shows that 
$\CH^{-1}(G)\simeq \CI_C$. 

\hfill $\Box$

Finally note the following observation. 
\begin{lemm}\label{Cpairs} 
Let $E$ be an object of $\CA^C_{1/2}$ with $\ch(E)=(-1,0,[C],n)$, 
$n\in \IZ$ such that $\CH^0(E)$ is a zero dimensional sheaf. 
Then $E$ is isomorphic to a stable pair $P_C=\big(\CO_X{\buildrel 
s\over \longto} F_C\big)$ with $F_C$ a pure dimension one sheaf with 
scheme theoretic support on $C$. Moreover $E$ is $\mu_{(\omega,b)}$-stable 
for any $b\in \IR$.  
\end{lemm}

{\it Proof.} The first part follows from \cite[Lemm. 4.5]{limit}. 
In particular there is an exact sequence 
\[
0\to F_C \to E \to \CO_X[1]\to 0
\]
in $\CA$. 
For the second, note that there are no strict epimorphisms 
$E\twoheadrightarrow G$ 
in $\CA^C_{1/2}$ with $G$ pure dimension one since
${\rm Hom}_{\CA}(E,G)\subset {\rm Hom}_X(\CH^0(E),G)$ and 
${\CH^0(E)}$ is zero dimensional. 
Furthermore, for any pure dimension one sheaf $G$ 
with support on $C_0$, there is an exact sequence 
\[ 
0\to {\rm Hom}_{\CA}(E,G) \to {\rm Ext}^1_X(\CO_X,G) \to \cdots 
\]
since ${\rm Hom}_X(G,F_C)=0$. Serre duality yields 
and isomorphism ${\rm Ext}^1_X(\CO_X,G)\simeq H^2(G)^\vee$, 
hence there are no strict monomorphisms $G\hookrightarrow E$ in 
$\CA^C_{1/2}$. 

\hfill $\Box$

\section{Stable pairs at small $b$}\label{smallsect}

The goal of this section is to analyze the structure of 
the moduli stacks $\calP_{(\omega,b)}(X,C,r,n)$ for $b>0$ sufficiently close 
to $0$, in particular to prove Theorem \ref{smallBidentity}. 
Here again $X=X^+$ and $C=C^+$ and one fixes a K\"ahler 
class $\omega$ such that $\int_{C_0}\omega=1$. 
For the present purposes 
it suffices to consider only generic values of $b$, in which case
strictly semistable objects with numerical 
invariants $(r,n)$ do not exist. 
The main technical result in this section is the stability 
criterion obtained in Proposition \ref{smallBstab} below.  
The proof is fairly long and complicated, and will be carried out 
in several stages in Section \ref{threeone}. Sections 
\ref{threetwo} and \ref{fourthree} summarize several applications 
of this result, explaining the connection between moduli stacks 
of $C$-framed perverse coherent sheaves and nested Hilbert schemes .

\subsection{A stability criterion}\label{threeone}
Fix $\ch(E)=(-1,0,[C]+r[C_0],n)$, $r\in \IZ_{\geq 0}$, 
$n\in \IZ$ and  K\"ahler class $\omega$ on $X$ such that $\int_{C_0}\omega=1$.

\begin{lemm}\label{smallBlemmaA} 
For fixed K\"ahler class $\omega$, and fixed $(r,n)\in \IZ_{\geq 0}\times \IZ$ the following holds for any stability parameter $b>0$ such that  
\be\label{eq:BboundA}
b<{1\over 2r} 
\ee
if $r>0$. 

Any $\mu_{(\omega,b)}$-stable object $E$ of $\CA^C$ 
 with $\ch(E)=(-1,0,[C]+r[C_0],n)$ fits in an exact 
 sequence 
\be\label{eq:smallBstabA}
0\to P_C \to E \to G \to 0
\ee
in $\CA^C$ such that 
\begin{itemize}
\item[$(i)$] 
$P_C=\big(\CO_X{\buildrel s\over \longto} F_C\big)$ is a stable pair 
on $X$ with $F_C$ scheme theoretically supported on 
$C$.
\item[$(ii)$] $G$ is a pure dimension one sheaf on $X$ 
with topological support on $C_0$ and $\ch_2(G)=r[C_0]$. 
Moreover its Harder-Narasimhan 
filtration 
\be\label{eq:HNfiltrationA} 
0=G_0 \subset G_1 \subset \cdots \subset G_h=G
\ee
 with respect to $\omega$-slope stability satisfies 
 \[ 
G_j/G_{j-1} \simeq \CO_{C_0}(a_j)^{\oplus s_j}, 
\]
where $a_j\in \IZ_{\geq -1}$ for $j=1,\ldots, h$, 
and $a_1>a_2>\cdots> a_h$.  
 \end{itemize}
\end{lemm} 

{\it Proof.} 
It will be first proven 
that for $b>0$ sufficiently close to $0$, the stability criterion 
in Proposition \ref{limitstabcrtA}
implies that $\CH^{-1}(E)=\CI_C$ for any  $\mu_{(\omega,b)}$-stable
object $E$ of $\CA^C$ with 
$\ch(E)=(-1,0,[C]+r[C_0],n)$.

Suppose $E$ is $\mu_{(\omega,b)}$-stable for some $b>0$. This implies that $E$ is an element of $\CA^C_{1/2}$ satisfying the stability criterion 
in Proposition \ref{limitstabcrtA} with respect to 
strict morphisms in $\CA_{1/2}^C$. 
According to Lemmas \ref{K0subobject}, 
\ref{structlemmaA}
 there is an injective morphism 
$\kappa: K_0\to E$ where 
$K_0$ is pure of dimension one and  $\chi(K_0)\geq 0$. 
Suppose $K_0$ is nontrivial, and $\ch_2(K_0) = 
r_{-1}[C_0]$ for some $r_{-1}>0$. 
Lemma \ref{Cframedtorsion} 
implies that the cokernel  $G={\rm Coker}(\kappa)$
fits in an exact sequence 
\[
0\to G_1\to G \to G_{1/2}\to 0
\]
where $G_i\in \CA^C_i$, $i=1,1/2$. In particular 
$G_1$ is a zero dimensional sheaf. Then the snake lemma implies that the kernel $K$ of the projection 
$E\twoheadrightarrow G_{1/2}$ is a one dimensional 
sheaf on $X$ which fits in an exact sequence 
\[
0\to K_0 \to K \to G_1 \to 0
\]
 Moreover, $K$ must be pure of dimension one 
 since $E$  belongs to 
 $\CA^C_{1/2}$. 
 Then the stability criterion in Proposition \ref{limitstabcrtA} implies that 
 \[
 {\chi(K)\over r_{-1}}  < -2b
 \]
 if $r_{-1}>0$. 
 Since $b>0$, this implies that 
  $\chi(K)<0$, hence also $\chi(K_0)<0$ since 
 $G_1$ is a zero dimensional sheaf. 
 This contradicts Lemma \ref{structlemmaA} ($iii$) 
 unless $K_0$ is trivial. In conclusion, $\CH^{-1}(E) 
 =\CI_C$ for any $b>0$. 
 
 This implies in particular that $\ch_2(\CH^0(E))= r[C_0]$
 for $b>0$. Therefore, if $r=0$, $\CH^0(E)$ must be a zero dimensional sheaf with topological support on $C$. Since $E$ 
 belongs to $\CA^C_{1/2}$, and $\ch_2(E) = [C]$, 
 Lemma \ref{Cpairs} implies that 
 $E$ must be isomorphic to a stable pair $P_C = (\CO_X {\buildrel 
 s\over \longto } F_C)$, with $F_C$ scheme theoretically 
 supported on $C$.

Next suppose $r>0$, $E$ is $\mu_{(\omega,b)}$-stable for some $b>0$, 
and let $\CH^0(E)\twoheadrightarrow G$ be a 
nontrivial pure dimension one quotient of $\CH^0(E)$. 
Condition (C.2) implies that 
$G$ is topologically supported on $C_0$, 
and $\ch_2(G)=r_G[C_0]$ for some $0<r_G\leq r$. 
 Let $F$ be the 
kernel of the epimorphism $E\twoheadrightarrow G$. 
Then  $F$ belongs to $\CA^C_{1/2}$ since $E$ does, 
hence  the epimorphism 
$E\twoheadrightarrow G$ is strict. Moreover 
$\CH^{-1}(F) \simeq \CH^{-1}(E)$. Then the stability criterion 
in Proposition \ref{limitstabcrtA}  implies that 
\[ 
 {\chi(G)\over r_G}> -2b .
\]
Note that $-2b>-1/r$ if the bound 
\eqref{eq:BboundA} holds.
Since $0<r_G \leq r$ this implies that
\[
\chi(G) > -{r_G\over r} \geq  -1,
\]
hence
$\chi(G) \geq 0$, since $\chi(G)\in \IZ$. 

Now consider the exact sequence 
\[ 
0\to Q\to \CH^0(E)\to G\to 0
\]
where $Q\subset \CH^0(E)$ is the maximal 
zero dimensional subsheaf of $\CH^0(E)$ and $G$ 
is pure of dimension one supported on $C_0$. 
The previous argument implies 
that any pure dimension one quotient $G\twoheadrightarrow 
G'$ must have $\chi(G')\geq 0$ if \eqref{eq:BboundA} 
holds. 
In particular, using Lemma \ref{sstablesheaves}, 
$G$ has a Harder-Narasimhan filtration of the form 
\eqref{eq:HNfiltrationA}. 

Let $E'=\mathrm{Ker}(E\twoheadrightarrow G)$. 
Obviously, $E'$ belongs to $\CA^C_{1/2}$ and there is an exact 
sequence 
\be\label{eq:stabpairseq} 
0\to \CI_C[1] \to E' \to Q \to 0 
\ee
in $\CA^C$. Then \cite[Lemm. 4.5]{limit} 
implies that $E'$ is isomorphic to the stable pair $P_C=\big(\CO_X {\buildrel s\over \longto } F_C\big)$ in 
$\CA^C$, where $s$ is determined by the 
natural projection $\CO_X\twoheadrightarrow \CO_C$. 

\hfill $\Box$

\begin{lemm}\label{smallBlemmaB} 
Under the same conditions as in Lemma \ref{smallBlemmaA} 
suppose the bound \eqref{eq:BboundA} is satisfied.  
Then for any $\mu_{(\omega,b)}$-stable object $E$ of $\CA^C$ with 
$\ch(E)=(-1,0,[C]+r[C_0],n)$, with $r>0$, the quotient $G$ 
in \eqref{eq:smallBstabA} must be of the form 
$G\simeq \CO_{C_0}(-1)^{\oplus r}$.\end{lemm} 

{\it Proof.} 
Using the notation in  Lemma \ref{smallBlemmaA}
note the exact sequence 
\be\label{eq:pairseq}
0 \to F_C \to P_C\to \CO_X[1] \to 0 
\ee
in $\CA$, which yields a long exact sequence 
exact sequence 
\[ 
\cdots \to {\rm Ext}^1_X(G,\CO_X) \to 
{\rm Ext}^1_X(G,F_C) \to 
{\rm Ext}^1_{\CA}(G,P_C) \to 
{\rm Ext}^2_{X}(G,\CO_X) \cdots 
\]
Note also that 
\[ 
{\rm Ext}^1_X(G,\CO_X) \simeq H^2(X,G)^\vee =0
\]
by Serre duality and the structure of the Harder-Narasimhan 
filtration of $G$ implies 
\[ 
{\rm Ext}^2(G,\CO_X) \simeq H^1(X,G)^\vee =0 
\]
as well. 
Therefore there is an isomorphism 
\be\label{eq:extisomA}
 {\rm Ext}^1_X(G,F_C) \simeq 
{\rm Ext}^1_{\CA}(G,P_C).
\ee
This implies that there is an 
extension 
\be\label{eq:FCA}
0\to F_C {\to}  F \to G \to 0 
\ee
determined by the extension class of $E$ up to isomorphism,
which fits in a commutative diagram 
\[ 
\xymatrix{ 
0\ar[r] & F_C \ar[d] \ar[r] & F \ar[d] \ar[r]&  G 
\ar[d]^-{1}\ar[r] & 0 \\
0\ar[r] & P_C \ar[r] & E  \ar[r]&  G \ar[r] & 0 \\}
\]
in $\CA$. 
In particular the middle vertical morphism is injective
 in $\CA$, and  
\be\label{eq:EmodF}
E/F \simeq P_C/F_C\simeq \CO_X[1].
\ee
Since $E$ belongs 
to $\CA_{1/2}^C$, $F$ has to be a pure dimension 
one sheaf. 

Let $F'_C$ be the quotient $F\otimes_X \CO_C/T$ 
where $T\subset F\otimes_X \CO_C$ is the maximal 
zero dimensional submodule. Then there is an exact sequence 
\be\label{eq:FCB}
0\to G' \to F \to F'_C \to 0 
\ee
where $G'$ is pure dimension one with topological support of $C_0$ and $F'_C$ pure dimension one with support on $C$. 
Moreover, $\ch_2(G')=\ch_2(G) =r[C_0]$ and 
$\ch_2(F'_C) =\ch_2(F_C) = [C]$. In particular 
$F'_C$ is scheme theoretically supported on 
$C$. 
Obviously $G'\subset F\subset E$ 
in $\CA^C$, and there is an exact sequence 
\[ 
0\to F'_C\to E/G'\to \CO_X[1]\to 0
\]
in $\CA$. This implies that $E/G'$ belongs to $\CA_{1/2}$ since 
both $F'_C,\CO_X[1]$ do. As $E/G'$ also belongs to $\CA^C$ 
according to Lemma \ref{Cframedclosure}, it follows that 
$E/G'$ belongs to $\CA^C_{1/2}$ i.e. the inclusion 
$G'\subset E$ is strict. 

The same holds for any proper saturated subsheaf $G''\subset G'$ 
(that is a subsheaf such that $G'/G''$ is pure dimension one.) 
For any such sheaf there is an exact sequence 
\[
0\to G'/G''\to E/G''\to E/G'\to 0
\]
in $\CA$ which implies that $E/G''$ belongs to $\CA_C^{1/2}$. 
Then the stability criterion implies that 
\be\label{eq:satsubsheaf}
{\chi(G'')\over r_{G''}} < -2b <0 
\ee
for any saturated subsheaf $G''\subseteq G'$, where $\ch_2(G'') 
=r_{G''}[C_0]$, $0<r_{G''}\leq r$.

Next recall that according to Lemma \ref{extlemma} there is an isomorphism 
\be\label{eq:extisomB}
\varphi_1:{\rm Ext}^1_X(F_C',G') {\buildrel \sim\over \longto}
\mathrm{Hom}_D(F_C', \CO_D\otimes_X G'). 
\ee
Since $G'$ is pure dimension one supported on $C_0$, 
there is an exact sequence 
\[ 
0\to G'(-D)\to G' \to \CO_D\otimes_X G'\to 0
\]
 This 
implies via the Riemann-Roch theorem that 
\[ 
\chi(\CO_D\otimes_X G') = r.
\] 
Now let $\phi:F_C'\to {\mathcal Ext}^0_D(\CO_D, G')$ be the morphism corresponding to the extension class of \eqref{eq:FCB} under the isomorphism 
\eqref{eq:extisomB}. The exact sequences \eqref{eq:FCA}, \eqref{eq:FCB} imply that there 
is an  injective morphism
$F_C\hookrightarrow F_C'$ 
such that the following diagram commutes 
\[
\xymatrix{ 
F_C \ar@{^{(}->}[r] \ar@{^{(}->}[dr] & F \ar[d] \\
& F'_C \\}
\]
since both $F_C,F'_C$ are pure supported on $C$ and 
$\ch_2(F_C)=\ch_2(F'_C)=[C]$.  
This implies that the restriction 
of the extension \eqref{eq:FCB} to $F_C\subset F'_C$ 
is split i.e. $F_C \subset \mathrm{Ker}(\phi)$. 
In conclusion the quotient $F_C'/F_C$ is a 
subsheaf of $\CO_D\otimes_X G'$ and 
there is a commutative diagram 
\[ 
\xymatrix{ 
& & 0 \ar[d]& 0\ar[d] & \\
& & { F}_C\ar[r]^-{1}\ar[d]& {F}_C \ar[d]& \\
0\ar[r] & G'\ar[d]_-{1}\ar[r] & F\ar[r] \ar[d] & F'_C \ar[d]\ar[r]&0\\
0\ar[r] & G'\ar[r] & G \ar[r] \ar[d] & F'_C/F_C\ar[d]\ar[r]&0\\
& & 0 & 0 & \\}
\]
with exact rows and columns. 
The bottom row of this diagram yields 
\be\label{eq:GprimeineqA}
\chi(G') =\chi(G) -\chi(F_C'/F_C) 
\geq \chi(G)-\chi(\CO_D\otimes_X G')=
\chi(G)-r. 
\ee

In order to conclude the proof, let
 \[
 0=G_0'\subset G'_i \subset \cdots \subset G'_{h'}=G 
 \]
 be the Harder-Narasimhan filtration of $G'$ with 
 respect to $\omega$-slope stability. Each nontrivial quotient 
$G'_j/G'_{j-1}$, $j=1,\ldots, h'$ must be isomorphic 
to a sheaf of the form $\CO_{C_0}(a'_j)^{s'_j}$, 
$s'_j\geq 1$, such that 
$a_1'>a'_2 > \cdots > a'_{h'}$ by Lemma \ref{sstablesheaves}. 
Inequality \eqref{eq:satsubsheaf} 
implies that $a'_1\leq -2$, therefore $a'_j\leq -2$ for all $j=1,\ldots , h'$. This implies 
\[
\chi(G') \leq -r,
\]
hence inequality \eqref{eq:GprimeineqA} yields 
\[
\chi(G) \leq 0.  
\]
Taking into account the constraints on the Harder-Narasimhan filtration of $G$ in Lemma \ref{smallBlemmaA} 
it follows that 
\[
G \simeq  \CO_{C_0}(-1)^{\oplus r}. 
\]

\hfill $\Box$

\begin{prop}\label{smallBstab}
For fixed K\"ahler class $\omega$, and fixed $(r,n)\in \IZ_{\geq 0}\times \IZ$  the following holds for any 
$b>0$ satisfying the bound \eqref{eq:BboundA}. 

An object $E$ of $\CA^C$ 
 is $\mu_{(\omega,b)}$-stable if and only if 
there is an exact sequence of the form 
\be\label{eq:smallBstabAB}
0\to P_C \to E \to \CO_{C_0}(-1)^{\oplus r}\to 0
\ee
in $\CA^C$ such that:
\begin{itemize}
\item[$(i)$] $P_C=\big(\CO_X\to F_C\big)$ is a stable pair 
on $X$ with $F_C$ scheme theoretically supported on 
$C$.
\item[$(ii)$] There is no linear subspace $0\subset V'\subset \IC^r$ 
such that the restriction 
\[
0\to P_C \to E' \to V'\otimes \CO_{C_0}(-1)\to 0
\]
of the extension  \eqref{eq:smallBstabAB} to $V'\otimes \CO_{C_0}(-1)$ is trivial. Extensions  \eqref{eq:smallBstabAB}
satisfying this property will be called nondegenerate. 
\end{itemize}
\end{prop} 

\begin{rema}\label{pairsextremark}
Since ${\rm Ext}^2_X(\CO_{C_0}(-1)^{\oplus r}, \CO_X)=0$, 
there is an isomorphism $${\rm Ext}^1_X(\CO_{C_0}(-1)^{\oplus r}, P_C) \simeq {\rm Ext}^1_X(\CO_{C_0}(-1)^{\oplus r}, F_C)$$ for any stable pair $P_C= \big(\CO_X \to F_C)$. This observation will be 
used in the proof of Proposition \ref{geombij}. 
\end{rema}

{\it Proof.} 
$(\Rightarrow)$ 
First note that if $r=0$ Lemma \ref{smallBlemmaA} 
shows that $E$ must be isomorphic to a stable pair $P_C$. 
Furthermore, any such stable pair is 
stable for all $b>0$ according to Lemma \ref{Cpairs}. 

Suppose $r>0$. The existence of an extension of the form 
\eqref{eq:smallBstabAB} follows from Lemmas \ref{smallBlemmaA}
and \ref{smallBlemmaB}. 
Nondegeneracy follows easily noting that if the restriction of the extension \eqref{eq:smallBstabAB} to some subsheaf 
$G'=V'\otimes \CO_{C_0}(-1) \subset \CO_{C_0}(-1)^{\oplus r}$ 
is trivial, then there is an epimorphism 
$ G'\hookrightarrow E$. Moreover, there is an exact sequence 
\[ 
0\to P_C\to E/G'\to  V''\otimes \CO_{C_0}(-1)\to 0 
\] 
in $\CA$ where $V''\simeq \IC^r/V'$. This implies that 
that $E/G'$ belongs to $\CA^C_{1/2}$ i.e. 
the epimorphism $G' \hookrightarrow E$ is strict. 
Then $G'\subset E$ violates the stability criterion for $b>0$. 

$(\Leftarrow)$ 
Conversely, suppose an object $E$ of $\CA^C$ 
fits in an extension of the form
 \eqref{eq:smallBstabA} satisfying the nondegeneracy condition. Then it follows easily that $E$ belongs to 
 $\CA^C_{1/2}$. One has to check the stability criterion 
 in Proposition \ref{limitstabcrtA} with respect to strict 
 monomorphisms and epimorphisms in $\CA^C$. Note that 
 property (C.2) implies that all pure dimension one sheaves 
 in $\CA^C$ must be topologically supported on $C_0$. 
 
 Consider first strict epimorphisms $E\twoheadrightarrow G$, 
 with $G$ a pure dimension one sheaf supported on $C_0$. 
 It is straightforward to check that ${\rm Hom}_{\CA}(P_C,G)=0$
 as in the proof of Lemma \ref{Cpairs}. 
 Therefore the exact sequence 
\eqref{eq:smallBstabAB} 
yields an isomorphism 
 \[ 
 {\rm Hom}_\CA(E, G) 
 \simeq {\rm Hom}_X(\CO_{C_0}(-1)^{\oplus r}, G). 
 \] 
Since 
 $\CO_{C_0}(-1)^{\oplus r}$ is $\omega$-slope semistable, 
any quotient $G$ must satisfy
\[
\chi(G) \geq 0 > -2b 
\]
for any $b>0$. 

Next suppose $F\hookrightarrow E$ is a strict monomorphism 
in $\CA_C$ with $F$ a nontrivial pure dimension one 
sheaf on $X$ supported on $C_0$. Let $F'$ 
denote the image of $F$ in $\CO_{C_0}(-1)^{\oplus r}$ and $F''$ the kernel of $F\twoheadrightarrow F'$. 
Then $F''$ must be a subobject of $P_C$ in $\CA^C$, hence 
it must be trivial, as shown in the proof of Lemma \ref{Cpairs}. 
Therefore $F=F'$ must be a subsheaf of 
$\CO_{C_0}(-1)^{\oplus r}$, which is $\omega$-slope semistable
with $\chi(\CO_{C_0}(-1))=0$.
This implies $\chi(F)\leq 0$. 
Since the bound \eqref{eq:BboundA} yields,
\[ 
-{1\over r} < -2b <0,
\]
 $F$ destabilizes $E$ only if $\chi(F)=0$, 
 which implies that $F\simeq V'\otimes \CO_{C_0}(-1)$ for some 
linear subspace $V'\subseteq \IC^r$. However this contradicts the 
nondegeneracy assumption. 

\hfill $\Box$ 

\subsection{Moduli spaces of decorated sheaves}\label{threetwo} 

Consider the moduli problem for data $(V,L,F,s,f)$ 
where $V,L$ are vector spaces of dimension $r,1$, $r\geq 1$, respectively, $F$ is a coherent sheaf on $X$, and 
 \[
 s:L\otimes \CO_X \to F, \qquad f:F\to V \otimes \CO_{C_0}(-1)
 \]
 are morphisms of coherent sheaves satisfying the 
 following conditions:
\begin{itemize} 
\item[$(a)$] $F$ is pure of dimension one with 
$\ch_2(F)=[C]+r[C_0]$, $\chi(F)=n$. 
\item[$(b)$] $f:F\to V \otimes \CO_{C_0}(-1)$ is 
surjective and $\mathrm{Ker}(f)$ is scheme 
theoretically supported on $C$. 
\item[$(c)$] $s:L\otimes \CO_X \to F$ is a nonzero 
morphism. 
\item[$(d)$] The extension 
\[
0\to \mathrm{Ker}(f) \to F\to V \otimes 
\CO_{C_0}(-1)\to 0
\]
satisfies the nondegeneracy condition of Proposition
\ref{smallBstab}. That is there is no proper nontrivial subspace 
$0\subset V'\subset V$ such that the restriction of the above 
extension to $V'\otimes \CO_{C_0}(-1)$ is trivial.

\end{itemize}

Two collections $(V,L,F,s,f)$, $(V',L',F',s',f')$  are
isomorphic if there exist linear isomorphisms 
$V{\buildrel \sim \over \longto}
V'$, $L {\buildrel \sim \over \longto}L'$ and an isomorphism 
of sheaves $F{\buildrel \sim \over \longto} F'$ satisfying 
 the obvious compatibility 
conditions with the data $(s,f)$, $(s',f')$ are satisfied. 
Then it is straightforward to prove 
the automorphism group of any collection 
$(V,L,F,s,f)$ is isomorphic to $\IC^\times$.

Let $T$ be a scheme over 
$\IC$, $X_T = X\times T$ and $\pi_T:X_T\to T$ 
denote the canonical projection. For any closed 
point $t\in T$, let $X_t=X\times \{t\}$ denote the fiber 
of $\pi_T$ over $t$. Let also $C_{0T}\subset X_T$, 
$C_T\subset X_T$ 
denote the closed subschemes $C_0\times T\subset 
X \times T$, $C\times T\subset X\times T$ respectively, 
and $\CO_{C_{0T}}(d)$ denote the pull-back of the sheaf 
$\CO_{C_0}(d)$ to $X_T$, for any $d\in \IZ$. Similar 
notation will be employed for each closed fiber $X_t$, 
$t\in T$. 

A flat family of data $(V,L,F,s,f)$ on $X$ parameterized 
by $T$ is a collection $(V_T,L_T,F_T,s_T,f_T)$ 
where 
\begin{itemize} 
\item $V_T$, $L_T$ are locally free $\CO_T$-modules 
and 
$F_T$ is a coherent $\CO_{X_T}$-module
flat over $T$.
\item $s_T:\pi_T^*V_T\to F_T$ and 
$f_T:F_T\to \pi_T^*L_T\otimes_{X_T}\CO_{C_{0T}}$ are morphisms of 
$\CO_{X_T}$-modules 
\item The restriction of the data $(V_T,L_T,F_T,s_T,f_T)$ to any fiber $X_t$, with 
$t\in T$, a closed point is a collection satisfying 
conditions $(a)$-$(d)$ above.
\end{itemize} 

For any $(r,n)\in \IZ_{\geq 0}\times \IZ$ let 
${\mathcal Q}(X,C,r,n)$ denote the resulting moduli stack 
of  data $(V,L,F,s,f)$ on $X$ satisfying conditions 
$(a)$-$(d)$. Let ${Q}(X,C,r,n)$ be the rigidification 
of ${\mathcal Q}(X,C,r,n)$ obtained by fixing isomorphisms 
$L\simeq \IC$ and $V\simeq \IC^r$. Then the closed points 
of ${Q}(X,C,r,n)$ have trivial stabilizers and 
${\mathcal Q}(X,C,r,n)$ is a $\IC^\times$-gerbe over 
${Q}(X,C,r,n)$. 

The moduli stacks ${\mathcal  Q}(X,C,r,n)$ will be used as an interpolating tool between the nested Hilbert schemes 
$H_p^{[n,k]}(C)$ introduced in Section \ref{refinedOS} 
and stable $C$-framed 
perverse coherent sheaves at small $b>0$. 

\begin{rema}\label{forgetsection}
For future reference, let $\CM(X,C,r,n)$ denote the moduli 
stack of data $(V,f,F)$ satisfying conditions $(a)$, $(b)$, $(d)$ 
above. Obviously, there is a natural 
morphism $\pi: {\mathcal Q}(X,C,r,n)\to\CM(X,C,r,n) $  
forgetting the data $(L,s)$. It is straightforward to check that the stabilizers of 
all closed points of $\CM(X,C,r,n)$ are isomorphic to $\IC^\times$. 
\end{rema}

\subsection{Relation to nested Hilbert schemes}\label{fourthree} 
Suppose $C\subset \IP^2$ is a reduced irreducible 
divisor with one singular point $p$, otherwise smooth. 
For any $l\in \IZ_{\geq 0}$ let $H^l(C)$ denote the 
Hilbert scheme of length $l$ zero dimensional subschemes 
of $C$. Let $H^{[l,r]}(C)\subset H^l(C)\times H^{l+r}(C)$ denote the 
cycle consisting of pairs of ideal sheaves $(J,I)$ such that 
\be\label{eq:nestedA}
m_p J \subseteq I \subseteq J,
\ee
where $m_p\subset \CO_{C,p}$ is the maximal ideal in the local 
ring at $p$. 

The main observation is that the nested Hilbert schemes 
$H^{[l,r]}(C)$, equipped with an appropriate scheme 
structure, are isomorphic to relative 
 {\it Quot} schemes over $H^l(C)$.   
Let ${\mathcal J}$ denote the universal ideal sheaf on 
$H^l(C)\times C$ and ${\mathcal J}_p$ its restriction to the 
closed subscheme $H^l(C)\times \{p\}$. 
Let ${Q}^{[l,r]}(C)$ the relative {\it Quot}-scheme 
parametrizing rank $r$ locally free quotients of ${\mathcal J}_p$
over $H^l(C)$. Standard results on {\it Quot}-schemes show that 
${Q}^{[l,r]}(C)$ is a quasi-projective scheme over 
$H^l(C)$. Note that a closed point of 
${Q}^{[l,r]}(C)$ 
over a closed point $[J]\in H^l(C)$ is a pair $(V,\xi)$ 
where $V$ is a $r$-dimensional vector space over $\IC$ and 
$\xi:J\otimes\CO_{p}\twoheadrightarrow V$ is a surjective map of complex 
vector spaces. 
In particular the fiber of ${Q}^{[l,r]}(C)$ 
is empty if $J$ has less than $r$ generators at $p$. 
Let $$I = \mathrm{Ker}(J\twoheadrightarrow J\otimes\CO_{C,p} 
{\buildrel \xi\over \longto} V).$$ 
Then it is straightforward to check that 
$(J,I)$ is a pair of  ideal sheaves on $C$ satisfying conditions
 \eqref{eq:nestedA} 
at $p$. 
 Note that the resulting  scheme structure on $H^{[l,r]}(C)$ 
may be different from the reduced induced scheme structure. 

The main result of this section is:
\begin{prop}\label{projhilbert}
For any $(r,n)\in \IZ_{\geq 1}\times \IZ$, $n\geq \chi(\CO_C)$,  there is an isomorphism 
\be\label{eq:nestedisom} 
q:{ Q}(X,C,r,n) 
{\buildrel \sim \over \longto} 
{Q}^{[l,r]}(C)
\ee
over $H^l(C)$, where $l=n-\chi(\CO_C)$. 
\end{prop}

The first step in the proof 
of Proposition \ref{projhilbert} 
is the observation that the moduli stack ${\mathcal Q}(X,C,r,n)$ 
admits a dual formulation which makes the connection with the 
Hilbert scheme of $C$ manifest. Let $J\subset \CO_C$ be the ideal sheaf of a zero dimensional subscheme of $C$ and consider an 
exact sequence of $\CO_X$-modules
\be\label{eq:idealexseqB} 
0\to V\otimes \CO_{C_0}(-1) \to F \to \CJ \to 0,
\ee
with $V$ a finite dimensional vector space. The extension \eqref{eq:idealexseqB} is called {\it nondegenerate} if for any nontrivial 
quotient $V\twoheadrightarrow V'$, the corresponding extension 
class $e$ is not in the kernel of the natural map 
\[ 
{\rm Ext}^1_X(\CJ, V\otimes \CO_{C_0}(-1)) 
\to 
{\rm Ext}^1_X(\CJ, V'\otimes \CO_{C_0}(-1)).
\]
Now let $\pi:{\mathcal Q}^*(X,C,r,l)\to H^l(C)$ 
be a 
moduli stack over $H^l(C)$ defined as follows. 
For any scheme $\tau:T\to H^l(C)$ let $\CJ_T$ be the 
flat family of ideal sheaves on $C$ obtained by pull-back. 
The objects of ${\mathcal Q}^*(X,C,r,l)$
over $
\tau :T \to H^l(C)$ are collections $(\CV_T,F_T,f_T,g_T)$ 
where $V_T$ is a locally free $\CO_T$-module, $F_T$ is a 
flat family of pure dimension one sheaves on $X$, and 
$g_T:\pi_T^* \CV_T\otimes_{X_T} \CO_{C_{0T}}(-1)\to F_T$, 
$h_T: F_T \to \CJ_T$ are morphisms of $\CO_{X_T}$-modules 
such that 
\begin{itemize} 
\item[$(a^*)$] For any closed point $t\in T$ there is an exact sequence 
of $\CO_{X_t}$-modules
\be\label{eq:idealexseqA} 
0\to V_t \otimes_{X_t} \CO_{C_{0t}}(-1) 
{\buildrel g_t\over\longto} F_t  {\buildrel h_t\over\longto} 
\CJ_t\to 0
\ee
\item[$(b^*)$] The extension \eqref{eq:idealexseqA} is nondegenerate
\end{itemize}
Isomorphisms are defined naturally. Then the following holds.

\begin{lemm}\label{dualstacks} 
For any $(r,n)\in\IZ_{\geq 0}\times \IZ$, 
$n\geq \chi(\CO_C)$, there is an isomorphism 
${j}:{\mathcal Q}(X,C,r,n) {\buildrel \sim 
\over \longto} {\mathcal Q}^*(X,C,r,l)$, $l=n-\chi(\CO_C)$.
\end{lemm}

{\it Proof.} Given any collection $(V,L,F,s,f)$ satisfying conditions 
$(a)$-$(d)$, let $G=\mathrm{Ker}(f)$. 
Note that the section $s:\CO_X\to F$ must factor through 
$s_C:\CO_C\to G$ since $H^0(\CO_{C_0}(-1))=0$ 
and $G$ is scheme theoretically supported on $C$. 
According to \cite[Prop. B8]{stabpairs-III}, 
the moduli space of pairs $(G,s_C)$ is isomorphic 
to the Hilbert scheme $H^l(C)$. The isomorphism
is obtained by taking the derived dual $G^\vee=R{\rm Hom}_C(G,\CO_C)$, which is an ideal sheaf on $C$. 

The isomorphism $j$ will be first constructed 
on closed points. 
Note that taking derived duals on $X$ 
one obtains an exact sequence 
\be\label{eq:dualexseq} 
0\to V^\vee\otimes \CO_{C_0}(-1) {\buildrel f^\vee\over \longto} 
{\mathcal Ext}^2(F,\CO_X) \to {\mathcal Ext}^2(G,\CO_X) \to 0.
\ee
The duality theorem for the closed embedding $\iota:C\hookrightarrow 
X$ yields an isomorphism 
\[
R\iota_*R{\mathcal Hom}_C(G,\omega_C)[-2]\simeq R{\mathcal Hom}_X(R\iota_*G, \CO_X).
\]
Note also that $\omega_C\simeq \CO_X((k-3)H)$, where $k\in \IZ_{>0}$ is the degree of $C\subset D$. Therefore there is an isomorphism of 
$\CO_X$-modules 
\[ 
\iota_*G^\vee \simeq {\mathcal Ext}^2_X(G,\CO_X)\otimes_X \CO_{X}((3-k)H),
\]
where $G^\vee$ denotes the derived dual on $C$. 
Moreover it is straightforward  to check that the extension 
\[
0\to G \to F {\buildrel f\over \longto} V\otimes \CO_{C_0}(-1)\to 0
\] 
is nondegenerate if and only if the dual \eqref{eq:dualexseq} is 
nondegenerate. 

In conclusion the functor ${j}$ has been constructed on 
closed points.
The construction in families is analogous, using 
\cite[Prop. B.8]{stabpairs-III}.

\hfill $\Box$ 

In order to conclude the proof of Proposition \ref{projhilbert}, 
recall that according to Lemma \ref{extlemma} 
there is an isomorphism 
\[
\varphi_1:{\rm Ext}^1_X(J,V\otimes \CO_{C_0}(-1)) 
{\buildrel \sim \over\longto} 
\mathrm{Hom}_D(J, V\otimes \CO_{p}).
\]
Moreover Corollary \ref{nondegext} 
shows that for  
given a morphism $\psi: J\to V\otimes \CO_{p}$, 
the extension 
\[ 
0\to V\otimes\CO_{C_0}(-1)\to F_{\varphi_1^{-1}(\psi)} 
\to \CJ\to 0
\]
is nondegenerate if and only if $\psi$ is surjective. 
Now note that there is an isomorphism 
\[
{ \rm Hom}_C(J, V\otimes\CO_p) \simeq {\rm Hom}_C 
(J\otimes_C \CO_p, V\otimes\CO_p), \qquad 
\psi \mapsto {\bar \psi},
\]
such that $\psi$ is surjective if and only if ${\bar \psi}$ is surjective. 
Then Proposition \ref{projhilbert} follows from 
Lemma \ref{dualstacks} by a straightforward comparison of flat families. 

\hfill $\Box$

\begin{rema}\label{forgetsectionB}
Note that Proposition \ref{projhilbert} implies that the stack 
${\mathcal Q}(X,C,r,n)$ is a $\IC^\times$-gerbe over the {\it Quot}
schemes $Q^{[l,r]}(C)$. A similar result holds for the moduli 
stacks $\CM(X,C,r,n)$ of decorated objects 
satisfying conditions $(a)$, $(b)$, $(d)$ introduced in 
Remark \ref{forgetsection}. 
Let $\CM^{[l,r]}(C)$ be the moduli stack of pairs $(J,\psi)$ 
where $J$ is any abstract sheaf which admits an isomorphism 
to a length $l$ ideal sheaf on $C$, 
and $\psi:J\to 
\CO_p^{\oplus r}$ a surjective morphism. Two such pairs are 
isomorphic if there is an isomorphism of sheaves 
$\xi:J\to J'$ such that 
$\psi'\circ \xi =\psi$. By analogy with the moduli spaces 
${Q}^{[l,r]}(C)$,  $M^{[l,r]}(C)$ are naturally identified with relative {\it Quot} schemes over the compactified 
Jacobian of $C$ of degree $l=n-\chi(\CO_C)$.  
By analogy with Proposition \ref{projhilbert}
 the stacks $\CM(X,C,r,n)$ 
are $\IC^\times$-gerbes over 
the moduli spaces $M^{[l,r]}(C)$. Moreover there is an 
obvious forgetful morphism $\pi:Q^{[l,r]}(C)\to M^{[l,r]}(C)$ 
determined by the natural morphism from the Hilbert 
scheme to the compactified Jacobian. 
\end{rema}

Let $\CN(D,k,r,n)$ be the moduli stack of pairs $(J,\psi)$ 
where $J$ is a rank one torsion free sheaf on a degree $k$ reduced irreducible divisor on $D$ and $\psi:J\to \CO_p^{\oplus r}$ a surjective morphism. Obviously there 
is a natural projection $\CN(D,k,r,n)\to \CU$ to an open 
subset of the linear system $\IP(H^0(I_p(kH)))$ 
where $I_p$ is the maximal ideal sheaf of $p$. 
 ${M}^{[l,r]}(C)$ is the fiber of this projection over the point $[C]\in \CU$. 
Since any $\CO_D$-module $J$ as above is automatically 
slope and Gieseker stable on $D$, one can easily check 
that such a pair $(J,\psi)$ is $\delta$-stable in the sense 
of \cite{framedmoduli} for sufficiently small $\delta>0$. 
 Then the results of \cite{framedmoduli} imply that 
 $\CN(D,k,r,n)$ is a quasi-projective moduli scheme. 
 
\begin{lemm}\label{smoothmoduli}
If $r<3k$, $\CN(D,k,r,n)$ is smooth. 
\end{lemm}

{\it Proof.} According to \cite{framedmoduli}, 
the  deformation theory of a pair $(J,\psi)$ 
is determined by the extension groups ${\rm Ext}^k_D(J, 
{\sf C}(\psi)[-1])$, $k=1,2$, 
where ${\sf C}(\psi)$ is the cone of $\psi$. 
In order to prove smoothness it suffices to show that 
${\rm Ext}^2_D(J, 
{\sf C}(\psi)[-1])=0$. Since $\psi$ is surjective, ${\sf C}(\psi)[-1]$ 
is quasi-isomorphic to the kernel $I={\rm Ker}(\psi)$.  
Let $T=\CO_p^{\oplus r}$. Then there is an exact sequence 
\[
\cdots \longto {\rm Ext}^1_D(J,J) \longto 
{\rm Ext}^1_D(J,T)\longto {\rm Ext}^2_D(J,I)\longto 
\cdots 
\]
Since $J$ is a stable $\CO_D$-module and $D\simeq \IP^2$ 
is Fano, ${\rm Ext}^2(J,J)=0$. Therefore it suffices to prove that the natural map 
\[
\psi_*:{\rm Ext}^1_D(J,J) \longto 
{\rm Ext}^1_D(J,T)
\]
is surjective. 
Using Serre duality on $D$, ${\rm Ext}^2_D(J,I) \simeq 
{\rm Ext}^0(I,J\otimes_D \omega_D)^\vee$. 
Since $I,J$ are both slope stable on $D$
with $\ch_1(I)=\ch_1(J)=kH$,  this group is trivial  if 
$\chi(I)>\chi(J\otimes_D \omega_D)$. 
However 
\[
\chi(I)=\chi(J)-\chi(T) =\chi(J)-r, \qquad 
\chi(J\otimes_D \omega_D)= \chi(J)-3k.
\]
Therefore the conclusion follows if $r<3k$. 

\hfill $\Box$

\begin{rema}\label{smoothbound} 
Note that $$
{\rm Hom}_C(J,\CO_p^{\oplus r})\simeq 
{\rm Hom}(J\otimes_C \CO_p,\IC^r).$$
Therefore the existence of a surjective morphism $\psi:J\to \CO_p^{\oplus r}$ requires $r$ to be smaller than the minimal number of generators of $J$ at $p$, $m(J)={\rm dim}(J\otimes_D \CO_p)$. However this  number is 
bounded above\footnote{We thank Vivek Shende for pointing out 
this bound.}  by the degree $k$ of $C$, therefore the condition 
$r<3k$ is always satisfied. 
\end{rema}

\subsection{Relation to small $b$ moduli spaces}\label{fourfour} 
Let ${\mathcal P}_{0+}(X,C,r,n)$ denote the moduli 
stack of $\mu_{(\omega,b)}$-slope stable objects of $\CA^C$, where 
$b$ satisfies the bound \eqref{eq:BboundA}. 
By analogy with \cite{limit,generating}, 
${\mathcal P}_{0+}(X,C,r,n)$ is an algebraic stack of finite type 
over $\IC$, and all stabilizers of closed points are isomorphic to 
$\IC^\times$. Recall that 
an object of ${\mathcal P}_{0+}(X,C,r,n)$ is a perfect
complex $E_T$ on $X_T$ such that its restriction 
$L\iota_t^*E_T$ is a  $\mu_{(\omega,b)}$-slope stable
object of the category $\CA^C$ associated to the fiber $X_t$ 
for any closed point $\iota_t:\{t\}
\hookrightarrow  T$. In this subsection
$b>0$ will be a small stability parameter 
of type $(r,n)$ satisfying the bound \eqref{eq:BboundA}.

Any flat family $(V_T,L_T,F_T,s_T,f_T)$ over $T$, 
determines a complex 
\[
E_T = (\pi_T^*L_T {\buildrel s_T\over \longto } F_T)
\]
on $X_T$. 
Since $F_T$ is flat over $T$, and $X_T$ is smooth 
projective over $T$, $E_T$ is perfect. 
Moreover, the derived restriction of $E_T$ to any 
closed fiber $X_t$ is simply obtained by restricting 
the terms of $E_T$ to $X_t$. It follows that the complex
$L\iota_t^*E_T$ satisfies the conditions of Proposition 
\ref{smallBstab}. 
Therefore this construction defines a morphism of stacks 
\[
{f}: {\mathcal Q}(X,C,r,n) \to 
{\mathcal P}_{0+}(X,C,r,n).
\]
\begin{prop}\label{geombij} 
The morphism ${f}$ is geometrically bijective i.e. 
it yields an equivalence 
\[
{f}(\IC) : {\mathcal Q}(X,C,r,n)(\IC) 
{\buildrel \sim \over \longto} 
{\mathcal P}_{0+}(X,C,r,n)(\IC).
\]
of groupoids of $\IC$-valued points. 
\end{prop}

{\it Proof.} 
Proposition \ref{smallBstab} and Remark \ref{pairsextremark}
imply that any 
object of ${\mathcal P}_{0+}(X,C,r,n)(\IC)$ 
is quasi-isomorphic to an object in the image of ${f}(\IC)$. 
One has to prove that if two data 
$(V,L,F,s,f)$ and $(V',L',F',s',f')$ are mapped to 
quasi-isomorphic complexes $E,E'$ then they must be 
isomorphic. 
This can be proven by analogy with 
\cite[Prop. 1.21]{stabpairs-I}. 
Given an object $E$ of $\CA^C$ satisfying the conditions of 
Proposition \ref{smallBstab} there is an exact triangle 
\[
\CO_X {\buildrel s\over \longto} F \to E 
\] 
in $D^b(X)$ 
where $F$ is a nondegenerate extension 
\be\label{eq:Fext}
0\to G \to F \to V\otimes \CO_{C_0}(-1)\to 0.
\ee
This
yields a long exact sequence 
\[
\cdots \to 
\mathrm{Hom}(F,\CO_X) \to \mathrm{Hom}(\CO_X,\CO_X) 
\to \mathrm{Hom}(E, \CO_X[1]) \to \mathrm{Hom}(F,\CO_X[1])\to \cdots 
\]
The first term is obviously trivial since $F$ is torsion and Serre 
duality implies that 
\[ 
\mathrm{Hom}(F,\CO_X[1])\simeq H^2(F)^\vee=0
\]
since $F$ is supported in dimension at most one. Therefore 
\[ 
\mathrm{Hom}(E, \CO_X[1])\simeq 
\mathrm{Hom}(\CO_X,\CO_X)\simeq \IC,
\]
which implies that there is a unique morphism $E\to \CO_X[1]$ up 
to multiplication by nonzero complex numbers. 
Then 
$F$ is quasi-isomorphic to the mapping cone of the morphism 
$E\to \CO_X[1]$, and the section $s$ is recovered from the 
induced map $\CO_X\to F$ as in \cite[Prop. 1.21]{stabpairs-I}. 
In order to finish the proof, note that given two extensions 
$F,F'$ of the form \eqref{eq:Fext}, an isomorphism of sheaves 
$F{\buildrel \sim \over \longto} F'$ 
induces isomorphisms $G{\buildrel \sim \over \longto} G'$, 
respectively $V {\buildrel \sim \over \longto} V'$ using the snake 
lemma. 
Therefore the data 
$(V,L,F,s,f)$ can be recovered up to isomorphism 
from the complex $E$.

\hfill $\Box$

In conclusion,  note that Propositions \ref{projhilbert} and \ref{geombij}
imply Theorem \ref{smallBidentity} as follows. 

{\it Proof of Theorem \ref{smallBidentity}.} 
Summarizing the results of Propositions \ref{projhilbert} and \ref{geombij}, 
the moduli stack $\calP_{0+}(X,C,r,n)$ of $C$-framed perverse coherent 
sheaves in the small $b>0$ chamber is geometrically bijective with the moduli stack of decorated 
sheaves ${\CQ}(X,C,r,n)$. The latter is in turn an $\IC^\times$-gerbe 
over the relative {\it Quot} scheme $Q^{[l,r]}(C)$, $l=n-\chi(\CO_C)$, 
 introduced 
above Proposition \ref{projhilbert}. 
Then, using \cite[Sect. 3.5, Eqn. (43)]{genDTI}, and \cite[Thm.  3.16]{genDTI},
the topological Euler character 
invariants $P^{top}_{0+}(X,C,r,n)$ are given by 
\[ 
P^{top}_{0+}(X,C,r,n)= \chi(Q^{[l,r]}(C)).
\] 
Next let $\pi:Q^{[l,r]}(C)\to H^l(C)$ be the canonical projection 
to the Hilbert scheme. Let $\CJ$ denote the universal ideal sheaf 
on $H^{l}(C)\times C$. 
By construction, there is a universal quotient 
$$(\pi\times 1_C)^*\CJ\bigg|_{Q^{[l,r]}(C)\times \{p\}}
 \twoheadrightarrow \CV,$$ 
where $\CV$ is a rank $r$ locally free sheaf on $Q^{[l,r]}\times \{p\}$. 
Therefore there is a surjective morphism
\be\label{eq:univsurjmap}
(\pi\times 1_C)^*\CJ\twoheadrightarrow \iota_{p*}\CV
\ee 
 of sheaves on $Q^{[l,r]}(C)\times C$, where $\iota_{p}: 
 Q^{[l,r]}\times \{p\}\hookrightarrow Q^{[l,r]}(C)\times C$
 denotes the canonical closed embedding. Moreover, both 
 sheaves in \eqref{eq:univsurjmap} are flat over 
 $Q^{[l,r]}(C)$.  Therefore the kernel $\CI$ of the morphism 
 \eqref{eq:univsurjmap} is also flat over $Q^{[l,r]}(C)$, and 
 the long exact sequence 
 \[
 0\to \CI \to (\pi\times 1_C)^*\CJ\to\iota_{p*}\CV\to 0
 \]
 restricts to an exact sequence on each curve 
 $\{q\}\times C$, with $q$ a closed point of $Q^{[l,r]}(C)$. 
 This implies that $\CI$ is a flat family of length 
 $l+r$ ideal sheaves on $C$ parameterized by $Q^{[l,r]}(C)$.
 Hence it determines a morphism $\tau:Q^{[l,r]}(C) \to 
 H^{l+r}(C)$. Moreover, by construction, the 
 image of the morphism $\pi\times \tau: Q^{[l,r]}(C)\to 
 H^l(C) \times H^{l+r}(C)$, equipped with the reduced induced 
 scheme structure is the nested Hilbert scheme $H^{[l,r]}(C)$. 
 
In conclusion, the generating function 
\[
Z_{0+}^{top}(X,C,u,T)= \sum_{n\in \IZ}\sum_{r\geq 0} 
T^ru^n P_{0+}^{top}(X,C,r,n)
\]
is equal to 
\[ 
u^{\chi(\CO_C)}\sum_{l\geq 0}\sum{r\geq 0} 
T^r u^l \chi(H^{[l,r]}(C)).
\] 
Now note that the Hilbert scheme $H^l(C)$ 
admits a stratification  
\[
\cdots \subset H^l_{\geq s}(C,p) \subset \cdots \subset 
H^l_{\geq 1}(C,p) =H_p^l(C) 
\]
where $H^l_{\geq s}(C,p)$, $s\geq 1$ denotes the closed subscheme parameterizing ideal 
sheaves $I\subset \CO_{C,p}$ with at least $s$ generators at $p$.
Let $\CS^l_s(C,p)= H^l_{\geq s}(C,p)\setminus 
H^l_{\geq s+1}(C,p)$ denote the locally closed strata. 
Then the natural projection morphism 
$H^{[l,r]}(C) \to H^l(C)$ is a smooth 
$Gr(s,r)$-fibration over the locally closed stratum 
$\CS^l_s(C,p)$, where 
$Gr(s,r)$ is the Grassmannian of $r$-dimensional quotients 
of $\IC^m$. In particular the fibers are empty over strata 
with $s<r$. This implies that 
\[
\bal
Z_{0+}^{top}(X,C,u,T) & = u^{\chi(\CO_C)}
\sum_{l\geq 0}\sum_{r\geq 0} 
T^r u^l \sum_{s\geq r} \chi(Gr(s,r)) \chi(\CS^l_s(C,p)) \\
& =  u^{\chi(\CO_C)}
\sum_{l\geq 0}\sum_{r\geq 0} 
T^r u^l \sum_{s\geq r} \binom{s}{r} \chi(\CS^l_s(C,p)) \\
& = u^{\chi(\CO_C)}
\sum_{l\geq 0} u^l\int_{H^l(C)} (1+T)^{m}d\chi
\eal
\]
where $m:H^l(C)\to \IZ$ is the constructible function which takes 
value $s$ on the stratum $\CS^l_s(C,p)$. 
This implies equation \eqref{eq:smallBinvariants} in Theorem \ref{smallBidentity}
making the substitutions $T=-a^2$, $u=q^2$. 

\hfill $\Box$

 \section{Motivic invariants at small $b$}\label{motivicsect}

Composing the 
morphism ${f}:{\mathcal Q}(X,C,r,n)\to 
{\mathcal P}_{0+}(X,C,r,n)$ with the natural morphism 
${p}:{\mathcal P}_{0+}(X,C,r,n)\to 
{Ob}(\CA)$ one obtains 
a stack function 
\[
{q} : {\mathcal Q}(X,C,r,n)\to 
{ Ob}(\CA)
\]
which determines an element of the motivic Hall algebra $H({\CA})$. 
The construction of $H(\CA)$ is briefly reviewed in Appendix 
\ref{motivicHallidentities}, as background material for the proof 
of Theorem \ref{topCframedinv}. 
The motivic Donaldson-Thomas theory of \cite{wallcrossing}
assigns to any stack function an invariant with values in a certain ring of motives, as reviewed below. Note that the formalism of \cite{genDTI} 
does not apply to abelian category of perverse coherent sheaves 
since no rigorous construction of holomorphic Chern-Simons 
functionals for such objects has been carried out yet in the literature. 
Therefore this section will rely on the conjectural construction 
of \cite{wallcrossing} employing motivic vanishing cycles 
for formal functions. 

The goal of this section is to compare the resulting motivic 
invariants with 
the motivic Hilbert scheme series defined in equation 
\eqref{eq:compactHilbmotivic}, Section \ref{onethree}. 
The two generating functions will be shown to agree 
subject to a conjectural comparison formula between 
the motivic weights of moduli stacks of stable pairs and sheaves. 
This is a natural motivic counterpart of previous 
results for numerical invariants \cite{stabpairs-III}, which 
will be proven here only for sheaves of sufficiently high 
degree. The general case is an open conjecture. 

The required elements 
in the construction of motivic Donaldson-Thomas invariants after \cite{wallcrossing} are the integral identity 
conjectured in \cite[Conj. 4, Sect. 4.4]{wallcrossing} 
and the orientation data \cite[Sect. 5.2]{wallcrossing}. 
The integral identity has been recently proven in 
\cite{integral_identity}, therefore  \cite[Thm. 8, Sect 6.3]{wallcrossing} yields a motivic integration map as soon 
as the derived category is equipped with orientation data. 
This will be assumed without proof in this paper. 
Moreover, 
explicit computations of motivic weights for sheaves 
will be carried out in sections \ref{fivethree}, 
\ref{fivefour} by reduction to a triangulated subcategory of 
quiver representations. In that context it will be further assumed that the 
orientation data on the ambient category $D^b(X)$ agrees with 
orientation data on the derived category of quiver representations 
constructed in \cite{COHA,orientation}.

\subsection{Review of motivic Donaldson-Thomas invariants}
Recall that if $X$ is a compact complex Calabi-Yau $3$-fold then the derived category of coherent sheaves $D^b(X)$ carries a structure of $3$-dimensional Calabi-Yau category ($3CY$ category for short), see \cite{wallcrossing} for  details. In particular  we endow it with  a  cyclic $A_\infty$-structure,
for example by fixing a Calabi-Yau metric on $X$.
Then according to \cite[Sect 3]{wallcrossing} there is a 
formal potential function $W_E$ on the vector space 
${\rm Hom}^1(E,E)$ for any object $E$ of $D^b(X)$. Replacing the category by its minimal model we can treat $W_E$ as a formal function on ${\rm Ext}^1(E,E)$.
Moreover, \cite[Prop. 7, Sect 3.3]{wallcrossing} shows that there is a 
direct sum decomposition of formal functions 
\[ 
W_E = W_E^{\min} \oplus Q_E\oplus N_E
\] 
where 
\begin{itemize} 
\item $W_E^{min}$ is the potential of the minimal $A_\infty$-model. 
\item 
$Q_E$ is a quadratic function on the quotient 
\[ 
{\rm Ext}^1(E,E)/{\mathrm Ker}(m_1:{\rm Ext}^1(E,E)\to 
{\rm Ext}^2(E,E)).
\]
\item $N_E$ is the zero function on the image of $m_1:{\rm Ext}^0(E,E)\to 
{\rm Ext}^1(E,E)$.
\end{itemize}
The argument sketched in \cite{wallcrossing} shows that $D^b(X)$ is ind-constructible locally regular category. This means that $W_E$ can be treated as a regular function along the stack of objects of $D^b(X)$ and formal one in the ``transversal" direction  ${\rm Ext}^1(E,E)$. 
According to \cite[Sect 4.3]{wallcrossing}
$W_E$ determines a motivic Milnor fiber 
$MF_0(W_E)$ at $0$, with values in an appropriate ring of motives. 
Note that the motivic Thom-Sebastiani theorem implies that 
\[ 
(1-MF_0(W_E)) = (1-MF_0(W_E^{min}))(1-MF_0(Q_E)).
\] 
Furthermore, suppose that the category $D^b(X)$ is endowed with
 orientation data and a polarization such that the construction of 
 motivic Donaldson-Thomas series in \cite[Sect 6]{wallcrossing}
 applies. In particular, to each object $E$ of $\CA$ one assigns a 
 motivic weight 
 \be\label{eq:motivicweightA} 
 w_E = \IL^{(E,E)_{\leq 1}/2} (1-MF_0(W_E)) \IL^{-\rk(Q_E)/2}.
  \ee 
 Following the conventions of \cite{wallcrossing}, given any 
 two objects $E_1,E_2$, set
 \[
 (E_1,E_2)_j = {\rm dim}({\rm Ext}^j(E_1,E_2)),\qquad 
 (E_1,E_2)_{\leq j} = \sum_{i\leq j} (-1)^i {\rm dim}({\rm Ext}^i(E_1,E_2))
 \]
  for any $j\in \IZ$. 
  
Since ${Ob}(\CA)\subset Ob(D^b(X))$ we can treat constructible families over ${Ob}(\CA)$ as constructible families over $Ob(D^b(X))$. This gives a homomorphism at the level of stack functions and motivic Hall algebras (since $\CA$ is a heart of $t$-structure there are no negative 
${\rm Ext}^i$ between its objects). As a result, we can apply the formalism of \cite{wallcrossing} to the category of perverse coherent sheaves.
The motivic invariant for a stack function $[{\mathcal \CX}\to 
  {Ob}(\CA)]$ is defined by integration of motivic weights, 
  which is defined in \cite[Sect 4.4]{wallcrossing}.
The  result is encoded in the morphism $\Phi$ constructed 
in \cite[Thm. 8, Sect 6.3]{wallcrossing} from the motivic Hall 
algebra $H(\CA)$ to the quantum torus. Let $\Gamma$ 
denotes the intersection of the image of the Chern character $ch: K_0(D^b(X))\to H^{ev}(X,{\mathbb Q})$  with $H^{ev}(X,{\bf Z})$ (instead of $\Gamma$ one can take the quotient of $K_0(\CA)$ by the subgroup generated by the numerical equivalence). In particular the lattice $\Gamma$ is equipped 
with a 
natural nondegenerate 
 antisymmetric pairing $\langle\ ,\ \rangle$.
The quantum torus is the associative algebra $\CR$ over an appropriate motivic ring described in \cite{wallcrossing} spanned by the symbols 
${\hat e}_\gamma$, $\gamma \in \Gamma$ over the ring of motivic 
weights, where 
\[
{\hat e}_{\gamma_1} {\hat e}_{\gamma_2} = \IL^{\langle\gamma_1,\gamma_2\rangle/2} {\hat e}_{\gamma_1+\gamma_2}. 
\]
Here $\IL$ denotes the motive of the affine line.

Then the integration map $\Phi:H(\CA)\to \CR$ assigns to a 
stack function $[{\mathcal Y}{\buildrel\pi\over \longto} {Ob}(\CA)]$ which factors through the substack 
${Ob}_\gamma(\CA)$, the element 
\[ 
\int_{{\mathcal Y}} w_{\pi(y)} {\hat e}_{\gamma}.
\]

\subsection{Motivic weights  at small $b$}\label{fivetwo} 
The next goal is to evaluate the integration map $\Phi$ on 
the stack function 
\[
q={p}\circ {f} :
{\mathcal Q}(X,C,r,n) \hookrightarrow 
{Ob}(\CA)
\]
determined by Proposition \ref{geombij}.
 As observed in Remark \ref{forgetsection}, 
 there is a natural forgetful morphism 
 $\pi:  {\mathcal Q}(X,C,r,n) \to {\mathcal M}(X,C,r,n)$ 
 to the stack of nondegenerate extensions $(V,F,f)$. 
 The fiber of $\pi$ over a closed point $(V,F,f)$ is isomorphic 
 to the projective space $\IP H^0(F)$.  
 Note that there is also a natural obvious  morphism
 ${m}: {\mathcal M}(X,C,r,n)\to 
 {   Ob}(\CA)$ sending the sheaf $F$ to itself. 
 Then the integration of motivic weights 
 may be carried out in two stages, first along 
 the fibers of $\pi$, and then then on 
 ${\mathcal M}(X,C,r,n)$. The first step will be  
 considered below, while the second one will be postponed 
 for Section \ref{fivethree}. 

Note that there is a one-to-one correspondence between 
nonzero sections in $H^0(F)$ and 
nontrivial extensions
\[
0\to F \to E \to \CO_X[1]\to 0
\]
in $\CA$.
Set $E_1=F$, 
$E_2=\CO_X[1]$ and $E_0=E_1\oplus E_2$. 
According to \cite[Thm. 8, Sect 6.3]{wallcrossing}, there is a 
relation 
\[
\Phi(E_1\cdot E_2) = \Phi(E_1)\Phi(E_2)
\]
in the motivic quantum torus. As shown in {\it Step 1} 
in the proof of loc. cit., this identity is equivalent to 
\[
\bal 
\int_{\alpha \in \mathrm{Ext}^1(E_2,E_1)} 
w_{E_\alpha} = \IL^{[(E_0,E_0)_{\leq 1} -(E_1,E_1)_{\leq 1} 
-(E_2,E_2)_{\leq 1}]/2} \IL^{(E_2,E_1)_1} 
w_{E_1}w_{E_2},
\eal
\]
where $E_\alpha = {\sf Cone}(\alpha)$ for any extension 
class $\alpha \in {\rm Hom}(E_2[-1],E_1)= {\rm Hom}(\CO_X,F)$. 
In particular for $\alpha=0$ one obtains the trivial extension, 
$E_1\oplus E_2$, which has been denoted by $E_0$ in the above 
equation. 

Note that $(E_2,E_2)_{\leq 1}=1$ and 
$w_{E_2}=\IL^{1/2}$ since $\CO_X[1]$ is a 
spherical object. 
Since the fibers of $\pi$ parametrize nonzero extensions, integration along the fiber yields 
\be\label{eq:fiberint}
\bal 
\int_{\alpha \in \mathrm{Ext}^1(E_2,E_1)} 
w_{E_\alpha} - w_{E_0} = 
\IL^{[(E_0,E_0)_{\leq 1} -(E_1,E_1)_{\leq 1} 
]/2}\IL^{(E_2,E_1)_1} 
w_{E_1} - w_{E_0}.
\eal 
\ee

Now recall that the sheaves $F$ are nondegenerate 
extensions of the form 
\[
0\to G \to F \to V \otimes \CO_{C_0}(-1)\to 0.
\]
where $G$ is a rank one torsion free sheaf on a reduced 
irreducible divisor $C\subset D$. 
The motivic weight $w_{E_0}$ will be computed 
below in those cases where $H^1(G)=0$.   
Let $p:{\rm Ext}^1_X(E_0,E_0) \to {\rm Ext}_X^1(F,F)$
be the canonical projection. 
Let $W_{E_0}$, $W_{F}$ be the formal potential functions 
on ${\rm Ext}^1(E_0,E_0)$, ${\rm Ext}^1(F,F)$, respectively. 
\begin{lemm}\label{stepone} 
If $H^1(G)=0$, 
$W_{E_0}=p^*W_{F}$.
\end{lemm} 

{\it Proof.} The proof will be based on 
\cite[Thm. 9, Sect 8]{wallcrossing}. Set $E_1=F$, 
$E_2=\CO_X[1]$. One then  has to check that 
the following conditions are satisfied
\begin{itemize}
\item[$(a)$] ${\rm Ext}^0(E_i,E_i)=\IC\, 1_E$, $i=1,2$,
\item[$(b)$] ${\rm Ext}^0(E_i,E_j)=0$, $i,j=1,2$, $i\neq j$, 
\item[$(c)$] ${\rm Ext}^{<0}(E_i,E_j) =0$, $i,j=1,2$.
\end{itemize}
Condition $(a)$ is satisfied because the automorphism group of a nondegenerate extension $F$ is $\IC^\times 1_F$.
Condition $(b)$ is satisfied since 
\[
{\rm Ext}^0(F,\CO_X[1])\simeq {\rm Ext}^2(\CO_X,F)^\vee
\simeq H^2(F)^\vee=0.
\]
Condition $(c)$ is also satisfied since 
\[
{\rm Ext}^{-1}(F,\CO_X[1]) = {\rm Ext}^0(F,\CO_X)=0.
\]
Moreover, note that
\[
{\rm Ext}^k(\CO_X, \CO_{C_0}(-1))\simeq H^k(\CO_{C_0}(-1))=0
\]
for all $k\in\IZ$, hence
\be\label{eq:zeroextB}
{\rm Ext}^k(\CO_X,F) \simeq {\rm Ext}^k(\CO_X,G), 
\ee
for all $k\in \IZ$. Using Serre duality,
\[
{\rm Ext}^1(G,\CO_X[1]) \simeq {\rm Ext}^1(\CO_X,G)^\vee \simeq H^1(G)^\vee=0.
\]
Therefore the ${\rm Ext}^1$ quiver of the collection of objects 
$\{E_1,E_2\}$ is of the form  
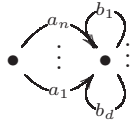
\[
\xymatrix{ 
\bullet \ar@/_1pc/[r]|{a_1}^-{\substack{\vdots\\ \\}}
\ar@/^1.2pc/[r]|{a_n} & \bullet \ar@(ur,ul)[]|{b_1}
\ar@(dr,dl)[]|{b_d} 
}\, 
\vdots
\]
where 
$n={\rm dim}({\rm Ext}^1(E_2,E_1))\simeq H^0(G)$, 
and  
$d={\rm dim}({\rm Ext}^1(E_1,E_1))$.
Note that there are no left directed arrows, hence all polynomial invariants of any quiver representation are determined by paths of the form $b_{i_1}b_{i_2}\cdots 
b_{i_j}$.
Then \cite[Thm. 9, Sect 8]{wallcrossing} implies that 
$W_{E_0}= p^*W_{F}$. 

\hfill $\Box$

Since $E_2=\CO_X[1]$, it follows that Lemma \ref{stepone} yields 
\[
w_{E_0} = \IL^{[(E_0,E_0)_{\leq 1} -(E_1,E_1)_{\leq 1} ]/2} w_F
\]
when the conditions of  Lemma \ref{stepone} are satisfied. 
Note that 
\[ 
\bal (E_0,E_0)_{\leq 1} -(E_1,E_1)_{\leq 1} 
& = (\CO_X[1],\CO_X[1])_{\leq 1}
+ (\CO_X[1], F)_{\leq 1} + (F,\CO_{X}[1])_{\leq 1} \\
& = (\CO_X[1],\CO_X[1])_{\leq 1}
- (\CO_X, F)_{\leq 0} - (F,\CO_{X})_{\leq 2} \\
& = {\rm dim}H^0(\CO_X) - {\rm dim}H^0(F) 
-\sum_{i=0}^2 (-1)^i {\rm dim}{\rm Ext}^i(F,\CO_X)\\
& = 1- {\rm dim}H^0(F) + {\rm dim}{\rm Ext}^1(F,\CO_X) 
- {\rm dim}{\rm Ext}^2(F,\CO_X) \\
& = 1- {\rm dim}H^0(F) + {\rm dim}H^2(F) - {\rm dim}H^1(F)\\
\eal 
\]
where Serre duality has been used at the last step. 
Under the conditions of Lemma \ref{stepone}, 
\[
{\rm dim}H^0(F)=n, \qquad {\rm dim}H^1(F)=0.
\] 
Therefore 
\[
(E_0,E_0)_{\leq 1} -(E_1,E_1)_{\leq 1}=1-n.
\]

If this is the case, equation \eqref{eq:fiberint} yields 
\be\label{eq:fiberintB} 
\bal 
\int_{\alpha \in \mathrm{Ext}^1(E_2,E_1)} 
w_{E_\alpha} - w_{E_0} = 
\IL^{(1-n)/2}(\IL^{n} -1) 
w_{F}.
\eal 
\ee
In particular this holds 
  for all sheaves $F$ for sufficiently large $n$. By analogy with 
  \cite[Thm. 4]{stabpairs-III} it is natural to conjecture that the following holds 
  for general $n\in \IZ$
  \be\label{eq:fiberintC} 
\bal 
\int_{\alpha \in \mathrm{Ext}^1(E_2,E_1)} 
w_{E_\alpha} - w_{E_0} = 
\IL^{(1-n)/2}(\IL^{h^0(F)} -1) 
w_{F}.
\eal 
\ee

  Using local toric models, the motivic weights $w_F$ will
  be represented below 
  as of motivic Milnor fibers of polynomial Chern-Simons 
  functions.

 \subsection{Local toric models}\label{fivethree}
A straightforward local computation shows that the formal 
neighborhood of the union 
$D\cup C_0$ equipped with the reduced scheme 
structure is isomorphic to the formal neighborhood of an identical 
configuration in a toric Calabi-Yau threefold. This is in fact easier to see starting with with the small crepant resolution $X^-\to X_0$, 
related to $X\to X_0$ by a flop of the exceptional curve. 
For the elliptic fibration example given in Section 
\ref{onethree},  $X^-$ is a smooth elliptic fibration 
with canonical section over the Hirzebruch surface $\IF_1$. 
The exceptional curve $C_0^-$ is contained in the section $D^-$, 
which is identified with $\IF_1$. Then the formal neighborhood 
of $D^-$ in $X^-$ is isomorphic to the formal neighborhood 
of the zero section in the total space $Z^-$ of the 
canonical bundle $K_{\IF_1}$. Moreover $D^-$ is identified with the zero 
section and $C_0^-$ is identified with the unique $(-1)$-curve 
on $Z^-$. Then one can construct a second smooth toric Calabi-Yau 
threefold $Z^+$ related to $Z^-$ by a toric flop along the curve 
$C_0^-$ as shown in detail below. This threefold contains a compact 
divisor $D^+\simeq \IP^2$ and an exceptional $(-1,-1)$ curve 
$C_0^+$
intersecting $D^+$ transversely at a point $p$. 

The toric presentation of both $Z^\pm$  is of the form 
  \[
  \begin{matrix}
  x_1 & x_2 & x_3 & x_4 & x_5 \\
  1 & 0 & 1 & 1 & -3 \\
  0 & 1 & 0 & 1 & -2. \\
  \end{matrix}
  \]
  The disallowed locus is $\{x_1=x_3=0\}\cup\{x_2=x_4=0\}$ 
  for $Z^-$ and $\{x_1=x_3=x_4\}\cup\{x_2=x_4=0\}$ 
  for $Z^-$. The toric fans $\nabla^\pm$ 
  of $Z^\pm$  are generated by the vectors 
  \[ 
  v_1=(1,0,1), \quad v_2=(1,1,1),\quad v_3=(0,1,1),\quad 
  v_4=(-1,-1,1), \quad v_5=(0,0,1)
  \]
  in $\IR^3$. In  each case the fan is a cone over a 
  two dimensional polytope embedded in the plane $z=1$ 
  in $\IR^3$. The toric flop relating $Z^-$ and $Z^+$ 
  corresponds to a change of  triangulation of the two dimensional 
  polytopes, as shown in Fig. \ref{torictrans}. 
    
  \begin{figure}[h]
  \setlength{\unitlength}{1mm}
\hspace{-60pt}
\begin{picture}(80,40)
\put(20,20){\line(-1,-1){10}}\put(7.5,7.7){$_{D_4^-}$}
\put(20,20){\line(1,1){10}}\put(31,31.5){$_{D_2^-}$}
\put(20,20){\line(1,0){10}}\put(31,19){$_{D_1^-}$}
\put(20,20){\line(0,1){10}}\put(15,31.5){$_{D_3^-}$}
\put(30,20){\line(0,1){10}}
\put(20,30){\line(1,0){10}}
\put(20,30){\line(-1,-2){10}}
\put(30,20){\line(-2,-1){20}}
\put(9.5,9.5){$\bullet$}
\put(29,19){$\bullet$} 
\put(19,29){$\bullet$}
\put(29,29){$\bullet$}
\put(19.5,19.5){$\bullet$}
\put(19,18.5){$_{D_5^-}$}
\put(80,20){\line(-1,-1){10}}\put(67.5,7.7){$_{D_4^+}$}
\put(80,30){\line(1,-1){10}}\put(91,31.5){$_{D_2^+}$}
\put(80,20){\line(1,0){10}}\put(91,19){$_{D_1^+}$}
\put(80,20){\line(0,1){10}}\put(75,31.5){$_{D_3^+}$}
\put(90,20){\line(0,1){10}}
\put(80,30){\line(1,0){10}}
\put(80,30){\line(-1,-2){10}}
\put(90,20){\line(-2,-1){20}}
\put(69.5,9.5){$\bullet$}
\put(89,19){$\bullet$} 
\put(79,29){$\bullet$}
\put(89,29){$\bullet$}
\put(79.5,19.5){$\bullet$}
\put(79,18.5){$_{D_5^+}$}
\end{picture}
\caption{Local toric models related by a flop. The polytope on the 
left is the $z=1$ section of the toric fan of the local $\IF_1$ model. The polytope on the right is a similar section of the toric fan of the 
local $\IP^2\cup \IP^1$ model. The two models are related by a toric flop corresponding to the obvious change of triangulation.}
\label{torictrans}
\end{figure}
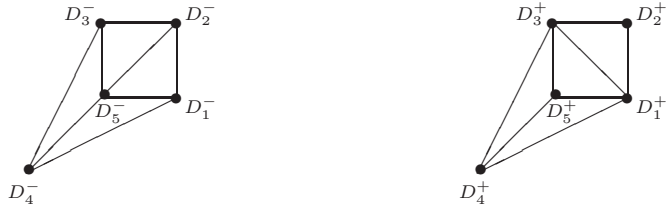

\noindent
The canonical toric divisors $x_i=0$ are denoted by $D_i^\pm$, $i=1,\ldots, 5$. They are in one-to-one correspondence with the rays of the toric fans as shown in Fig. \ref{torictrans}.  Note that 
$D^\pm =D^\pm_5$ are the only compact divisors on $Z^\pm$.
  
The derived categories of $Z^\pm$ are equivalent and are 
generated by line bundles. A collection of line bundles 
generating $D^b(Z^-)$ is obtained by pulling back 
an exceptional collection on the Hirzebruch surface 
$\IF^1$ of the form 
\[ 
\CO_{\IF_1}, \quad\CO_{\IF_1}(C_0^-), \quad 
\CO_{\IF_1}(H), \quad \CO_{\IF_1}(2H).
\] 
Here $C_0^-$ denotes the exceptional curve on $\IF_1$ 
and $H$ the hyperplane class. Note that the resulting 
line bundles on $Z^-$ are isomorphic to the toric 
line bundles 
\[
\CO_{Z^-}, \quad \CO_{Z^-}(D_2^-), \quad \CO_{Z^-}(D_4^-), 
\quad  \CO_{Z^-}(2D_4^-).
\]
The direct sum $\CT^-$ of all above line bundles 
is a tilting object, and the derived category of $Z^-$ is 
equivalent to the derived category of modules over the 
algebra $R{\mathrm{End}}_{Z^-}(\CT^-)^{op}$. 
The equivalence is given by the derived functor 
$R{\rm Hom}_{Z^-}(\CT^-, \bullet)$. 
As a result the derived category of $Z^-$ is equivalent to the 
$3CY$ category which is a Calabi-Yau category associated with the abelian category $(Q,W)-mod$ of finite-dimensional representations of the following quiver  $Q$ 
\bigskip

\be\label{eq:Fonequiver}
\xymatrix{ 
\bullet \ar[r]|{a_2} \ar@/_/[r]|{a_3} \ar@/^/[r]|{a_1} 
& \bullet \ar@/_/[r]|{b_2}
\ar@/^2pc/[rr]|{b_3}
\ar@/^/[r]|{b_1} & \bullet \ar[r]|{c}
\ar@/^1.5pc/[ll]|{r} &\bullet \ar@/^3.5pc/[lll]|{s_2}
\ar@/^2.5pc/[lll]|{s_1}
}
\ee
\bigskip
\noindent
with potential
\be\label{eq:quivpotentialA}
W = r(b_1a_2-b_2a_1) + s_1(cb_1a_3-b_3a_1) + 
s_2(cb_2a_3-b_3a_2).
\ee
Recall that this category can be described as the category of finite-dimensional representations of the Jacobi algebra ${\mathbb C}Q/(\partial W)$, the quotient of the path algebra of $Q$ by the ideal generated by cyclic derivatives of $W$.

For future reference note that the 
line bundles 
\[
\CO_{D^-}, \qquad \CO_{D^-}(D_2^-), \qquad 
\CO_{D^-}(D_4^-), \qquad \CO_{D^-}(2D_4^-) 
\]
form an exceptional collection $\CT_{D^-}$ 
on the Hirzebruch 
surface $D^-\simeq\IF_1$. 
The functor $R{\rm Hom}(\CT_{D^-},\bullet)$ 
yields an equivalence of the derived category $D^b(D^-)$ 
to the derived category of the abelian category $(Q_0,S)-mod$  of the finite-dimensional representations of following quiver $Q_0$ 
 \bigskip

\be\label{eq:FonequiverB}
\xymatrix{ 
\bullet \ar[r]|{a_2} \ar@/_/[r]|{a_3} \ar@/^/[r]|{a_1} 
& \bullet \ar@/_/[r]|{b_2}
\ar@/^2pc/[rr]|{b_3}
\ar@/^/[r]|{b_1} & \bullet \ar[r]|{c}&\bullet 
}
\ee
\bigskip
\noindent
with relations 
\[S:\quad 
b_1a_2-b_2a_1,\qquad cb_1a_3-b_3a_1,\qquad cb_2a_3-b_3a_2.
\]
The abelian category $(Q_0,S)-mod$  has homological dimension $2$, and there is an obvious injective fully faithful exact functor 
of abelian categories 
\[
\iota: (Q_0,S)-mod\ \longto\ (Q,W)-mod.
\] 
For simplicity, extension groups in the two categories will be denoted by ${\rm Ext}^\bullet_{(Q_0,S)}, {\rm Ext}^\bullet_{(Q,W)}$ respectively. 
It will be useful to note that the following relations 
hold:
\be\label{eq:extrelations} 
\bal 
{\rm Ext}^0_{(Q,W)}(\iota \rho_1, \iota \rho_2)
\simeq & \ {\rm Ext}^{0}_{(Q_0,S)}(\rho_1,\rho_2) \\
{\rm Ext}^k_{(Q,W)}(\iota \rho_1, \iota \rho_2)
\simeq & \ {\rm Ext}^{k}_{(Q_0,S)}(\rho_1,\rho_2) 
\oplus 
{\rm Ext}^{3-k}_{(Q_0,S)}(\rho_2,\rho_1)^\vee, \quad 
k=1,2.\\
\eal
\ee

Using the results of \cite{tiltingII}, the direct sum $\CT^+$ of the 
following collection of line bundles 
\[
\CL_1=\CO_{Z^+}(2D_4^+),\qquad 
 \CL_2=\CO_{Z^+}(D_4^+),\qquad 
 \CL_3=\CO_{Z^+}(D_2^+), \qquad 
 \CL_4=\CO_{Z^+}.
\]
is a tilting object in the derived category of $Z^+$. 
Therefore it yields a similar equivalence of $D^b(Z^+)$ to the 
derived category of the same quiver with potential. 

The next step is to compute the image of dimension one sheaves on $Z^+$
via the tilting functor. First note the following 
 result which follows from   \cite[Lemm 9.1]{algCS}.

\begin{lemm}\label{localptwo}
Let $G$ be a rank one torsion free sheaf on a degree $k\in \IZ_{>0}$
reduced irreducible divisor on
$D^+\simeq \IP^2$ with $H^0(G)=0$. 
 Then the complex 
$R{\rm Hom}(\CT^+,G)[1]$ is quasi-isomorphic to a quiver representation 
$\rho_G$ of dimension vector 
\[ 
v_G = \left(2k-\chi(G), k-\chi(G),
-\chi(G), -\chi(G)\right).
\]
which belongs to the subcategory of $(Q_0,S)$-modules. 
Moreover 
\be\label{eq:soothnessA}
{\rm Ext}^2_{(Q_0,S)}(G,G) =0,
\ee 
and $\rho_G(c)$ is an isomorphism if $\chi(G)\neq 0$. 
\end{lemm}

\begin{proof} 
Note that the open subset $U=\{x_2\neq 0\}\subset Z^+$ 
is isomorphic to the total space 
of the normal bundle $N_{D^+/Z^+}\simeq \omega_{\IP^2}$. 
This follows observing that $U$ is isomorphic to a toric variety 
determined by the toric data 
\[
\begin{array}{ccccc}
x_1' & x_3' & x_4' & x_5' \\
1 & 1 & 1 & -3 \\
\end{array}
\]
where 
\[
x_1'=x_1, \qquad x_3'=x_3, \qquad x_4'=x_2^{-1}x_4\qquad 
x_5'=x_2^{-2}x_5 
\]
and the disallowed locus is $\{x_1'=x_3'=x_4'=0\}$. 

Denote the open immersion $U\hookrightarrow Z^+$ by 
$j$ and the close immersions of $D^+$ into $Z^+$ and $U$ by $i$ and $i'$ respectively. Clearly, $i=j\circ i'$. Denote the tilting bundle on $Z^+$ by $\CT^+$. Given a sheaf on $Z^+$ of the form $i_*G$, there is an isomorphism 
\[
\RHom_{Z^+}(T,i_*G)\simeq \RHom_U(j^*T,i'_*G).
\]
By adjunction, this is further isomorphic to $\RHom_{\IP^2}(\cO^{\oplus 2}\oplus \cO(1)\oplus \cO(2),G)$. By the  derived Morita equivalence, this induces an equivalence between $D^b(\IP^2)$ and the derived category of the abelian category $\cA$ consisting of representations $\rho$ of the directed quiver $Q_0$ with dimension vectors $(v_1,v_2,v_3,v_3)$ and 
$\rho(c)$ an invertible linear map.

Since $\cA$ is a fully faithful subcategory of $(Q_0,S)-{\rm mod}$, we have
$${\rm Ext}^2_{(Q_0,S)}(\rho_G,\rho_G)={\rm Ext}^2_{\cA}(\rho_G,\rho_G)={\rm Ext}^2_{\IP^2}(G,G)=0$$ when $G$ is stable. 
\end{proof}

The next goal is to compute the image of nondegenerate extensions 
\[ 
0\to G\to F \to V\otimes \CO_{C^+_0}(-1)\to 0
\]
via the tilting functor. In order to obtain a single quiver representation
as a opposed to a complex thereof, $F$ must 
be twisted by a suitable line 
bundle $L$ prior to tilting. 
There are several possible results depending 
on the choice of $L$. The one recorded below turns out to be most effective 
for the computation of motivic weights. 

As shown in the proof of Lemma \ref{dualstacks}, taking derived 
duals on $X$ sends a  nondegenerate extension  as above to 
an extension of the form 
\[
0\to V^\vee \otimes \CO_{C_0^+}(-1)\to  
{\mathcal Ext}^2_{Z^+}(F,\CO_{Z^+})
\to \CJ\otimes_C\omega_C\to 0
\] 
where $\CJ = RHom_{C}(G,\CO_C)$ is an ideal sheaf on $C$.
The dualizing sheaf  of $C$ 
is $\omega_C 
\simeq \CO_{Z^+}((k-3)D_4^+)|_{C}$. 
Let $W=V^\vee$.
The dual extension is also subject to a nondegeneracy condition. 
Namely the corresponding extension class 
$e\in {\rm Ext}^1(\CJ\otimes_C \omega_C, W\otimes 
\CO_{C_0^+}(-1))$ is not in the kernel of the map 
\[ 
{\rm Ext}^1(\CJ\otimes_C \omega_C, W\otimes 
\CO_{C_0^+}(-1))\longto 
{\rm Ext}^1(\CJ\otimes_C \omega_C, W'\otimes 
\CO_{C_0^+}(-1))
\]
for any nontrivial quotient $W\twoheadrightarrow W'$. The tilting functor 
will be applied to the twist 
$F'={\mathcal Ext}^2_{Z^+}(F,\CO_{Z^+})\otimes_{Z^+}((2-k)D_4^+))$ 
which fits in an extension 
\[
0\to W \otimes \CO_{C_0^+}(-1)\to  F'
\to \CJ(-D_4^+)\to 0.
\] 
Then the following holds:
\begin{lemm}\label{dualtiltinglemma}
Consider a  nondegenerate extension 
\be\label{eq:sheafext}
0\to W\otimes \CO_{C_0^+}(-1) \to F' \to \CJ' \to 0
\ee
where $\CJ'=\CJ(-D_4^+)$ for an ideal sheaf $\CJ$ 
on a degree $k\in \IZ_{>0}$
reduced irreducible divisor $C^+$ on $D^+\simeq \IP^2$. 
Then  
$R{\rm Hom}_{Z^+}(\CT^+, F')[1]$ 
is quasi-isomorphic to a quiver representation $\rho_{F}$ which 
fits in an extension 
\be\label{eq:repextB} 
0\to W\otimes \rho_3 \to \rho_{F} \to \rho_{\CJ'}\to 0.
\ee
In addition, $\rho_{F}$ belongs to the subcategory of $(Q_0,S)$-modules, and 
\be\label{eq:smoothnessB} 
{\rm Ext}^2_{(Q_0,S)}(\rho_F,\rho_F)=0.
\ee
\end{lemm}

{\it Proof.} 
Observe that $\CO_{C_0^+}(-1)$ is mapped to the simple module 
$\rho_3[-1]$ corresponding to the third vertex of the quiver $Q$. According to Lemma \ref{localptwo} 
the twisted derived dual $J'$ of $G$ will be mapped to a representation $\rho_{J'}[-1]$ of $Q_0$ since $H^0(J')=0$. 
Moreover the linear map $\rho_{J'}(c)$ is invertible. 
Then we claim 
\be\label{eq:vanextA}
{\rm Ext}_{(Q_0,S)}^k(\rho_3, \rho_{\CJ'})=0 
\ee
for all $k\in \IZ$.
Suppose $\rho_{J'}$ has dimension vector $d_1,\ldots,d_4$, recall that $\rho_{J'}$ corresponds to a Maurer-Cartan element $x$ of the $L_\infty$ algebra $\Ext^*_{(Q_0,S)}(\oplus \rho_i\otimes V_i,\oplus \rho_i\otimes V_i)$, where the dimension of 
$V_i$ equals $d_i$. The extension space $\Ext^*_{(Q_0,S)}(\rho_3,\oplus \rho_i\otimes V_i)$ is an $L_\infty$ module over 
$\Ext^*_{(Q_0,S)}(\oplus \rho_i\otimes V_i,\oplus \rho_i\otimes V_i)$. The Maurer-Cartan element $x$ defines a differential $\delta^x$ on  $\Ext^*_{(Q_0,S)}(\rho_3,\oplus \rho_i\otimes V_i)$ such that the cohomology groups compute $\Ext^*(\rho_3,\rho_{J'})$. The complex  $\Ext^*_{(Q_0,S)}(\rho_3,\oplus \rho_i\otimes V_i)$ has the form
\[
\xymatrix{
0\ar[r] &{\rm Hom}(\IC,V_3)\ar[r]^{\delta^x}&
{\rm Hom}(\IC, V_4)\ar[r] &0
}
\]
Since the linear map $\rho_{J'}(c)$ is invertible, this complex is acyclic. For future reference, note that a similar argument proves that 
\be\label{eq:vanextB}
{\rm Ext}_{(Q_0,S)}^k(\rho_3, \rho_{3})=0 
\ee
for all $k\in \IZ\setminus \{0\}$. 

According to relations \eqref{eq:extrelations}
 the extension group $\Ext^1_{(Q,W)}(\rho_{J'},\rho_3)$ decomposes into $\Ext^1_{(Q_0,S)}(\rho_{J'},\rho_3)\oplus \Ext^2_{(Q_0,S)}(\rho_3,\rho_{J'})^\vee$. Since we have just proved the second summand vanishes, it follows 
any extension of the form \eqref{eq:sheafext}
 must be mapped by tilting to a representation $\rho_F$ 
  of $(Q_0,S)$.

Since $\rho_F$ is an extension of $\rho_{J'}$ by $W\otimes \rho_3$, the extension group $\Ext^2_{(Q_0,S)}(\rho_F,\rho_F)$ is computed by the complex $\Ext^*(\rho_J\oplus \rho_3,\rho_J\oplus W\otimes\rho_3)$ with the differential $\delta^x$ where $x$ is the Maurer-Cartan element corresponding to the extension class in $\Ext^1_{(Q_0,S)}(\rho_{J'},W\otimes\rho_3)$. The vanishing results 
\eqref{eq:vanextA}, \eqref{eq:vanextB} imply that 
$\Ext^2_{(Q_0,S)}(\rho_F,\rho_F)$ is isomorphic to the cokernel of 
the map 
\[
\xymatrix{ \Ext^1_{(Q_0,S)}(\rho_{J'},\rho_{J'})\oplus \Ext^1_{(Q_0,S)}(\rho_{J'},W\otimes\rho_3)\ar[r]^-{\delta^x} & \Ext_{(Q_0,S)}^2(\rho_{J'},W\otimes\rho_3)
}
\]
Because $x\in \Ext^1_{(Q_0,S)}(\rho_{J'},W\otimes\rho_3)$, the above morphism simplifies to
\be\label{eq:quiverdelta}
\xymatrix{\Ext^1_{(Q_0,S)}(\rho_{J'},\rho_{J'})\ar[r]^-{\delta^x} & {\rm Ext}_{(Q_0,S)}^2(\rho_{J'},W\otimes\rho_3).}\ee
Vanishing of ${\rm Ext}^2_{(Q_0,S)}(\rho_F,\rho_F)$
is equivalent with the above morphism being surjective. 
Furthermore, relations \eqref{eq:extrelations} and the vanishing results \eqref{eq:vanextA} imply that 
$$
\bal \Ext^1_{(Q_0,S)}(\rho_{J'},\rho_{J'}) & =\Ext^1_{(Q,W)}(\rho_{J'},\rho_{J'}), \\
  \Ext_{(Q_0,S)}^2(\rho_{J'},W\otimes\rho_3) & =\Ext_{(Q,W)}^2(\rho_{J'},W\otimes\rho_3).\eal 
$$ Then
derived equivalence with $D^b(Z^+)$ maps the 
morphism \eqref{eq:quiverdelta} to the connecting 
morphism 
\[
\xymatrix{
\Ext_{Z^+}^1(J',J')\ar[r]^-{\delta} &\Ext_{Z^+}^2(J',
W\otimes \CO_{C^+_0(-1)})}.
\]

In order to show that  $\delta$ is a surjection recall 
that according to 
Lemma \ref{extlemma}
there are isomorphisms 
$$\varphi_k:
\Ext^k_{Z^+}(J' ,W\otimes\CO_{C^+_0}(-1)){\buildrel \sim
\over \longto}
\Ext_{D^+}^{k-1}(J,W\otimes\CO_p).$$ 
Moreover Corollary \ref{nondegext} shows that an extension 
$e\in \Ext^1_{Z^+}(J' ,W\otimes\CO_{C^+_0}(-1))$ 
is nondegenerate if and only if the corresponding morphism 
$\varphi_1(e)$ is surjective. In particular this 
holds for the extension 
class $e^x$ corresponding to the Maurer-Cartan element $x$. 
Let $\psi =\varphi_1(e_x)$ and 
$\psi_*: \Ext^1_{D^+}(J',J')\to \Ext^1_{D^+}(J', W\otimes\CO_p)$ the natural induced morphism of extensions.  Clearly 
the following diagram commutes. 
\[
\xymatrix{
\Ext^1_{Z^+}(J',J')\ar[r]^-{\delta}\ar[d]^{\simeq} & 
\Ext_{Z^+}^2(J',W\otimes\CO_{C^+_0(-1)})\ar[d]^{\simeq}\\
\Ext^1_{D^+}(J',J')\ar[r]^-{\psi_*} &\Ext^1_{D^+}
(J', W\otimes\CO_p)}
\]
Since $\psi$ is surjective, surjectivity of $\psi_*$ follows from 
the vanishing result ${\rm Ext}^2_D(J',{\rm Ker}(\psi))=0$ obtained
 in the proof of 
Lemma \ref{smoothmoduli}, and Remark \ref{smoothbound}. 

\hfill $\Box$

\subsection{Motivic weights in local model}\label{fivefour}
Next it will be shown that 
Lemma \ref{dualtiltinglemma} yields a presentation of 
the motivic  weights $w_F$ as motivic Milnor fibers 
of polynomial functions. Note that the quiver $Q$ in 
\eqref{eq:Fonequiver} is the ${\rm Ext}^1$ quiver associated 
to four spherical objects $S_i$, $i=1,\ldots, 4$ in the 
derived category $D^b(Z^+)$. Moreover the objects $S_i$, 
$i=1,\ldots, 4$ generate the subcategory consisting of 
complexes with topological support on $D^+\cup C_0^+$. 
The images of these objects via the tilting functor generate 
the subcategory of complexes of quiver representations 
with nilpotent cohomology.  In particular the representation 
$\rho_F$ corresponding to a sheaf $F$ as in Lemma \ref{dualtiltinglemma} is obtained by successive extensions 
of the $S_i$, $i=1,\ldots, 4$.

For a dimension vector $v=(v_i)_{1\leq i\leq 4}$, let $\IA(v)$ denote the affine space parameterizing 
all representations of the quiver $Q$ without relations.
Note that there is an obvious direct sum decomposition 
\[ 
\IA(v) = \IA^r(v) \oplus \IA^l(v)
\]
where $\IA^r(v)$, $\IA^l(v)$ denote the linear subspaces 
associated to the right directed, and left directed arrows 
respectively in diagram \eqref{eq:Fonequiver}. 
There is also a natural $\IG(v)=\prod_{i=1}^4GL(v_i)$ 
action on $\IA(v)$. 
 
 The potential \eqref{eq:quivpotentialA} determines a 
 $\IG(v)$-invariant quartic polynomial 
 function $\CW$ on ${\mathbb A}(v)$ such that quiver representations of dimension vector 
$v=(v_i)_{1\leq i\leq 4}$ 
are in one-to-one correspondence with 
closed points in the critical locus ${\mathrm{Crit}}(\CW)$. 

Let $\rho_F\in {\mathbb A}(v)$ be a closed point corresponding to a sheaf $F$ satisfying the conditions 
of Lemma \ref{dualtiltinglemma}. Let $\CW_{\rho_F}$ be the Taylor series expansion of $\CW$ at $\rho_F$. Since 
$\rho_F$ 
is an iterated extension of the spherical objects $S_i$, $i=1,\ldots, 4$, the computation of $w_{\rho_F}=w_F$ will be carried out in close 
analogy with  the proof of 
\cite[Thm. 8, Sect. 6.3]{wallcrossing}.

Suppose $E_1,E_2$ are any two objects in derived category of quiver representations with nilpotent cohomology. Let 
$E_0=E_1\oplus E_2$.
Suppose moreover that the potential function $W_{E_0}$ on
\[
{\rm Ext}^1(E_0,E_0) = 
{\rm Ext}^1(E_1,E_1) \oplus 
{\rm Ext}^1(E_2,E_1) \oplus 
{\rm Ext}^1(E_1,E_2) \oplus 
{\rm Ext}^1(E_2,E_2)
\]
is minimal i.e. has no quadratic part. 
Let 
$\alpha \in {\rm Hom}(E_2[-1],E_1)$
be a nontrivial element, and let $E_\alpha ={\sf Cone}(\alpha)$.
 As in {\it Step 3} in the 
proof of \cite[Thm. 8, Sect. 6.3]{wallcrossing}, 
let $W_{(0,\alpha, 0,0)}$ denote the Taylor expansion 
of $W_{E_0}$ at the point $(0,\alpha,0,0)$ in ${\rm Ext}^1(E_0,E_0)$.
Then $W_{(0,\alpha,0,0)}$ is related by a formal change 
of variables to a direct sum of the form 
\[ 
W_{E_\alpha}^{min}\oplus 
{\widetilde Q}_{E_\alpha} \oplus {\widetilde N}_{E_\alpha}
\] 
where ${\widetilde Q}_{E_\alpha}$ is a nondegenerate 
quadratic form and ${\widetilde N}_{E_\alpha}$ the zero function on a linear subspace. 
This implies that there is an identity 
\be\label{eq:milnorfibers}
(1-MF_0(W_{(0,\alpha,0,0)})) = (1-MF_0(W_{E_\alpha}^{min}) 
(1-MF_0({\widetilde Q}_{E_\alpha})).
\ee
Note that ${\widetilde Q}_{E_\alpha}$ is not the same as the 
intrinsic quadratic form $Q_{E_\alpha}$. In fact the discrepancy  between 
these two forms leads to the need to introduce orientation data 
in order to obtain a well defined integration map. 

 Two identities for the quadratic form 
 ${\widetilde Q}_{E_\alpha}$ follow from the proof of \cite[Thm. 8, Sect 6.3]{wallcrossing}.
   First, the rank of 
 ${\widetilde Q}_{E_\alpha}$ is expressed in terms of 
 dimensions of 
 Ext groups as follows 
 \be\label{eq:rankquadrform}
 \rk({\widetilde Q}_{E_\alpha}) = 
 (E_\alpha,E_\alpha)_{\leq 1} - 
 (E_0, E_0)_{\leq 1}.
 \ee
 Next, there is a cocycle identity for motivic Milnor 
 fibers above  \cite[Def. 18, Sect 6.3]{wallcrossing}
 which reads 
\be\label{eq:cocycle}
\bal
& \IL^{-\rk({Q}_{E_\alpha})/2} 
(1-MF_0({Q}_{E_\alpha})) = \\
& \IL^{-\rk({\widetilde Q}_{E_\alpha})/2} 
(1-MF_0({\widetilde Q}_{E_\alpha})) 
\prod_{i=1}^2 \IL^{-\rk({Q}_{E_i})/2} 
(1-MF_0({Q}_{E_i})). \\
\eal
\ee

In the present case $E_0$ is a direct sum of simple objects 
\[
E_0= \bigoplus_{i=1}^4 S_i^{\oplus v_F(i)}
\]
where $v_F=(v_F(i))_{1\leq i\leq 4}$ is the dimension vector 
of the extension $\rho_F$ of Lemma \ref{localptwo},
\[
{v}_H = \left((N+2)k - n, (N+1)k-n,
Nk-n+r, Nk-n\right).
\]
Then
equation \eqref{eq:milnorfibers} yields 
\be\label{eq:milnorfibersB}
1-MF_0(\CW_{\rho_F}) = (1-MF_0(W_F^{min})) 
(1-MF_0({\widetilde Q}_{\rho_F})). 
\ee
where ${\widetilde Q}_{\rho_F}$ is a quadratic form
which satisfies two identities analogous to 
\eqref{eq:rankquadrform}, \eqref{eq:cocycle}. 
Therefore the rank of ${\widetilde Q}_{\rho_F}$ is given by 
\be\label{eq:rankquadrformB} 
\bal
\rk({\widetilde Q}_{\rho_F}) &  = 
(F,F)_{\leq 1} - (E_0,E_0)_{\leq 1} \\
& = (F,F)_{\leq 1} + {\mathrm {dim}}(\IA(v_F)) -
{\mathrm {dim}}(\IG(v_F)).\\
\eal 
\ee
Moreover there is a cocycle identity 
\be\label{eq:cocycleB}
\bal 
\IL^{-\rk({\widetilde Q}_{\rho_F})/2} (1-MF_0({\widetilde Q}_{\rho_F})) = 
\IL^{-\rk({ Q}_{F})/2} (1-MF_0({Q}_{F})) 
\eal 
\ee
since $Q_{S_i}=0$, $i=1,\ldots, 4$ for the spherical objects. 
Equations \eqref{eq:milnorfibersB}, \eqref{eq:rankquadrformB}, \eqref{eq:cocycleB} then yield the 
following expression 
\be\label{eq:motivicweightD} 
w_F = \IL^{(\mathrm{dim}(\IG(v_F)) - \mathrm{dim}(\IA(v_F))/2}
(1-MF_0(\CW_{\rho_F})), 
\ee
where $\CW_{\rho_F}$ is the polynomial function
\[
\CW_{\rho_F}(\rho) = \CW(\rho+\rho_F)
\]
for any $\rho\in \IA(v_F)$. Note that $MF_0(\CW_{\rho_F})= MF_{\rho_F}(\CW)$ by functoriality of motivic Milnor fibers. 

In general explicit computations of pointwise Milnor fibers are difficult. 
The following Lemma shows that the computation is 
tractable on a certain subset of the critical locus
of $\CW$. 
Let  $MC_0= {\rm Crit}(\CW)\cap \IA^r(v)$ be the 
subscheme of critical points with trivial left directed arrows.
The potential $\CW:\IA(v)\to 
\IC$ is of the form 
\[
\CW = \sum_{\kappa=1}^K y_\kappa P_\kappa
\]
where $(y_\kappa)_{1\leq \kappa\leq K}$ are natural linear coordinates on 
$\IA^l(v)$ 
and $P_\kappa:\IA^r(v)\to \IC$ are polynomial functions. 
Then $MC_0$ is determined by 
\[
y_\kappa=0, \qquad 
P_\kappa=0, \qquad 
\kappa=1,\ldots, K.
\]
Let $\CX_0=\CW^{-1}(0)$ denote the central fiber. 
Note that there is a commutative diagram 
\[
\xymatrix{ 
{\rm Crit}(\CW) \ar@{^{(}->}[r] \ar[d]^-{p^{cr}}& {\CX}_0 \ar[d]^-{p}\\
 MC_0 \ar@{^{(}->}[r] \ar@/^1.5pc/[u]_-{\iota}& \IA^r(v)\\}
\]
where $p:{\CX}_0\to \IA^r(v)$ is the restriction of 
the canonical projection $\IA(v)\twoheadrightarrow \IA^r(v)$
and $\iota$ is the zero section 
$y_\kappa=0$, $\kappa=1,\ldots,K$. 
Note that the fibers of $p,p^{cr}$ are linear subspaces of $\IA^l(v)$. 
Let ${{MC}^{}}_0^{sm}$ denote the 
smooth open locus of $MC_0$. 

\begin{lemm}\label{smoothpoint}
Let $\rho\in {{MC}^{}}_0^{sm}$. Then the motivic 
weight at $\iota(\rho)$ is 
\[
1-MF_{\iota(\rho)}(\CW) = \IL^{\mathrm{dim}\, \IA^l(v)}.
\]
\end{lemm}

{\it Proof.} 
Let $\CU\subset \IA(v)$ be the open subset where the Jacobian 
matrix of the polynomial functions $(P_\kappa)$, $\kappa=1,\ldots, K$ has maximal rank. Then $\CU\cap MC_0=
{{MC}^{}}_0^{sm}$. Let $\CY_0$ be the restriction of the central fiber 
$\CX_0$ to $\CU$ and $q:\CY_0\to \IA^r(v)$ 
the restriction of $p$.  Note that  the singular locus $\CY_0^{sing}\subset \CY_0$ 
is determined by the equations 
\[ 
y_\kappa=0, \qquad P_\kappa=0, \qquad \kappa=1,\ldots,K. 
\]
This follows from the fact that there is a 
factorization 
\[
{\CU} {\buildrel P\over \longto} \IA^l(v)\times \IA^l(v) 
{\buildrel Q\over \longto} \IC 
\]
of $\CW|_{\CU}:\CU\to 
\IC$, where 
\[
P(y_\kappa, x) = (y_\kappa, P_\kappa(x)) 
\]
for any $x\in \IA^r(v)$, $(y_\kappa)\in \IA^l(v)$ 
and 
\[
Q(y_\kappa,z_\kappa) = \sum_{\kappa=1}^K y_\kappa z_\kappa.
\]
Since the Jacobian matrix of $(P_\kappa)$ has maximal rank on 
$\CU$, the map $P$ is smooth. Moreover, the singular locus 
of the central fiber of $Q$ is obviously $y_\kappa=z_\kappa=0$
for all $\kappa=1,\ldots,K$. This implies the claim. 

In conclusion, $\CY_0^{sing}$ 
coincides with the image $\iota({{MC}^{}}_0^{sm})\subset \CY_0$.
Note also that the fibers of $p$ over closed points 
$\rho\in {{MC}^{}}_0^{sm}$ 
are isomorphic to $\IA^l(v)$. 
Then a normal crossing resolution of $\CY_0$ can be obtained by a single embedded blow-up. Let $\sigma:{\CU}'\to \CU$ be the blow-up of $\CU$ along the 
linear subspace 
\[
y_\kappa=0, \qquad \kappa=1,\ldots, K.
\]
The total transform $\sigma^{-1}(\CY_0)$ consists of the 
strict transform $\CY_0'$ and an exceptional divisor $D$ 
isomorphic to a $\IP(\IA^l(v))$-bundle over $\IA^r(v)$. 
The strict transform $\CY_0'$ is smooth and intersects $D$ 
transversely along a divisor $D'\subset \CY_0'$, which is 
isomorphic to a $\IP(\IA^l(v))$-bundle over ${{MC}^{}}_0^{sm}$. 
Moreover both $\CY_0'$ and $D$ 
multiplicity $1$ in $\sigma^{-1}(\CY_0)$.  

For any point $\rho\in {{MC}^{}}_0^{sm}$, 
$\sigma^{-1}(\iota(\rho))$ intersects both $\CY_0'$ 
and $D$ along the fiber $D_\rho\subset D$, which is 
isomorphic to $\IP(\IA^l(v))$.   
Therefore, from the definition \cite[Sect. 4, pp. 67]{wallcrossing}
\[ 
1-MF_{\iota(\rho)}(\CW) = 1- (1-\IL) [\IP(\IA^l(v))]= 
\IL^{\mathrm{dim}(\IA^l(v))}.
\]

\hfill $\Box$ 

\noindent

\subsection{Comparison with refined Hilbert scheme 
invariants}\label{fivefive}
The compact motivic version of Hilbert scheme invariants 
has been defined in equation \eqref{eq:compactHilbmotivic}, 
which is reproduced below for convenience 
\[
Z^{mot}_{C}(q,a) = \sum_{l,r\geq 0} q^{2l} a^{2r}\IL^{r^2/2} 
[{H}^{[l,r]}(C)].
 \]
Note that the Chow motive of the nested Hilbert 
$[{H}^{[l,r]}(C)]$ is equal to the Chow motive $[{Q}^{[l,r]}(C)]$ 
of the relative {\it Quot} scheme 
defined above Proposition \ref{projhilbert}.  
Moreover, the stack ${\mathcal Q}(X,C,r,n)$ is a $\IC^\times$ 
gerbe over the relative 
{\it Quot} scheme $Q^{[l,r]}(C)$, $l=n-\chi(\CO_C)$,
according to Proposition \ref{projhilbert}. 
As observed in Remark \ref{forgetsectionB},
the moduli stack ${\mathcal M}(X,C,r,n)$ is also a $\IC^\times$ 
gerbe over a coarse moduli scheme ${M}^{[l,r]}(C)$,
and there is a 
natural forgetful morphism $\pi:Q^{[l,r]}(C)\to M^{[l,r]}(C)$. 
Note  also that there is a natural stratification of $M^{[l,r]}(C)$ such that the 
restriction of $\pi$ to each stratum is a smooth projective bundle
with fiber $\IP^{h^0(F)-1}$. 
Since the motivic weights $w_F$ are invariant under isomorphisms, 
$F\simeq F'$, they descend to motivic weights $w_{[F]}$ on the coarse 
moduli space $M^{[l,r]}(C)$. 

Then using the conjectural identity \eqref{eq:fiberintC} a stratification 
argument implies that the virtual
motive of the stack function ${f}:{\mathcal Q}(X,C,r,n)\hookrightarrow 
{Ob}(\CA)$ is given by 
\[
{1\over \IL-1}
\Phi([{f}:{\mathcal Q}(X,C,r,n) \to {Ob}(\CA) 
]) = \IL^{(1-n)/2} \int_{M^{[l,r]}(C)} [\IP^{h^0(F)-1}] w_{[F]}.
\]
Applying Lemmas \ref{smoothpoint}, \ref{dualtiltinglemma}, 
one then obtains 
\[ 
\bal 
& {1\over \IL-1}
\Phi([{f}:{\mathcal Q}(X,C,r,n) \to {Ob}(\CA) 
]) = \\
& \IL^{(1-n)/2} \IL^{(\mathrm{dim}(\IG(v_F)) - \mathrm{dim}(\IA(v_F))/2+ 
\mathrm{dim}(\IA^l(v_F))/2} \int_{M^{[l,r]}(C)} [\IP^{h^0(F)-1}] \\
\eal
\]
Note that 
\[
(\mathrm{dim}(\IG(v_F)) - \mathrm{dim}(\IA(v_F))/2+ 
\mathrm{dim}(\IA^l(v_F))/2=(r^2-k^2)/2
\]
by a straightforward computation. 
Therefore the final formula is  
\[
{1\over \IL-1}
\Phi([{f}:{\mathcal Q}(X,C,r,n) \to {Ob}(\CA) 
]) 
= \IL^{(r^2-k^2+1-n)/2} [{Q}^{[l,r]}(C)].
\]
Then the resulting generating function of $C$-framed virtual 
motivic invariants in the small $b>0$ chamber is 
\[
\bal 
Z^{mot}_{0+}(X,C; u,T) & =  \IL^{(1-k^2)/2} \sum_{r\geq 0} \sum_{l\geq 0} 
u^{n} T^{r} \IL^{(r^2-n)/2} [{Q}^{[l,r]}(C)]\\
& = \IL^{(1-k^2-\chi(\CO_C))/2} u^{\chi(\CO_C)} 
\sum_{r\geq 0} \sum_{l\geq 0} 
u^{l} T^{r} \IL^{(r^2-l)/2} [{ Q}^{[l,r]}(C)]\\
\eal
\]
In conclusion, note that identity \eqref{eq:motivicidentity} holds i.e. 
\[
Z^{mot}_{0+}(X,C;q^2\IL^{1/2},a^2) = 
\IL^{(1-k^2)/2}q^{2\chi(\CO_C)}Z^{mot}_C(q,a).
\]

\appendix
\section{Wallcrossing formula}\label{proofwallcrossing} 

For completeness, a proof of Proposition 
\ref{factorization}  is presented here in detail.
 In the view of Theorem \ref{smallBidentity}, this 
proves Theorem \ref{topCframedinv}. Although the motivic 
Donaldson-Thomas theory of \cite{wallcrossing} 
is consistently used throughout 
this paper, the proof of equation \eqref{eq:factformulaD}
will be based on the alternative 
wallcrossing formalism 
developed in \cite{J-I,J-II,J-III,J-IV,genDTI}. 
The first is more general, but requires more work on the foundations, 
as explained in detail in Section \ref{motivicsect}. 
As stated in the main text, several similar computations have already been carried out in the literature, for example in 
 \cite[Sect 4.3]{generating} and \cite[Thm 3.15]{NH}, and 
 also \cite[Sect 2]{chamberII}, \cite[Sect 3]{ranktwo}.
 The approach explained below follows closely 
 \cite[Sect 2]{chamberII}, \cite[Sect 3]{ranktwo}.  
For clarity the proof will be structured in several steps, and a 
brief review of motivic Hall algebras will be provided in the process.

\subsection{Critical stability parameters}\label{criticalvalues}
In the framework of Section \ref{twotwo} fix a polarization 
$\omega$ of $X$ such that $\int_{C_0}\omega=1$. 
Recall that $b_c\in \IR$ is called critical of type $(r,n)\in 
\IZ_{\geq 0}\times \IZ$ if there 
exist strictly $\mu_{(\omega,B)}$-semistable $C$-framed coherent sheaves $E$ with numerical invariants 
$\ch(E)=(-1,0,[C]+r[C_0],n)$. 

First note the following consequence of 
 the defining conditions $(C.1)$, $(C.2)$ for the subcategory  
 $\CA^C\subset \CA$
in Section \ref{twotwo} 
\begin{lemm}\label{subquot} 
Let $E$ be a $C$-framed perverse coherent sheaf with $\ch(E)=(-1,0,[C]+r[C_0],n)$, $r\in \IZ_{>0}$. 
Let $F\subset E$ and $E\twoheadrightarrow G$ be a nontrivial subobject, respectively quotient of $E$ in $\CA^C$, where 
$F,G$ are pure dimension one sheaves. Then $F,G$ are topologically supported on $C_0$ and $\ch_2(F)=r_F[C_0]$, 
$\ch_2(G)=r_G[C_0]$ for some integers $0< r_F, r_G \leq r$. 
\end{lemm} 

{\it Proof.} It suffices to prove one case, since the other is analogous. Suppose $E\twoheadrightarrow G$ is a pure 
dimension
one quotient in $\CA^C$ and let $E'={\rm Ker}(
E\twoheadrightarrow G)$. Since  
 $\ch(E)=(-1,0,[C]+r[C_0])$ and $G$ is pure dimension one, 
$\ch_0(E')=-1$, $\ch_1(E')=0$. The second defining condition 
$(C.2)$ of $\CA^C$  in Section \ref{twotwo} 
implies that $G$ must be topologically supported on $C_0$. 
Therefore $\ch_2(G)=r_G[C_0]$, 
$r_G\in \IZ_{>0}$ and $\ch_0(E')=(-1,0,[C]+(r-r_G)[C_0], 
n')$. Moreover the first defining condition $(C.1)$ implies 
that $\CH^{-1}(E')$ must be the ideal sheaf of 
a closed subscheme $Z_{E'}\subset X$, 
which according to $(C.2)$ must be topologically supported on the union
$C\cup C_0$. Since $(C.2)$ also requires $\CH^0(E)$ 
to be topologically  supported on $C\cup C_0$, 
it follows that $r-r_G\geq 0$. 

\hfill $\Box$

\begin{lemm}\label{HNlemmA} 
Let $b_c\in \IR$ be a critical stability parameter of type $(r,n)$
and $E$ a strictly 
$\mu_{(\omega,b_c)}$-semistable object of $\CA^C$ 
with $\ch(E)=(-1,0,[C]+r[C_0],n)$. 
Then one of the following two cases holds.

$(i)$ There is an exact sequence 
\be\label{eq:HNseqQI} 
0\to E'\to E \to G \to 0 
\ee
in $\CA^C_{1/2}$, where  $G$ is an $\omega$-slope 
semistable pure dimension one sheaf 
set theoretically supported on $C_0$ with 
$\mu_\omega(G) = -2b_c$. Moreover, $\ch_2(G) = r_G[C_0]$ with $0< r_G\leq r$.

$(ii)$ There is an exact sequence 
\be\label{eq:HNseqQII} 
0\to F\to E \to E'' \to 0 
\ee
in $\CA^C_{1/2}$, where  $F$ is an $\omega$-slope 
semistable pure dimension one sheaf 
set theoretically supported on $C_0$ with 
$\mu_\omega(F) = -2b_c$. Moreover, 
$\ch_2(F) =r_F[C_0]$ with $0<r_F\leq r$.
\end{lemm}

{\it Proof.} This lemma follows 
from the stability criterion \ref{limitstabcrtA} 
applied to $C$-framed perverse coherent sheaves. 
By definition, if $E$ is strictly $\mu_{(\omega,b_c)}$-semistable,
one of the following two cases must hold. 

$(a)$ There is 
a strict 
epimorphism $E \twoheadrightarrow G$ in $\CA^C_{1/2}$ 
with $G$ a nontrivial pure dimension one sheaf on $X$ such that 
\[
\mu_{(\omega,b_c)}(G) = -3b_c.
\]

$(b)$ There is an $\mu_{(\omega,B)}$-semistable $C$-framed 
perverse coherent sheaf $E$ of type $(r,n)$ and a strict 
monomorphism $F\hookrightarrow $ in $\CA^C_{1/2}$ 
with $F$ a pure dimension one sheaf on $X$ such that 
\[
\mu_{(\omega,b_c)}(F) = -3b_c.
\]

Suppose $(a)$ holds. 
According to Lemma \ref{subquot}, $G$ must be topologically supported on $C_0$ and $\ch_2(G)=r_G[C_0]$ with 
$0<r_G \leq r$. 
Suppose $G\twoheadrightarrow 
G'$ is a nontrivial pure dimension one quotient and let
$K\subset E$ be the kernel of the resulting surjective 
morphism $E\twoheadrightarrow G''$ in $\CA^C$. Then $K$ 
must belong to $\CA^C_{1/2}$ since $E$ does, hence 
$E\twoheadrightarrow G''$ is a strict epimorphism. 
If $\mu_{(\omega,b_c)}(G'') < -3b_c$ this quotient destabilizes $E$, leading to 
a contradiction. Therefore $\mu_{(\omega,b_c)}(G'')\geq -3b_c = \mu_{(\omega,b_c)}(G)$, which proves that $G$ is 
$\omega$-slope semistable. This leads to case $(i)$ in Lemma 
\ref{HNlemmA}. 

Case $(b)$ leads analogously to case $(ii)$. 

\hfill $\Box$  

\begin{coro}\label{bplusminus} 
Under the conditions of Lemma \ref{HNlemmA}, there exist 
$b_-,b_+\in \IR$ with $b_-<b_c<b_+$ such that $b_c$ is the 
only critical stability parameter of type $(r,n)$ in the 
interval $[b_-,\ b_+]$. 
\end{coro} 

{\it Proof.} Choose some $b_-<b_c<b_+$. Lemma 
\ref{HNlemmA} implies that any critical 
stability parameter $b'_c$ must be of the form 
\[
b'_c = -{n'\over 2r'}
\] 
with $n',r'\in \IZ$, $1\leq r'\leq r$. Therefore the set of stability 
parameters in the interval $[b_-,\ b_+]$ is a subset of the set 
of integers $n'$ satisfying 
\[ 
-2r|b_+|\leq n' \leq 2r|b_-|.
\]
The latter is a finite set for fixed $b_-,b_+, r$. Therefore there exist 
$b'_-<b_c<b'_+$ sufficiently close to $b_c$ such that 
there are no critical stability parameters of type $(r,n)$ in the interval 
$[b'_-,\ b'_+]$. 

\hfill $\Box$

\begin{lemm}\label{HNlemmB} 
Suppose $b_c$ is a critical stability parameter of type $(r,n)\in \IZ_{\geq 1}\times \IZ$. Then there exist two constants $\epsilon_+,\epsilon_-$, 
such that the following holds for any stability parameters 
\[
b_c-\epsilon_-< b_-<b_c<b_+< b_c+\epsilon_+
\] 
Suppose $E$ is a $\mu_{(\omega, b_c)}$-semistable 
$C$-framed perverse coherent sheaf with $\ch(E)=(-1,0,
[C]+r[C_0],n)$. Then 

 $(i)$ $E$ is either $\mu_{(\omega, b_+)}$-semistable 
 or has a Harder-Narasimhan filtration 
\[ 
0\subset E'\subset E 
\]
with respect to $\mu_{(\omega, b_+)}$-stability, where 
$E'$ is an $\omega$-slope semistable pure dimension one 
sheaf with topological support on $C_0$ and $\mu_\omega(E') 
=-2b_c$. The quotient  
$E''=E/E'$ is an $\mu_{(\omega, b_+)}$-semistable $C$-framed 
perverse coherent sheaf.

$(ii)$ $E$ is either $\mu_{(\omega, b_-)}$-semistable 
or has a Harder-Narasimhan filtration 
\[ 
0\subset E'\subset E 
\]
with respect to $\mu_{(\omega, b_-)}$-stability, where 
$E'$ is a $\mu_{(\omega, b_-)}$-semistable $C$-framed 
perverse coherent sheaf. The quotient $E''=E'/E$ is 
an $\omega$-slope semistable pure dimension one 
sheaf with topological support on $C_0$ and $\mu_\omega(E'')=
-2b_c$.

\end{lemm}

{\it Proof.} 
It suffices to prove $(i)$, the second statement being entirely analogous. 
The existence of a Harder-Narasimhan filtration 
follows from the fact that $\mu_{(\omega,B)}$-stability 
is a weak stability condition \cite[Lemma 3.6]{generating}. 
Moreover, by construction all successive quotients of the 
Harder-Narasimhan filtration of an object of $\CA^C_{1/2}$ 
also belong to $\CA^C_{1/2}$. 

If $E$ is $\mu_{(\omega, b_+)}$-stable 
there is nothing to prove, hence it will be assumed this is not the case. 
Then let 
\be\label{eq:bplusHN}
0=E_0\subset E_1\subset \cdots \subset E_h=E, \qquad h\geq 2,
\ee
be its Harder-Narasimhan filtration with respect 
to $\mu_{(\omega,b_+)}$-stability.
Let 
$E\twoheadrightarrow E''$ 
be the last quotient  of the 
Harder-Narasimhan filtration. Hence $E''$ is a
$\mu_{(\omega,b_+)}$-semistable 
object of $\CA^C_{1/2}$. 
Since $\ch_0(E)=-1$, one has 
$\ch_0(E'') \in \{0,1\}$. 

Suppose $\ch_0(E'')=0$. Then 
$E''$ must be an $\omega$-slope semistable pure dimension one sheaf with topological support 
on $C_0$.
Moreover the kernel $E'= {\rm Ker}(E\twoheadrightarrow E'')$ 
is an object of $\CA_C^{1/2}$ because it admits a filtration
such that all successive quotients belong to $\CA^C_{1/2}$. 
Therefore the morphism 
$E\twoheadrightarrow E''$ is a strict epimorphism. 
 By the properties of the Harder-Narasimhan filtration, 
\[ 
\mu_{(\omega, b_+)}(E'') < -3b_+,
\]
which implies 
\[
\mu_{(\omega, b_c)}(E'') < -3b_c,
\]
since for any $b\in \IR$ 
\[
\mu_{(\omega,b)}(E'')= \mu_\omega(E'') - b. 
\]
According to Proposition \ref{limitstabcrtA}, this leads to a contradiction 
since $E$ is assumed $\mu_{(\omega, b_c)}$-semistable. 

In conclusion 
$\ch_0(E'')=-1$. This implies that all other successive 
quotients, $E_j/E_{j-1}$, $1\leq j\leq h-1$, are $\omega$-slope 
semistable pure dimension one sheaves with topological support on $C_0$. 

Next note that by construction 
the induced filtration 
\[
0\subset E_{j}/E_{j-1} \subset \cdots \subset  E/E_{j-1}
\]
on each quotient $E/E_j$, $j=1,\ldots, h-1$ is again a 
Harder-Narasimhan filtration for $\mu_{(\omega,b_+)}$-stability.
Therefore 
\[
\mu_{(\omega, b_+)}(E_{j+1}/E_j)> - 3b_+
\]
for all $1\leq j\leq h-1$. At the same time 
\[
\mu_{(\omega,b_c)}(E_1) \leq  -3b_c 
\]
since $E$ is $(\omega,b_c)$-semistable. 
Then, using the standard properties of Harder-Narasimhan 
filtrations, one obtains 
\be\label{eq:slopeineqA}
-2b_+ < \mu_\omega(E_{h-1}/E_{h-2}) < \cdots < 
\mu_\omega(E_1) \leq -2b_c.
\ee
However, since all $E_{j}/E_{j-1}$, $1\leq j\leq h-1$ are nontrivial 
pure dimension sheaves with topological support on $C_0$, 
\[ 
\ch_2(E_{j}/E_{j-1}) = r_j[C_0]
\]
for some integers $0< r_j \leq r$, and  
\[
\mu_\omega(E_{j}/E_{j-1})= {\chi(E_j/E_{j-1})\over r_j}.
\]
Now note that there exists $\epsilon_+>0$ sufficiently small 
such that there are no rational numbers $n'/r'$, $1\leq r'\leq r$, 
$n'\in \IZ$ in the interval $(-2b_c-2\epsilon_+, -2b_c)$. 
Therefore, if $b_c< b_+< b_c+\epsilon_+$ inequalities 
\eqref{eq:slopeineqA} imply that $h=2$ and $\mu_\omega(E_1)=
-2b_c$.  

\hfill $\Box$

\subsection{Motivic Hall algebra identities}\label{motivicHallidentities}
For completeness, recall the construction of 
the motivic Hall algebra \cite{wallcrossing,genDTI} of the perverse 
coherent sheaf category $\CA$. Let ${\mathcal Ob}(\CA)$ be the stack of 
all objects of $\CA$, which is algebraic, locally of finite type over $\IC$. 
A stack function is a pair $(\CX,f)$ where $\CX$ is an algebraic stack 
of finite type over $\IC$, and $f:\CX\to {\mathcal Ob}(\CA)$ a morphism 
of algebraic stacks. The underlying vector 
space of the motivic Hall algebra $H(\CA)$ is the $\IQ$-vector 
space generated by isomorphism classes of stack functions subject 
to the relations 
\[ 
[(\CX,f)] \sim [(\CY, f|_{\CY})] + [(\CX\setminus \CY, f|_{\CX\setminus \CY})]
\]
for any closed algebraic substack $\CY\subset \CX$. The algebra structure is defined by a $\IQ$-linear  convolution product 
\[
[(\CX_1,f_1)]\ast [(\CX_2,f_2)] = [(\CX,f)]
\] 
where $(\CX,f)$ is the stack function determined by a diagram
of the form 

\[
\xymatrix{ 
\CX \ar[d] \ar[rr]^{g} \ar@/^2pc/[rrr]^-{f=p_2\circ g} & & {\mathcal{E}x}(\CA)
\ar[d]_-{p_1\times p_3} 
\ar[r]^-{p_2}& 
\obj(\CC)\\
{\mathcal X}_1\times 
{\mathcal X}_2\ar[rr]^-{f_1\times f_2} && 
{\mathcal Ob}(\CA)\times {\mathcal Ob}(\CA)& \\}
\]
Here ${\mathcal E}x(\CA)$ denotes the moduli stack of three term exact 
sequences $0\to E_1\to E_2\to E_3\to 0$ in $\CA$ and 
$p_i:{\mathcal E}x(\CA)\to {\mathcal O}b(\CA)$ the three natural 
forgetful morphisms mapping such a sequence to $E_i$, $1\leq i\leq 3$, 
respectively. The convolution product is associative and has a unit, 
the stack function
$[({\rm Spec}(\IC) \to {\mathcal O}b(\CA)]$ determined by 
the zero object of $\CA$. 
Finally, note that the natural commutator $[\ ,\ ]$ of the 
associative product $\ast$ determines a 
Poisson algebra structure on $H(\CA)$. 

In the present case, any pair $(r,n)\in \IZ_{\geq 0}\times 
\IZ$ and any $b\in \IR$, the moduli stack of 
$\mu_{(\omega,b)}$-semistable $C$-framed perverse coherent sheaves 
determines an element 
\[ 
{p}_{(\omega, b)}(r,n) = 
\big[ {\mathcal P}_{(\omega,b)}(X,C,r,n) \hookrightarrow 
Ob(\CA)\big]
\]
of the Hall algebra $H(\CA)$. Similarly the moduli 
stack of $\mu_{(\omega,b)}$-stable coherent pure 
dimension one sheaves $F$ 
on $X$ with $\ch_2(F)=r[C_0]$ and $\chi(F)=n$ 
determines an element 
\[ 
{s}_{(\omega,b)}(r,n) = 
\big[ {\mathcal M}_{(\omega,b)}(X,r,n) \hookrightarrow 
Ob(\CA)\big]
\]
of the same Hall algebra. Since the polarization $\omega$ is fixed 
throughout this section, while $b$ is varied, 
the simpler notation $p_b(r,n)$ will be used in the 
following. Similarly, for fixed polarization, $s_{(\omega,b)}(r,n)$ is in fact 
independent on $b$, hence it will be denoted by $s(r,n)$. 

Wallcrossing formulas \cite{wallcrossing, genDTI} are obtained
by converting Lemma \ref{HNlemmB} into stack function identities, 
then applying a suitable integration map.
In the construction of \cite{genDTI}, the integration map is defined 
on a certain Poisson subalgebra of 
$H^{ind}_{alg}(\CA)\subset 
H(\CA)$ which has a complicated technical definition 
\cite[Sect. 5.2]{J-II}. Omitting the technical details,  
it suffices to note that the stack function determined by any 
moduli space of $(\omega,b)$-limit slope semistable 
stable objects of $\CA^C$ belongs 
to $H^{ind}_{alg}(\CA)\subset 
H(\CA)$ as long as there are no strictly semistable objects. 
In particular this is the case with the stack functions 
$p_b(r,n)$ for 
$b\in \IR$ non-critical of type $(r,n)$. 
This fails when strictly semistable objects are present, as is the case 
with the stack functions $s(r,n)$. In such cases it is proven in 
\cite[Thm. 8.7]{J-III} that the associated log stack functions 
\be\label{eq:logstackfcts}
{t}(r,n) = -\sum_{l\geq 1} {(-1)^l\over l} 
\sum_{\substack{(r_i,n_i)\in \IZ^2,\ r_i\geq 1,\ 1\leq i\leq l\\
r_1+\cdots+r_l=r, \ n_1+\cdots +n_l =n,\\
n_i/r_i = n/r, \ 1\leq i\leq l}}
{s}(r_1,n_1) \ast \cdots\ast {s}(r_l,n_l)
\ee
belong to $H^{ ind}_{alg}(\CA)$.
The sum in the right hand side is finite for fixed $(r,n)$
since there is a finite set of 
decompositions $r=r_1+\cdots+r_l$ with $r_i\geq 1$, $1\leq i\leq l$.

The integration map 
\[
I_\nu :H^{ind}_{alg}(\CA) \to \Lambda_\nu(\CA)
\] 
is a morphism of Poisson algebras 
determined by a choice of constructible 
function $\nu$ on the stack of all objects $Ob(\CA)$. 
It takes values in a Poisson algebra $\Lambda_\nu(\CA)$ 
spanned over $\IQ$ by  
$\{{\sf e}_{\alpha}\}$, $\alpha\in K(\CA)$, 
where $K(\CA)$ is the quotient of the Grothendieck group of $\CA$ by 
numerical equivalence, and $\chi(\alpha, \alpha')$. 
The Poisson bracket 
is given by 
\[
[{\sf e}_\alpha, {\sf e}_{\alpha'}] = (-1)^{\epsilon(\nu)
\chi(\alpha,\alpha')} 
\chi(\alpha, \alpha').
\]
where $\chi(\alpha, \alpha')$ is the 
natural antisymmetric bilinear pairing on $K(\CA)$, and 
$\epsilon(\nu)\in \{0,1\}$. In principle $\nu$ can be either 
the constant function $\nu=1$, in which case 
$\epsilon(\nu)=0$, or Behrend's  constructible function,
in which case $\epsilon(\nu)=1$. 
In the present context only the integration map with respect to the constant constructible function $\nu=1$ is rigorously 
constructed \cite[Thm. 6.11]{J-IV}. 
This yields topological Euler character invariants of objects in $\CA^C$ 
defined by 
\[
I(p_{b}(r,n))= -P^{top}_b(r,n) {\sf e}_{(-1,0,[C]+r[C_0],n)},
\qquad 
I(t(r,n)) = -N^{top}(r,n) {\sf e}_{(0,0,r[C_0],n)}.
\]

Employing the formalism reviewed above,
 Lemma \ref{HNlemmB} translates into 
the following Hall algebra identities 
\be\label{eq:hallidA} 
\bal 
{p}_{b_c}(r,n) - 
{p}_{b_+}(r,n) 
& = \sum_{\substack{r_1,r_2, n_1,n_2\in \IZ, \ r_1\geq 1, \
r_2\geq 0 \\ r_1+r_2=r,\ n_1+n_2=n
\\ n_1/r_1 = -2b_c}} { s}(r_1,n_1) \ast
{ p}_{b_+}(r_2,n_2)
\\
{p}_{b_c}(r,n) - 
{p}_{ b_-}(r,n) 
& = \sum_{\substack{r_1,r_2, n_1,n_2\in \IZ, \ r_1\geq 0, \
r_2\geq 1 \\ r_1+r_2=r,\ n_1+n_2=n
\\ n_2/r_2 = -2b_c}} 
{p}_{b_-}(r_1,n_1)\ast {s}(r_2,n_2).
\\
\eal 
\ee 
for $b_-<b_c<b_+$ sufficiently close to $b_c$. 
The sum in the right hand side is finite for fixed $(r,n)$
since there is a finite set of 
decompositions $r=r_1+\cdots+r_l$ with $r_i\geq 1$, $1\leq i\leq l$. 
Then repeating the steps in \cite[Lemmas 2.1-2.4 ]{chamberII}, identities 
\eqref{eq:hallidA} imply 
\be\label{eq:hallidB} 
\bal 
& {p}_{b_-}(r,n) - {p}_{b_+}(r,n) 
=\\
&\sum_{l\geq 2} {(-1)^{l-1}\over (l-1)!} 
\sum_{\substack{ (r_i,n_i)\in \IZ^2, \ 1\leq i\leq l,\\
r_i \geq 1,\ 1\leq i\leq l-1, \ r_l\geq 0\\
r_1+\cdots+r_l=r, \ n_1+\cdots +n_l =n,\\
n_i/r_i = -2b_c, \ 1\leq i\leq l-1\\ }}
[{t}(r_1,n_1), \cdots,  [{t}(r_{l-1}, n_{l-1}), 
{p}_{b_+}(r_l,n_l)]\cdots].\\
\eal
\ee
Again, the sum in the right hand side of equation \eqref{eq:hallidB} is finite
because the set of 
decompositions $r=r_1+\cdots+r_l$ with $r_i\geq 1$, $1\leq i\leq l$ is finite. This identity is in fact a wallcrossing formula 
for stack functions. 
Applying the integration map to both sides of equation 
\eqref{eq:hallidB} 
yields the following wallcrossing formula for topological Euler character 
invariants
\be\label{eq:wallcrossingformA} 
\bal 
&P^{top}_{b_-}(r,n) - P^{top}_{b_+}(r,n) 
=\\
&\sum_{l\geq 2} {1\over (l-1)!} 
\sum_{\substack{ (r_i,n_i)\in \IZ^2, \ 1\leq i\leq l,\\
r_i \geq 1,\ 1\leq i\leq l-1, \ r_l\geq 0\\
r_1+\cdots+r_l=r, \ n_1+\cdots +n_l =n,\\
n_i/r_i = -2b_c, \ 1\leq i\leq l-1\\ }} P^{top}_{b_+}(r_l,n_l) 
\prod_{i=1}^{l-1} n_i N^{top}(r_i,n_i)
\\
\eal
\ee

\subsection{Summing over critical values}\label{summcrtval}
Recall that $b>0$ is called a small stability parameter of type 
$(r,n)$ if 
there are no critical parameters of type $(r,n)$ 
in the interval $[0,b)$. All moduli stacks 
$\calP_{(\omega,b)}(X,C,r,n)$  of $\mu_{(\omega,b)}$-semistable objects 
of $\CA^C$ with numerical invariants $(-1,0,[C]+r[C_0],n)$ 
for small $b$ are canonically isomorphic, and will be denoted 
by $\calP_{0+}(X,C,r,n)$. 
Combining the results of Proposition \ref{projhilbert}, 
Lemma \ref{dualstacks}, and Proposition \ref{geombij}, 
it follows that for any $(r,n)\in \IZ_{\geq 0}\times \IZ$, 
$\calP_{0+}(X,C,r,n)$ is geometrically bijective to an 
$\IC^\times$-gerbe 
over the relative {\it Quot} scheme 
$Q^{[l,r]}(C)$, where $l=n-\chi(\CO_C)$. 
Then the topological Euler character invariants in this 
chamber are simply given by 
\[ 
P^{top}_{0+}(r,n) = \chi(Q^{[l,r]}(C)).
\]
If $n<\chi(\CO_C)$ the stack $\calP_{0+}(X,C,r,n)$ is empty, and 
$P^{top}_{0+}(r,n)=0$. Therefore in this chamber the generating 
function of topological invariants is 
\be\label{eq:zeroplusgenfct} 
Z^{top}(X,C,T,u) = u^{\chi(\CO_C)} \sum_{l\geq 0} \sum_{r\geq 0} 
T^ru^l \chi(Q^{[l,r]}(C)).
\ee
At the same time, 
by analogy with 
\cite[Thm. 3.21]{generating}, there exists a constant $\mu_{r,n}
\in \IR$ depending only on $(r,n)$ such that for 
$b< -\mu_{r,n}/2$, the moduli 
stack of $(\omega, b)$-semistable $C$-framed perverse 
coherent sheaves ${\mathcal P}_{(\omega,b)}(X,C,r,n)$ is isomorphic 
to the moduli stack of $C$-framed stable pairs. 
Hence for $b<-\mu_{r,n}/2$, the invariants 
$P_b(r,n)$ are equal to the topological Euler character invariants 
of stable pairs, denoted by $P_{-\infty}^{top}(r,n)$. 
Recall that equation \eqref{eq:factformulaD} 
is a factorization formula of the form 
\be\label{eq:factformulaG} 
Z^{top}_{-\infty}(X,C,T,u) = 
Z^{top}(X,C_0,T,u) Z^{top}_{0+}(X,C,T,u).
\ee
where 
\[
Z^{top}_{-\infty}(X,C,T,u)=\sum_{r\geq 0} \sum_{n\in \IZ} 
T^r u^n P^{top}_{-\infty}(r,n)
\]
and 
\[
Z^{top}(X,C_0,T,u)  = \sum_{r\geq 0}\sum_{n\in \IZ} 
T^r u^n N^{top}(r,n).
\]
This formula will be proven by successive 
applications of the wallcrossing formula \eqref{eq:wallcrossingformA}.

First note that the set of critical parameters $-\mu_{r,n}/2 \leq b_c 
< 1/(2r)$ of type $(r,n)$ is finite since all such parameters must be of 
the form $b_c = -{1\over 2r'}$ with $1\leq r'\leq r$. 
For any $(r,n)\in \IZ_{\geq 0}\times \IZ$, let $P_{0-}^{top}(r,n)$ 
denote the value of $P^{top}_b(r,n)$ for any $b<0$ such that 
there are no critical parameters of type $(r,n)$ in the interval 
$[b,\ 0)$. Then note that
the wallcrossing formula \eqref{eq:wallcrossingformA} at $b_c=0$ 
yields $P^{top}_{0+}(r,n)= P^{top}_{0-}(r,n)$. Therefore it suffices 
to relate $P^{top}_{-\infty}(r,n)$ to $P_{0-}^{top}(r,n)$.

Let $b_{-\infty}< \min\{0,-\mu_{r,n}/2\}$ be an arbitrary stability parameter.
$\Delta(r,n;b_{-\infty})$ be the set of all  decompositions 
\[
r= r'+\sum_{i=1}^l \sum_{j=1}^{k_i} r_{i,j}, \qquad 
n= n'+\sum_{i=1}^l \sum_{j=1}^{k_i} n_{i,j},
\]
with $l\geq 1$, $k_i\geq 1$ for all $1\leq i\leq l$, $r'\geq 0$, 
$r_{i,j}\geq 1$,
satisfying 
\[
-b_{-\infty}>{n_{1,1}\over r_{1,1}} =\cdots = {n_{1,k_i}\over r_{i,k_i}} > {n_{2,1}\over r_{2,1}} =\cdots = {n_{2,k_i}\over r_{2,k_i}}> \cdots 
> {n_{l,1}\over r_{l,1}} = \cdots = {n_{l,k_l}\over r_{l,k_l}}>0
\] 
Note that this is a finite set for fixed $(r,n)$ and $b_\infty$. 
Then successive applications of equation \eqref{eq:wallcrossingformA}
yield
\[ 
\bal 
P^{top}_{-\infty}(r,n) - P^{top}_{0-}(r,n) 
=& \sum_{l\geq 1} \sum_{(r',r_{i,j},n', n_{i,j})\in \Delta(r,n; b_{-\infty})} 
P^{top}_{0-}(r',n')\\
& \prod_{i=1}^l \prod_{j=1}^{k_i} {1\over k_i!} 
n_{i,j}N^{top}(r_{i,j},n_{i,j}).
\eal
\]
By simple combinatorics, the above equation may be rewritten as 
\be\label{eq:wallcrossingformE}
\bal
& P^{top}_{-\infty}(r,n) - P^{top}_{0-}(r,n) =\\
& \sum_{l\geq 1} {1\over l!} 
\sum_{\substack{(r',n'),(r_i,n_i)\in \IZ^2,\ 1\leq i\leq l\\
r'\geq 0,\ r_i\geq 1,\ 1\leq i\leq l,\\
r'+r_1+\cdots+r_l=r,\ n'+n_1+\cdots +n_l=n,\\
0 < n_i/r_i< -b_{-\infty},\ 1\leq i\leq l \\ }} P^{top}_{0-}(r',n') 
\prod_{i=1}^l n_i N^{top}(r_i,n_i). \\
\eal 
\ee
This formula holds for any $b_{-\infty}< \min\{0,-\mu_{r,n}/2\}$. 
Moreover, the invariants $P^{top}_{0-}(r',n')$ are zero if $n'<\chi(\CO_C)$. 
Therefore, for $|b_{-\infty}|$ sufficiently large, the upper bound 
$n_i/r_i<-b_{-\infty}$ will be automatically satisfied. Hence 
equation \eqref{eq:wallcrossingformE} becomes 
\be\label{eq:wallcrossingformF}
\bal
& P^{top}_{-\infty}(r,n) - P^{top}_{0-}(r,n) =\\
& \sum_{l\geq 1} {1\over l!} 
\sum_{\substack{(r',n'),(r_i,n_i)\in \IZ^2,\ 1\leq i\leq l\\
r'\geq 0,\ r_i\geq 1,\ 1\leq i\leq l,\\
r'+r_1+\cdots+r_l=r,\ n'+n_1+\cdots +n_l=n,\\
n_i/r_i>0,\ 1\leq i\leq l \\ }} P_{0-}^{top}(r',n') 
\prod_{i=1}^l n_i N^{top}(r_i,n_i), \\
\eal 
\ee
where the sum in the right hand side is finite. 

The last step is to convert equation \eqref{eq:wallcrossingformF} 
into a relation between the generating functions.
Multiplying \eqref{eq:wallcrossingformF} by $T^{r}u^{n}$ 
 and summing over $r\geq 0$, $n\geq 1$ yields
 \[
 Z^{top}_{-\infty}(X,C,T,u)  = 
 {\rm exp}\bigg[\sum_{r>0}
 \sum_{n>0} (-1)^n n N^{top}(r,n) T^{r} u^{n}\bigg] 
 Z^{top}_{0-}(X,C,T,u).
\]
Now Lemma \ref{sstablesheaves} implies that the moduli stack
of $\omega$-slope semistable sheaves $F$ with topological support 
on $C_0$ and $\ch_2(F)=r[C_0]$, 
$\chi(F)=n$ is isomorphic to the moduli stack of semistable 
rank $r$ bundles $E$ on $\IP^1$ with $\chi(E)=n$. 
If $n$ is not a multiple of $r$ there are no such bundles. If $n=kr$, $k\in \
\IZ$, there is only one such bundle up to isomorphism, 
$\CO_{\IP^1}(k-1)^{\oplus r}$.
Therefore the moduli stack is empty unless $n=kr$, $k\in \IZ$, in which 
case it is isomorphic to 
the quotient stack 
$[{\rm Spec (\IC)}/GL(r,\IC)]$.
Then \cite[Ex. 6.2]{genDTI} shows that 
\[
N^{top}(r,n) = \left\{ \begin{array}{ll} {(-1)^{r-1}\over r^2}, & {\rm if}\ n \equiv 0 \ {\rm mod}\ r,\\
& \\
0, & {\rm otherwise}.
\end{array} \right. 
\]
By direct substitution, 
\[ 
\bal 
\sum_{r>0}\sum_{n>0}  nN^{top}(r,n) T^{r} u^{n} & = 
\sum_{k\geq 1} k \sum_{r\geq 1} {(-1)^{r-1}\over r} 
(Tu^{k})^r\\
& = \sum_{k\geq 1} {\rm ln}\big(1+Tu^{k})\big)\\
& = {\rm ln} \, \prod_{k\geq 1} \big(1+Tu^{k})\big)^k.\\
\eal 
\] 
Hence 
\[
{\rm exp}\bigg[\sum_{r>0}\sum_{n>0} (-1)^n nN^{top}(r,n) T^{r} u^{n}
\bigg] = \prod_{k\geq 1} \big(1+Tu^{k})\big)^k.
\]
In order to conclude the proof, it remains to show that 
\[ 
Z^{top}_0(X,T,u) = \prod_{k\geq 1}\big(1+Tu^k)\big)^k.
\]
Since the formal neighborhood of $C_0$ in $X$ is isomorphic 
to the formal neighborhood of the zero section in the total 
space $Y$ of $\CO_{\IP^1}(-1)^{\oplus 2}$, it suffices to prove the 
corresponding result for stable pairs on $Y$. 
This follows from \cite[Thm. 3.15]{NH}, 
which proves analogous formulas for counting invariants defined 
by integration with respect to Behrend's constructible function. 
For concreteness note that the variables $q_0,q_1$ used in 
\cite[Thm. 3.15]{NH} are related to $T,u$ by 
\[
q_1^{-1} = T, \qquad q_0q_1= u.
\]
Moreover,  equation (3.4) in \cite[Thm. 3.15]{NH}
and the last formula in \cite[Sect. 3.2]{NH} yield 
\[
Z^{top}(Y,T,u) = \prod_{k\geq 1}\big(1+Tu^{k})\big)^k
\]
as claimed above. 

\bibliography{newrefB.bib}
 \bibliographystyle{abbrv}
\end{document}